\documentclass[a4paper, 11pt]{amsart}
\usepackage{anysize}
\marginsize{3,5cm}{2,5cm}{2,5cm}{2,5cm}
\usepackage{amssymb}
\usepackage{amsmath,amsbsy,amssymb,amscd}
\usepackage{t1enc}
\usepackage[cp1250]{inputenc}
\usepackage[british]{babel}
\usepackage[all]{xy}
\usepackage{color}
\usepackage{amsfonts}
\usepackage{latexsym}
\usepackage{amsthm}
\usepackage{mathrsfs}
\usepackage{hyperref}
\usepackage{graphicx}
\usepackage{comment}

\usepackage{color}
\usepackage{caption}
\usepackage{subcaption}
\usepackage{enumerate}
\usepackage{xcolor}
\usepackage[draft]{todonotes}

\definecolor{orange}{rgb}{1,0.5,0}

\newtheorem{definition}{Definition}[section]
\newtheorem{theorem}[definition]{Theorem}

\newtheorem{proposition}[definition]{Proposition}
\newtheorem{corollary}[definition]{Corollary}
\newtheorem{lemma}[definition]{Lemma}

\newtheorem{remark}[definition]{Remark}

\def\cA{\mathcal{A}}
\def\Re{\mathrm{Re\,}}

\def\C{\mathbb{C}}

\def\geq{\geqslant}
\def\leq{\leqslant}
\def\R{\mathbb{R}}

\def\Z{\mathbb{Z}}
\def\N{\mathbb{N}}

\def\Q{\mathbb{Q}}

\def\epsilon{\varepsilon}

\def\inj{\operatorname{inj}}

\newcommand{\vep}{\varepsilon}
\newcommand{\beq}{\begin{equation}}
\newcommand{\eeq}{\end{equation}}
\newcommand{\bea}{\begin{eqnarray}}
  \newcommand{\eea}{\end{eqnarray}}
  \newcommand{\beab}{\begin{eqnarray*}}
  \newcommand{\eeab}{\end{eqnarray*}}

  \newcommand{\be}{\begin{equation}}
  \newcommand{\ee}{\end{equation}}

\title{Horocycle flows at product of two primes}
\author{Giovanni Forni}
\address{Department of Mathematics, University of Maryland, College Park, MD USA and Laboratoire AGM, CY Cergy Paris Universit\'e, France}
\email{gforni@umd.edu}
\author{Adam Kanigowski}
\address{Department of Mathematics, University of Maryland, College Park, MD USA and Faculty of Mathematics and Computer Science, Jagiellonian University, Lojasiewicza 6, Krakow, Poland }
\email{akanigow@umd.edu}
\author{Maksym Radziwi\l\l}
\address{Department of Mathematics, Northwestern University, 2033 Sheridan Rd, Evanston, IL 60208, USA}
\email{maksym.radziwill@northwestern.edu}
\date{}

\begin{document}
\maketitle
\begin{abstract}We show that if $\Gamma$ is a co-compact arithmetic lattice in $SL(2,\R)$ or $\Gamma=SL(2,\Z)$ then the horocycle orbit of every non-periodic point $x\in SL(2,\R)/\Gamma$ equidistributes (with respect to Haar measure) when sampled at integers having exactly two prime factors.
\end{abstract}

\tableofcontents

\section{Introduction}
Investigations around  a prime number theorem (PNT for short) in ergodic theory started with Bourgain's classical result \cite{Bo1} (see also \cite{Wi}) giving rise to the almost everywhere (a.e) convergence of ergodic averages along prime times: given a measure-theoretic dynamical system $(X,\mathcal{B},\mu,T)$, for each $f\in L^r(X,\mu)$, $r>1$,
\beq\label{prnuth1}
\lim_{N\to\infty}\frac1{\pi(N)}\sum_{p\leq N}f(T^px)\text{ exists}
\eeq for a.e.\ $x\in X$;
here, and in what follows, $p$ stands for a prime number, and $\pi(N)$ denotes the number of primes in $[1,N]$. Let us emphasize that the methods of \cite{Bo1} allow to prove a.e. convergence  for other sparse subsets of the integers such as polynomial sequences. This result generalizes the celebrated Birkhoff's ergodic theorem to some natural density zero subsets of the natural numbers. Given a topological dynamical system, i.e. a homeomorphism $T:X\to X$ of a metric space $X$, it is natural to ask about the behavior of the orbit of {\em every} fixed initial condition $x\in X$. For the everywhere convergence of regular averages (along natural numbers) a convenient condition for such convergence is unique ergodicity, i.e. existence of exactly one invariant (ergodic) measure. In this case the limit is always given by the integral of the function with respect to the only invariant measure. Let us point out that convergence of regular ergodic averages might still hold for systems which are not uniquely ergodic. The most classical example of such phenomenon is the horocycle flow on non-compact quotients of $SL(2,\R)$ in which case it follows by a result of Dani \cite{Dani} that the orbit of  {\em every} point is either periodic or equidistributed with respect to the  Haar measure. It is therefore natural to ask whether there are topological conditions that would guarantee convergence of averages for every point when sampled at some (density zero) subsets of the integers (such as primes or polynomial sequences).  In 2016, P.\ Sarnak \cite{Sa} proposed, as a general program, to characterize those topological dynamical systems in which a PNT holds, i.e.\ \eqref{prnuth1} holds for all continuous functions and all points. Sarnak (see also e.g.\ Tao \cite{Tao}) viewed PNT in dynamics as a natural and more difficult step  in the hierarchy of problems following the celebrated M\"obius orthogonality conjecture \cite{Sa}. Recall that the M\"obius orthogonality conjecture from 2010 predicts that for every topological system $(X,T)$ of zero (topological) entropy, we have
$$
\frac{1}{N}\sum_{n\leq N}\mu(n)f(T^nx)\to 0,
$$
for every $x\in X$ and $f\in C(X)$ and where $\mu$ denotes the M\"obius function. The above conjecture has been proved for many classes of dynamical systems. From the ergodic perspective the simplest tool in studying M\"obius orthogonality is the so called DKBSZ criterion which reduces the problem to studying {\em joinings} of the dynamical systems $T^p$ and $T^q$, where $p,q$ are different primes. The progress on M\"obius orthogonality did not translate into a progress on  PNT in dynamics. In fact there are uniquely ergodic systems for which the DKBSZ criterion can be applied, i.e. $T^p$ and $T^q$ are disjoint for any $p,q$ but a prime number theorem still fails, \cite{Ka-Le-Ra}. The main difficulty is that to understand prime orbits one needs {\em quantitative} information on joinings of $T^p, T^q$ (and so methods from classical ergodic theory do not apply). Let us now make this more precise. The difficulty has already been observed by Sarnak and Ubis \cite{SU} and is based on the approach by Duke, Friedlander and Iwaniec
\cite{Du-Fr-Iw}.  More precisely, following \cite{Du-Fr-Iw}, the main method to obtain a PNT, one studies the asymptotics (depending on $\vep>0$) of the  sums:
$$
\sum_{n\leq N/d} f(T^{dn}x) \;\; \text{type I sums (linear sums)}
$$
and
$$
\sum_{n\leq \min(N/d_1,N/d_2)}f(T^{nd_1}x)\overline{f(T^{d_2n}x)}\;\;\text{ type II sums (bilinear sums)},$$
where $d,d_1,d_2$  are ``large'', meaning  $\leq N^{\alpha-\vep}$, where $\alpha>0$ (the level) is a ``large'' constant, and typically, $\alpha=\frac12$ for type I sums and $\alpha=\frac13$ for type II sums. Such results can not be proved using joinings or classical methods from ergodic theory (as one needs rates). For this reason a PNT has been so far proved for a very limited number of dynamical systems: by Vinogradov's theorem \cite{Vi}, a PNT holds for all rotations on the circle. All other known cases are rather recent: nilsystems \cite{Gr-Ta},
Rudin-Shapiro sequences \cite{Ma-Ri}, (some regular) Toeplitz systems \cite{Fr-Ka-Le},
enumeration systems \cite{Bo2}, \cite{Gr}, automatic sequences \cite{Mu}, some finite rank symbolic systems \cite{Bo3}, \cite{Fe-Ma} or analytic Anzai skew products \cite{Ka-Le-Ra}.

In this paper we will focus on sparse ergodic theorems for one of the most classical classes of dynamical systems, namely horocycle flows acting on finite-volume quotients of $SL(2,\R)$. We will now describe some classical and recent results on orbits of horocycle flows.

As mentioned above, results on the behavior of integer orbits where obtained by Dani, \cite{Dani}. In \cite{BSZ} Bourgain, Sarnak and Ziegler have shown M\"obius orthogonality for the horocycle flow using qualitative type II sums (DKBSZ criterion) and Ratner's joinings classification. Venkatesh, \cite{Venkatesh}, has shown that in the co-compact case, for sufficiently small delta,  the orbit $\{h_{n^{1+\delta}}x\}$ is equidistributed for every $x$. This proved a special case of a general conjecture of Margulis and Shah, which predicts that the horocycle orbit is equidistributed at polynomial or prime times.  Venkatesh's result was generalized in \cite{TV15}, and by a different method in \cite{FFT}, to get an exponent $\delta$ not dependent on the spectral gap. Streck, \cite{Streck}, generalized Venkatesh's result to the non-compact setting.
Sarnak and Ubis, \cite{SU},  using sharp bound on  type I sums only, proved that in the modular case, i.e.\ when $\Gamma=SL(2,\Z)$ the orbit $(h_tx)$ of the horocycle flow $(h_t)$ acting on $SL(2,\R)/\Gamma$ at prime times $p$ hits any open set of measure $>9/10$ (for all $x\in X$ having dense orbits). This was generalized to more general lattices by Streck, \cite{Stre}.
 Better results seemed to be unavailable because of the absence of an effective Ratner's theory on joinings (quantitative type II sums). The approach of studying quantitative type I sums has also been used by McAdam, \cite{McAdam}, who has shown that there exists $k\in \N$ such that the horocycle  orbit of every point are dense along numbers which have at most $k$-prime factors (for some constant $k\geq 10$). All the above results either used quantitative type I information or qualitative type II information and there was a good reason for that: the landmark Ratner's joinings theorems are non quantitative. This has changed with  a recent breakthrough result by Lindenstrauss, Mohammadi and Wang \cite{LMW} where they obtained quantitative bounds on type II sums for the horocycle flow acting on quotients of arithmetic lattices. We should point out that the quantitative information is of level $\alpha>0$ which is much smaller than $1/3$ and so one still can't get a PNT using \cite{LMW}. However, this is potentially enough to establish that a semi-prime number theorem holds, i.e. equidistribution of the sequence at numbers which have exactly two prime factors. Our main result is a semi-prime number theorem for the horocycle flow acting on co-compact arithmetic qoutients or on the quotient by $SL(2,\Z)$. More precisely:
\begin{theorem}\label{th:glowne} Let $\Gamma$ be a co-compact arithmetic lattice. Then the time-1 map $T=h_1$ acting on $X=SL(2,\R)/\Gamma$ satisfies a SPNT. More precisely, for any  point $x\in X$ and any $f\in C(X)$
$$
\lim_{N\to \infty} \frac{1}{\pi_2(N)}\sum_{p_1\cdot p_2\leq N}f(h_{p_1\cdot p_2}x)=\int_X f d \mu_X,
$$
where $\pi_2(N)$ denotes the number of semi-primes up to $N$.
\end{theorem}

Our second result deals with the case $\Gamma=SL(2,\Z)$. We need some notation to formulate it. For $x\in X=SL(2,\R)/SL(2,\Z)$ let $O_d(x)=\overline{\{h_nx\;:\; n\in \Z\}}$ and $O_c(x)=\overline{\{h_tx\;:\; t\in \R\}}$. It follows from \cite{Dani} that for any $x\in X$ either $O_c(x)=X$ or $x$ is periodic for $(h_t)$. Moreover, $O_d(x)=O_c(x)$ unless $x$ is periodic for $(h_t)$ and the orbit $\{h_nx\}$ is finite. Let $\mu_x$ denote the unique measure $\R$-generic for $x\in X$. We have:

\begin{theorem}\label{th:glowne'} Let $\Gamma=SL(2,\Z)$. Then the time-1 map $T=h_1$ acting on $X=SL(2,\R)/\Gamma$ satisfies an SPNT. More precisely,
 for any point $x\in X$ for which $\{h_nx\}$ is infinite and any $f\in C_c(X)$
$$
\lim_{N\to \infty} \frac{1}{\pi_2(N)}\sum_{p_1\cdot p_2\leq N}f(h_{p_1\cdot p_2}x)=\int_X f d \mu_x
$$
\end{theorem}
Note that if $x\in X$ is such that $\{h_nx\}$ is finite then the fact that $x$ satisfies a SPNT follows from the theorem on semi-primes in arithmetic progressions. We should point out that the proof in the co-compact case is more straightforward and the real difficulty comes with the modular case. We also think that the methods could be more generally applied whenever $\Gamma$ is a congruence lattice

\subsection{Outline of the proof and new methods}
In this section we will describe our methods for obtaining our main results. The first step is to prove a general criterion which guarantees that a bounded sequence $(b_n)$ is equidistributed along semi-primes. As shown in Proposition \ref{prop:typ2} it is enough to show that 
$$
|\sum_{n\leq N}b_{pn}\overline{b_{qn}}|\ll \frac{N}{\log^{100}N}.
$$
 for most primes $p,q\leq N^\epsilon$. It is important to emphasize that the $\epsilon>0$  could be a function that goes to zero with $N$, however it should be converging to zero slowly enough, as estimates on bilinear sums for most primes in the range  $[\exp(\log^{\varepsilon} N), \exp(\log^{1 - \varepsilon} N)]$ are needed. This will turn out to be important later on (in relation with the Siegel-Walfisz theorem). A first and immediate attack is to apply this criterion to the case $b_n=f(h_nx)$ where $(h_t)$ is the horocycle flow. Note that using the renormalization with the geodesic flow $(a_t)$, i.e. $a_th_s=h_{e^ts}a_t$, we need to show that
$$
|\sum_{n\leq N}(f\times f)(a_{\log p}\times a_{\log q})(h_n\times h_n)(a_{-\log p}x,a_{-\log q}x)\ll \frac{N}{\log^{100}N},
$$
 i.e. we need to understand the $(h_n\times h_n)$ orbit of the point $(a_{-\log p}x,a_{-\log q}x)$ for the function $(f\times f)(a_{\log p}\times a_{\log q})$. The result of \cite{LMW} together with Venkatesh method (to go from continuous time to discrete time) tells us that the above estimate will hold unless: (i)~the point $(a_{-\log p}x,a_{-\log q}x)$ is close to an $SL(2,\R)$-invariant subspace $H_{\cdot}(x_0,y_0)$ of volume $\leq N^{-\delta}$ (close to a periodic point) or (ii)~the point $(a_{-t}\times a_{-t})(h_r\times h_r)(a_{-\log p}x,a_{-\log q}x)$ has injectivity radius at most $N^{\delta A}e^{-t}$ for every $t\in [\log N^\delta, \log N]$ and $r\in [0,N]$. Note that alternative (ii) cannot hold in the co-compact case as the injectivity radius is uniformly bounded below. The rest of the analysis boils down to analyzing the cases (i) and (ii).  First in Proposition \ref{l:modular1} we show that if a point is close to an $SL(2,\R)$- periodic orbit  $H_{\cdot}(x_0,y_0)$ then the direction of the divergence happens in the direction of the centralizer. This is a modest generalization of the corresponding result in \cite{LMW} in which the authors claimed divergence along some element of $SL(2,\R)\times SL(2,\R)$. As a result we get (see Corollary \ref{cor:uni}) that if a point $(x,y)$ satisfies (i), then there is a point $(u,v)\in H_{\cdot}(x_0,y_0)$ and $K_i(t)\leq T^{2\delta}$ such that $d_{X\times X}((h_t x, h_t y),(h_{K_1(t)+t}u,h_{K_2(t)+t}v))<T^{-1+3\delta}$, i.e. the orbit of the point $(x,y)$ slides along the orbit $(u,v)$ in the direction of the centralizer. We will now discuss the co-compact and modular case separately.

Let us first discuss the co-compact case as it is significantly easier. In this case (ii) never holds since the injectivity radius is bounded below. Moreover in the co-compact case, we show that (i) never holds for the point   $(a_{-\log p}x,a_{-\log q}x)$. This is done by making the argument of \cite{BSZ} quantitative. More precisely in this case the lattice is commensurable with the integer unit group in the quaternion algebra (see Section \ref{sec:comp}). Moreover, the bound on the $vol(H_{\cdot}(x_0,y_0))$ gives us a bound on the size of the denominators of the corresponding element of the commensurator group in the representation (see Lemma \ref{lem:per'}). This  in particular implies that if the point $(a_{-\log p}x,a_{-\log q}x)$ is close to $H_{\cdot}(x_0,y_0)$, then trace of the matrix representing the element from the commensurator has to be equal (up to rescaling) to $\sqrt{p/q}+\sqrt{q/p}$. This however implies that there is an element $\alpha=(x_0,x_1,x_2,x_3)$ in the quaternion algebra and with rational entries for which the determinant $N(\alpha)$ equals $0$. This can only happen if $x_i=0$ for all $i$ which implies that $p=q$. This gives a contradiction with (i). As mentioned, (ii) never holds and so in the co-compact case we show that the SPNT criterion holds for any $p,q\leq N^\delta$.

Let us now move to the modular case which is much more interesting and involved. Before we give a more detailed description of what goes into the analysis in this case let us just discuss one particular case which shows why the analysis is involved. Let $x$ be a periodic point $\in SL(2,\R)/SL(2,\Z)$ of period $N^{\psi(N)}$, where the $\psi(N)$ goes to zero very slowly, for example $\psi(N) = \log^{-\delta} N$ for some small $\delta > 0$. In this case there is no hope of applying the type II sums criterion as in this case the point $(a_{-\log p}x,a_{-\log q}x)$ cannot be polynomially distributed in space (recall that the function $(f\times f)(a_{\log p}\times a_{\log q})$ has polynomially large Sobolev norm in $p,q$, both of which can be as large as $N^{\psi(N)}$ and so we need polynomial equidistribution). Let us additionally assume that the period of $x$ is an integer (maybe even a prime). In this case the only tool we have is semi-primes in arithmetic progressions, i.e. the Siegel-Walfisz theorem for semi-primes. This theorem holds unconditionally only  in the moduli range $\log^A N$  and so it can't be directly applied to orbits of size $N^{\psi(N)}$. What we show (see Proposition \ref{thm:Maks2}) is that the Siegel-Walfisz theorem for semi-primes in the relevant range holds with a multiplicative twist, i.e. a multiplicative character $\chi(\cdot)$ (see Proposition \ref{thm:Maks2}). Thanks to this, the problem for such $x$ (lying in a periodic integer orbit) is now reduced to studying
$$
\sum_{n\leq R}\chi(n)f(h_nx),
$$
 where $\chi(\cdot)$ is a multiplicative function and $R=N^{\psi(N)}$ is the period. We can then apply a quantitative variant of the DKBSZ-criterion (using \cite{BSZ}), to reduce the above problem to studying again bilinear sums
 $$
|\sum_{n\leq R}(f\times f)(a_{\log p'}\times a_{\log q'})(h_n\times h_n)(a_{-\log p'}x,a_{-\log q'}x)\ll \frac{R}{\log^{100}R},
$$
in the range $p',q'\leq R^\delta$. It is here where we again use the \cite{LMW} result, which implies that the above holds unless (i) or (ii) holds, but with different parameters ($R$ not $N$). A crucial argument that will be described more in detail below shows that if (i) or (ii) hold for 
$(a_{-\log p'}x,a_{-\log q'})$ then $x$ is close to a periodic orbit $\{h_tw\}$ with period $\leq R^\delta$. But the point $x$  is a periodic point of period $R$ and we show, using a result of Str\"ombergsson on equidistribution of pieces of closed horocycles, \cite{Strom04}, that it can not be close to a periodic point with a much shorter period. This shows what type of problems arise while working in the modular case.

Let us now pass to a more structured description of the general case. As already mentioned, the analysis boils down to cases (i) and (ii). The case (ii) is simpler as we just show  that if $(a_{-\log p}x,a_{-\log q})$
then the shift $h_{t_0}x$ of the point   $x$ is  itself is close to a periodic point (see Proposition \ref{prop:E3}). The reasoning here is a special case of what happens in case (i) which we will now describe.  If the point $(a_{-\log p}x,a_{-\log q})$ satisfies (i) then in Proposition \ref{prop:E2} we show the following crucial dichotomy: either it is still equidistributed so that we can apply the SPNT criterion or the point $x$ is close to a periodic orbit $w\in SL(2,\R)/SL(2,\Z)$ of period $\leq N^\delta$. Let us explain this: for simplicity assume that the point  $(a_{-\log p}x,a_{-\log q}x)$ actually lies on the  subvariety $H_{\cdot}(x_0,y_0)$ (there is an extra quite involved approximation argument if it is close to but not on it). In this case by Ratner's works on joinings, \cite{Ratner}, it follows that the dynamics of $(h_t\times h_t)$ is algebraically conjugated to the horocycle flow $(h_t)$ on $SL(2,\R)\slash \Gamma_{p,q}$, where $\Gamma_{p,q}\subset SL(2,\Z)$ is a lattice with index $\leq N^\delta$. If the lattice was fixed (not depending on $N$) then one could use results of Str\"ombergsson \cite{Strom13} or Flaminio-Forni, \cite{FlaFor} together with a more recent work of Streck \cite{Streck} to show that a point equidistributes polynomially unless it is close to a periodic orbit of period $\leq N^\delta$. In our case however the lattice depends on $N$ and we need uniform bounds on the ergodic integrals also in terms of the lattice. This is done largely in the appendix where the argument of Str\"ombergsson and Flaminio, Forni are made quantitative to also reflect properties of the lattice. In our case we show that $\Gamma_{p,q}$ is a congruence lattice and so we have a uniform bound on the spectral gap by Selberg, \cite{Selberg}, and also we know that the co-volume of $\Gamma_{p,q}$ is not to large. These two properties allow to generalize the now classical results for a fixed lattice $\Gamma$ to uniform bounds depending on the spectral gap and co-volume. We should also point out that we use a uniform version of Streck's result, \cite{Streck}, who showed that points that do not equidistribute need to be close to short periodic orbits.
 
 Having Propositions \ref{prop:E2} at hand, the dichotomy becomes the following: either the point $(a_{-\log p}x,a_{-\log q})$ satisfies the SPNT-criterion (i.e. the first alternative in \cite{LMW}) or it is close to a periodic orbit of period $\leq N^\delta$. In the case the period is an integer or more generally close to a rational with small denominator, then one needs to apply the generalized Siegel-Walfisz theorem as we already discussed above. It is also interesting to describe what happens if the period (or in fact its inverse) is far from rationals with small denominators (minor arc case). In this case we show that being close to a periodic point $w$ implies that $h_tx$  is close to $h_{m(t)}w$ where $m$ is a certain explicit function (see Lemma \ref{l:modular3}) with the important property that it can be approximated by polynomials on large subsets of $[0,N]$. Then the analysis boils down to the analysis of the orbit $\{m(p_1p_2)\alpha\}_{p_1\cdot p_2\leq N }$ on the circle, where $\alpha=\text{period}^{-1}$. This is done in Proposition \ref{thm:maks1} using the classical $A-B$ process in the theory of exponential sums, and Vinogradov's method in prime number theory.
 
  In particular we show that in the minor arc case the orbit becomes equidistributed in the closure of the periodic orbit. One final point is that in the analysis we always have upper bounds on the size of the approximating periodic orbit but what is crucial is that if a point $x\in X$ is generic for the Haar measure (i.e. not periodic) then the lengths of the periodic approximants need to grow to $\infty$ with $N$ as otherwise the point would be generic for a fixed periodic orbit. In the end we use a qualitative version of a result of Sarnak, \cite{Sarn}, which states that long periodic orbits equidistribute towards Haar measure. 

  \section{Acknowledgment}

GF acknowledges support of NSF grant DMS-2154208.
AK acknowledges support of NSF grant DMS-2247572.
MR acknowledges support of NSF grant DMS-2401106.

\section{A sufficient condition for a SPNT}\label{sec:conspnt}
\subsection{Type II sums}

We start with the following general criterion.
\begin{proposition} \label{prop:typ2}
Let  $\varepsilon > 0$ be given. For all $N \geq 2$,  let $\mathcal{P}_{\varepsilon, N}$ be the set of primes 
in the interval $[\exp(\log^{\varepsilon} N), \exp(\log^{1 - \varepsilon} N)]$ with $\exp(\log^{\varepsilon} N) > (\log N)^{1000}$ and
let $\mathcal{S}_{\varepsilon, N}$ be a subset of $\mathcal{P}_{\varepsilon, N}$ with
the property that in every dy-adic interval $[P, 2P]$ we discard at most $\ll P / (\log N)^{100}$
primes. Let $(a_n)$ be a sequence with $|a_n| \leq 1$. Suppose that, for $N\geq 2$,
\begin{equation}\label{eq:typ2}
\sum_{n \leq N} a_{n q_1} \overline{a_{n q_2}} \ll \frac{N}{(\log N)^{100}}
\end{equation}
for primes $q_1, q_2 \in \mathcal{S}_{\varepsilon, N}$ with $1/5\leq \frac{q_1}{q_2}\leq 5$ and $q_1\neq q_2$. Then,
\begin{equation}\label{eq:sp}
\sum_{\substack{p q \leq N}} a_{p q} \ll \varepsilon \cdot \sum_{p q \leq N} 1 + \frac{1}{\varepsilon^{53}} \cdot \frac{N}{(\log N)^{50}}\,.
\end{equation}
 Moreover, assume that for every $M\in [N^{9/10},N]$,  for $N$ large enough,
\begin{equation}\label{eq:typ2'}
\sum_{n \leq M} a_{n q_1} \overline{a_{n q_2}} \ll \frac{M}{(\log \log M)^{10}}
\end{equation} 
 holds for all $q_1\neq q_2$, with $q_1,q_2\in\Big[ e^{(\log\log \log N)^3}, e^{(\log\log N)^{10}}\Big]$ and $1/5\leq q_1/q_2\leq 5$. Then for every multiplicative function $\nu$ with $|\nu|\leq 1$,  for all $N\in N $ large enough,
  \begin{equation}\label{eq:mob'}
\Big|\sum_{n\leq N}\nu(n)a_n\Big|\ll N(\log\log N)^{-4}\,.
\end{equation}
\end{proposition}

\begin{proof} The second part, i.e. \eqref{eq:mob'}, can be deduced from the proof of  Theorem 2 in \cite{BSZ}. We will explain how it follows from this proof. In Theorem 2 we take $\tau=(\log \log N)^{-10}$. Note that by (1.4) this gives us \eqref{eq:mob'}. Note that in the statement of Theorem 2, the authors require that for  $p_1,p_2\leq e^{1/\tau}$ with $p_1\neq p_2$, (1.3) holds. But in fact they need less: first, see (2.17) the authors they apply (1.3) for $x_1,x_2\in P_j$ and $P_j=[(1+\alpha)^{j},(1+\alpha)^{j+1}]$, where $\alpha=\sqrt{\tau}$ and $j\in [j_0,j_1]$ where $j_0=\alpha^{-1}(\log (\alpha^{-1}))^3$ and $j_1=j_0^2$. Moreover the length of the sum is $\frac{N}{(1+\alpha)^j}$. In particular with this choice of parameters it follows that $\frac{N}{(1+\alpha)^j}\in [N^{9/10},N]$ and so our range for $M$ in \eqref{eq:typ2'} is sufficient. Second, since $P_j$ are $(1+\alpha)$-adic it immediately follows that $1/2\leq x_1/x_2\leq 2$. This shows that our bound $1/5\leq q_1/q_2\leq 5$ is sufficient ($x_1, x_2$  just is $q_1,q_2$ in our notation). Finally the range for $q_1$ and $q_2$. Since $x_1,x_2\in P_j$ it follows that $(1+\alpha)^{j_0}<x_1,x_2<(1+\alpha)^{j_1}$. So we only need to show that  $e^{(\log\log \log N)^3}\leq (1+\alpha)^{j_0}$ and that $e^{(\log\log N)^{10}}>(1+\alpha)^{j_1}$. The second inequality is given in (2.18). For the first one note that $(\log \log \log N)^3\leq (\log\log N)^5(\log \log \log N)^3 \alpha \leq j_0 \log(1+\alpha)$. This implies that the assumptions in \eqref{eq:typ2'} are enough for the proof of Theorem 2 in \cite{BSZ}.
 We will therefore focus on the first part of Proposition \ref{prop:typ2}.

We wish to bound
$$
\sum_{p,q} a_{p q} W \Big ( \frac{pq}{N} \Big )
$$
where $W$ is a smooth function compactly supported in $(0, 1)$
and equal to $1$ on $(\eta, 1 - \eta)$. This is sufficient as it 
introduces an error of at most $\ll \eta$ times the trivial bound. 

Let $K$ be a smooth function such that $K$ is compactly supported in $(1/2, 5)$ and
$$
\sum_{P} K \Big ( \frac{n}{P} \Big ) = 1
$$
for every integer $n \geq 1$ and $P$ running over powers of two. 
We introduce such a partition of unity on both the $p$ and $q$ variables. 
Thus we have to bound, 
$$
\sum_{P, Q} \Big ( \sum_{p, q} a_{p q} K \Big ( \frac{p}{P} \Big ) K \Big ( \frac{q}{Q} \Big) W \Big ( \frac{p q}{N} \Big ) \Big )
$$
We now make a number of observations about $P$ and $Q$. First, $\eta N / P < Q < 2 N / P$. Thus for each choice of $P$ there are at most $\ll \log (1 / \eta) \ll 1 / \eta$ choices of $Q$. Second, 
we can restrict $Q = 2^k$ to $k$ such that $\log^{\varepsilon }N  < k < \log^{1 - \varepsilon} N$. The
reason for this is that the contribution of $k$ outside of this interval is
$$
\ll \varepsilon \cdot \frac{N \log\log N}{\log N}. 
$$
Thus we have $P Q \asymp N$, and $Q = 2^k$ with $\log^{\varepsilon} N \leq k \leq \log^{1 - \varepsilon} N$. 

We open $W$ into a Mellin transform, 
$$
W(u) := \frac{1}{2\pi i} \int_{- i \infty}^{i \infty} \widetilde{W}(s) u^{-s} ds
$$

We notice that $\widetilde{W}(s)$ is of rapid decay at infinity, in fact integrating by parts twice since 
and $W$ has compact support in $(1/2, 5)$ we get
$$
 \widetilde{W}(s) =  \int_{0}^{\infty} W(x)x^{s- 1} dx=  - \frac{1}{s} \int_{0}^{\infty} W'(x)x^{s} dx = \frac{1}{s(s+1)} \int_{0}^{\infty} W''(x)x^{s+1} dx
$$
hence
$$
\vert \widetilde{W}(s)\vert  = \vert \int_{0}^{\infty} W(x)x^{s- 1} dx \vert  \ll \frac{\| W'' \|_{1}}{ \vert s(s+1)\vert } \ll  \frac{1}{\eta} \cdot \frac{1}{\vert s (s+1) \vert}.
$$
Thus it remains to bound, 
$$
\sum_{P, Q = 2^k} \int_{- \infty}^{\infty} |\widetilde{W}(it)| \cdot \Big | \sum_{p, q} a_{p q} (p q)^{-it} K \Big ( \frac{p}{P} \Big ) K \Big ( \frac{q}{Q} \Big ) \Big | dt
$$
where $\eta N / P \leq Q \leq 2 N / P$ and $\log^{\varepsilon} N \leq k \leq \log^{1 - \varepsilon} N$. The integral and sum over $P, Q$ incurs an additional error of $1/ \eta^2$. We therefore focus on bounding the inner sum. Since the sum is symmetric in $P,Q$ and  $P Q \asymp N$ we can WLOG assume that  $P\gg N^{1/2}$.

It is sufficient to establish for each admissible $P$ and $Q$ the bound, 
\begin{equation} \label{eq:tobound}
\sum_{p, q} a_{p q} (p q)^{-it} K \Big ( \frac{p}{P} \Big ) K \Big ( \frac{q}{Q} \Big ) \ll \frac{P Q}{(\log N)^{40}}
\end{equation}
To do this, we apply Cauchy-Schwarz, getting, by the prime number theorem,
$$
\Big ( \frac{P}{\log P} \Big )^{1/2} \cdot \Big ( \sum_{P / 2 \leq p \leq 5 P} \Big | \sum_{q} q^{-it} \cdot a_{p q} K \Big ( \frac{q}{Q} \Big ) \Big |^2 \Big )^{1/2}
$$
Instead of summing over primes $p$ we now sum over all integers.
Expanding the square the inner term is 
$$
\sum_{q_1, q_2} (q_1 / q_2)^{-it} K \Big ( \frac{q_1}{Q} \Big ) \overline{K} \Big ( \frac{q_2}{Q} \Big )
\sum_{P / 2 \leq n \leq 5P} a_{n q_1} \overline{a_{n q_2}}. 
$$
The contribution of $q_1 = q_2$ is $\ll Q P$. 
We also bound the contribution of the exceptional $q_1 \in [Q / 2, 5 Q] \cap \mathcal{P}_{\varepsilon,X }\cap \mathcal{S}_{\varepsilon, X}^{c}$ by 
$$
\ll \frac{P Q^2}{(\log N)^{100}}
$$
and similarly for the contribution of $q_2 \in [Q / 2, 5 Q] \cap \mathcal{P}_{\varepsilon, X} \cap \mathcal{S}_{\varepsilon, X}^{c}$. We can thus assume now that $q_1, q_2 \in \mathcal{S}_{\varepsilon, X}$ and that $q_1 \neq q_2$. By assumption the sum over $n$ is $\ll P / (\log P)^{100}\ll P / (\log N)^{100}$. Combining all these cases together shows that the above sum is
$$
\ll \frac{Q^2 P}{(\log N)^{100}} + \frac{Q^2 P}{(\log N)^{100}} \ll \frac{1}{\varepsilon^{100}} \cdot \frac{Q^2 P}{(\log N)^{100}}
$$
This shows that \eqref{eq:tobound} is
$$
\ll \frac{1}{\varepsilon^{50}} \cdot  \frac{P Q}{(\log N)^{50}} \ll \frac{P Q}{(\varepsilon \log N)^{50}}.
$$
Summing over all partitions $P$ and $Q$ and executing the integral over $s$
we get a final bound, 
$$
\ll \frac{1}{\eta^2 \varepsilon^{50}} \frac{N}{(\log N)^{50}}
$$
which is entirely sufficient. We notice that we can choose $\eta = \varepsilon$ to conclude.

\end{proof}

\section{Quantitative equidistribution results for the square of horocycle flows}
In view of the criterion from the previous section, it is crucial for our results to understand the behavior of orbits of the flow $h_t\times h_t$ in a quantitative sense. This was done in a recent breakthrough paper \cite{LMW}. In this section we recall the main results from \cite{LMW} and also present a minor strengthening which will be important for our analysis. 

Generally, one wants to apply Proposition \ref{prop:typ2} to the sequence $a_n=f(T^nx)$, where $T$ is a continuous map of a compact metric space $(X,d)$ and $f\in C(X)$. In the proofs it will follow that the constant $C>0$ (and hence the constants $C',C''$) will {\bf not} depend on $x\in X$, and in particular we will get uniform (over $x\in X$) bounds in \eqref{eq:sp} and \eqref{eq:mob'}.
In fact, the main application of the above proposition is to horocycle flows. Let $G=SL(2,\R)$, let $\Gamma$ be a lattice in $G$, let $X=G\slash\Gamma$ and let  $m_{X}$ denote the Haar measure.  Let
$$
h_t=\begin{pmatrix} 1&t\\0&1\end{pmatrix},\;\;\; v_t=\begin{pmatrix} 1&0\\t&1\end{pmatrix} \;\; \text{ and } a_t=\begin{pmatrix} e^{t/2}&0\\0&e^{-t/2}\end{pmatrix}
$$
be the unstable horocycle, the stable (opposite) horocycle and the geodesic flow, respectively,  acting on $(X, m_{X})$. We recall the following classical commutation relations:
$$
a_sh_ta_{-s}=h_{e^{s}t},\;\;\;\;\text{ and } \;\;\;\;\;a_sv_ta_{-s}=v_{e^{-s}t},\;\;\text{ for all }s,t\in \R.
$$
 Let  $H:=\{(g,g):\: g\in SL(2,\R)\}$.


\begin{theorem}[Theorem 1.2. in \cite{LMW}]\label{LMWdrugie} Assume $\Gamma$ is an arithmetic lattice. For every $(x,y)\in X\times X$ and large enough $R$ (depending explicitly on $X$), for any $T\geq R^{A_1}$, at least one of the following holds:
\begin{enumerate}
\item[E1.] For every $\varphi\in C^\infty_c(X\times X)$, we have
$$
\Big|\frac1T\int_0^T\varphi((h_r\times h_r)((x,y))\,dr-\int_{X\times X}\varphi\, dm_{X\times X}\Big|\leq S(\varphi)R^{-\kappa},
$$
where  $S(\varphi)$ is a certain Sobolev norm.
\item[E2.] There exists $(x_0,y_0)\in X\times X$ with $vol(H_{\cdot}(x_0,y_0))\leq R^{A_1}$, and for every $r\in[0,T]$ there exists $g\in SL(2,\R)\times SL(2,\R)$, $\|g\|<R^{A_1}$, such that
$$
d_{X\times X}\Big(h_s\times h_s(x,y),gH.(x_0,y_0)\Big)\leq R^{A_1} \Big(\frac{1+|s-r|}{T}\Big)^{1/A_2}
$$
for all $s\in[0,T]$.
\item[E3.] For every $r\in [0,T]$ and $t\in [\log R,\log T]$, the injectivity radius of $(a_{-t}\times a_{-t}))(h_r\times h_r)(x,y)$ is at most $R^{A_1}e^{-t}$.
\end{enumerate}
The constants $A_1, A_2,\kappa$ are positive and depend on $X$ but not on $(x,y)$.
\end{theorem}

\begin{remark}\label{rem:com} In applications $R=T^{\delta'}$ for some sufficiently small $\delta'>0$. In this case notice that if $\Gamma$ is a co-compact lattice then E3. never holds (for sufficiently large $T$) as the injectivity radius is uniformly bounded away from $0$ on $X$. 
\end{remark}

In fact we will apply Theorem \ref{LMWdrugie} to points $(x,y)$ of the form $(a_{-\log p}x, a_{-\log q} x)$, $x\in X$, where $p,q$ are different prime numbers which are $\leq T^{\delta^2}$ (with sufficiently small $\delta$).
We will now present a slight strenghtening of the above result where we show additionally that for points $(x,y)$ both generic for the Haar measure $\mu_X$, the element $g$ (in condition E2) can be taken from the centralizer of the flow $h_t\times h_t$. 
\subsection{Divergence along the direction of the centralizer}
The following result shows that the elements $g$ from condition E2 can be taken from the centralizer of the flow. The constants $\kappa, A_1,A_2$ are as in Theorem \ref{LMWdrugie}. Moreover, $0<\delta<\kappa/100$ be a sufficiently small parameter (to be specified later). In what follows we fix a compact set $K\in X$ satisfying $m_{X}(K)\geq 99/100$. For $x'\in X$ let $T_{x'}$ be the smallest number such that for $T'\geq T_{x'}$  we have 
\begin{equation}\label{eq:zvis}
\frac{1}{T'}\int_0^{T'}\chi_K(h_tx')dt\geq 98/100.
\end{equation}
Note that if $x'$ is generic for $m_{X}$, then $T_{x'}<\infty$.

\begin{proposition}\label{l:modular1}For $x',y'\in X$ let $T_{x',y'}:=2e^{T_{y'}\|y'\|}$. Then if $T>T_{x',y'}$ is satisfying that for some $R^{A_1}\leq T^{\frac{\delta}{1000A_2}}$  there exists $g\in SL(2,\R)\times SL(2,\R)$, $\|g\|<R^{A_1}$ and $(x_0,y_0)\in X\times X$ with $vol(H.(x_0,y_0))<R^{A_1}$ such that
\begin{equation}\label{eq:modu'}
d_{X\times X}((h_s\times h_s)(x',y'),gH.(x_0,y_0))<R^{A_1}\cdot T^{-\frac{\delta}{A_2}}\end{equation}
for  all $s<T^{1-\delta}$. Then there exists $(u,v)\in H.(x_0,y_0)$ and numbers $\{K_i\}_{i=1}^{T^\delta}$,  $\{K'_i\}_{i=1}^{T^\delta}$, with $\max_i (|K'_i|,|K_i|)\leq T^{2\delta}$ such that 
\begin{equation}\label{eq:moducent}
d_{X\times X}\Big((h_s\times h_s)(x',y'),(h_{K_i},h_{K'_i})(h_s\times h_s)(u,v)\Big)<R^{10A_1}\cdot T^{-\frac{\delta}{A_2}}\end{equation}
for  all $s\in I_i=[iT^{1-\delta}, (i+1)T^{1-\delta}]$, $i\leq T^\delta$.
\end{proposition}

\begin{proof}[Proof of Proposition \ref{l:modular1}] We split the proof of the proposition into two steps that we put as separate claims below.\\

{\bf CLAIM I:} Under the assumptions of Proposition \ref{l:modular1}, there exists $K\in \R$, $|K|\leq 2R^{A_1}$ such that 
\begin{equation}\label{eq:moducent'}
d_{X\times X}((h_s\times h_s)(x',y'),(id,h_K)H.(x_0,y_0))<30 R^{3A_1}\cdot T^{-\frac{\delta}{A_2}}\end{equation}
for  all $s<T^{1-\delta}$.\\

{\bf CLAIM II:} If \eqref{eq:moducent'} holds then \eqref{eq:moducent} holds. \\

We now proceed to the proofs of the above steps.
\begin{proof}[Proof of {\bf CLAIM I}]  Notice that by Lemma \ref{per:stab} and \eqref{eq:indvol} it follows that there exists $\alpha\in Comm(\Gamma)$, $R'=ind(\alpha)<C R^{A_1}$ and 
$$
\Delta_1,\ldots, \Delta_i\in \Gamma, \;\; i\leq R'
$$
so that $\Gamma\slash (\alpha\Gamma\alpha^{-1}\cap \Gamma)= \{\Delta_i\}_{i=1}^{R'}$ and 
$$
H_{\cdot}(x_0,y_0)=\{(\xi\Gamma, \xi\Delta_i\alpha \Gamma)\;:\; \xi\in G, i\leq R'\}.
$$
Then writing $gH_{\cdot}(x_0,y_0)= (id,g')H_{\cdot}(x_0,y_0)$ and denoting $g'$ by $g$ (to simplify notation), we get that \eqref{eq:modu'} implies that for every $s\leq T^{1-\delta}$ there are $\xi_s,\gamma_s, \gamma'_s, \Delta_{s}$ such that 
$$
d_G(h_sx', \xi_s\gamma_s)<R^{A_1}\cdot T^{-\frac{\delta}{A_2}} \;\;\text{ and }\;\; d_G(h_sy', g\xi_s\Delta_s\alpha \gamma'_s)<R^{A_1}\cdot T^{-\frac{\delta}{A_2}}.
$$
Since $\|g\|\leq R^{A_1}$, it follows that $d_G(g h_sx', g\xi_s\gamma_s)<R^{3A_1}\cdot T^{-\frac{\delta}{A_2}}$.

Using right invariance and triangle inequality this implies that for every $s\leq T^{1-\delta}$, 
\begin{equation}\label{eq:ts}
d_G(gh_sx',h_sy'\gamma'^{-1}_s\alpha^{-1}\Delta_s^{-1}\gamma_s)<2 R^{3A_1}\cdot T^{-\frac{\delta}{A_2}}.
\end{equation}

Assume there exists $\tilde{t}\in [0,T^{1-\delta}]$ such that  
$$d_G(gh_{\tilde{t}}x', h_{\tilde{t}}y'\gamma'^{-1}_{0}\alpha^{-1}\Delta_{0}^{-1}\gamma_{0})= 10R^{3A_1}\cdot T^{-\frac{\delta}{A_2}}.$$
We assume WLOG that $\tilde{t}$ is the smallest number in $[0,T^{1-\delta}]$ with this property. Let $z=y'\gamma'^{-1}_{0}\alpha^{-1}\Delta_{0}^{-1}\gamma_{0}x'^{-1}$. Using \eqref{eq:ts} for $s=0$ gives $d_G(g,z)< 2R^{3A_1}\cdot T^{-\frac{\delta}{A_2}}$. Then by triangle inequality, for any $t\in [0,\tilde{t}]$, $d_G(z,h_tzh_{-t})<12R^{3A_1}\cdot T^{-\frac{\delta}{A_2}}$ and $d_G(z,h_tzh_{-t})\geq 8R^{3A_1}\cdot T^{-\frac{\delta}{A_2}}$. These two bound together with the second part of Lemma \ref{lem:zz} (polynomial divergence) imply that  there is a set $V\subset [0,\tilde{t}]$  with  $|V|\geq \frac{\tilde{t}}{10}$ such that  for  $t\in V$, we have
\begin{equation}\label{eq:poldiv}
d_G(gh_{t}x', h_{t}y'\gamma'^{-1}_{0}\alpha^{-1}\Delta_{0}^{-1}\gamma_{0})\geq 4 R^{3A_1}\cdot T^{-\frac{\delta}{A_2}}.
\end{equation}

By the definition of $\tilde{t}$ and \eqref{eq:ts} (using triangle inequality) imply that for any  $t\in [0,\tilde{t}]$,
\begin{equation}\label{eq:asc}
d_G(h_ty'\gamma'^{-1}_t\alpha^{-1}\Delta_t^{-1}\gamma_t, h_{t}y'\gamma'^{-1}_{0}\alpha^{-1}\Delta_{0}^{-1}\gamma_{0})<12R^{3A_1}\cdot T^{-\frac{\delta}{A_2}}.
\end{equation}
{\bf Fact:} If $T>T_{x',y'}$, then there exists $t\in V$ and $\hat{\gamma}_t\in \Gamma$ with  $\|h_ty'\hat{\gamma}_t\|\leq \log T$. 
\begin{proof}[Proof of the {\bf Fact}] Since $T\geq T_{y'}$ it follows that \eqref{eq:zvis} holds.  Note that  if $t\in V$ is such that  $\chi_{K}(h_{t}y')=1$ then such  $\hat{\gamma}_t$ exists (as we return to the compact set $K$). If $\tilde{t}\geq T_{y'}$ then the set of $t\leq \tilde{t}$ for which $\chi_{K}(h_{t}y')=1$ has measure $\geq 98\tilde{t}/100$ and so by $|V|\geq \frac{\tilde{t}}{10}$ the proof is finished in this case. On the other hand if $\tilde{t}\leq T_{y'}$
then for any $t\leq \tilde{t}<T_{y'}$ the point $\|h_ty'\|\leq \|h_t\|\|y'\|\leq T_{y'} \|y'\|< \log T$. This finishes the proof.
\end{proof}

Let $t$ come from the {\bf Fact}. Denoting $\bar{\gamma}_t=\hat{\gamma}_t^{-1}\gamma'^{-1}_{0}$ and $\tilde{\gamma}_t=\hat{\gamma}_t^{-1}\gamma'^{-1}_t$, \eqref{eq:asc} translates to 
$$
d_G(h_ty' \hat{\gamma}_t \bar{\gamma_t}\alpha^{-1}\Delta_{0}^{-1}\gamma_{0},h_ty'\hat{\gamma}\tilde{\gamma}_t\alpha^{-1}\Delta_t^{-1}\gamma_t)<12R^{3A_1}\cdot T^{-\frac{\delta}{A_2}}.
$$
Since  $\|h_ty'\hat{\gamma}_t\|\leq  \log T$ it follows that 
$$
d_G\Big(\bar{\gamma_t}\alpha^{-1}\Delta_{0}^{-1}\gamma_{0}\Big(\tilde{\gamma}_t\alpha^{-1}\Delta_t^{-1}\gamma_t\Big)^{-1},e\Big)=$$$$d_G(\bar{\gamma_t}\alpha^{-1}\Delta_{0}^{-1}\gamma_{0},\tilde{\gamma}_t\alpha^{-1}\Delta_t^{-1}\gamma_t)<12R^{3A_1}\cdot T^{-\frac{\delta}{A_2}}(\log T)^3.
$$
But since $\Delta_{0}, \Delta_t\in \Gamma$ and the index of $\alpha$ is $\leq CR^{A_1}$ it follows that the above inequality can only hold if 
$$\bar{\gamma_t}\alpha^{-1}\Delta_{T_0q}^{-1}\gamma_{T_0q}\Big(\tilde{\gamma}_t\alpha^{-1}\Delta_t^{-1}\gamma_t\Big)^{-1}=e.$$
This implies that 
$$\gamma'^{-1}_{0}\alpha^{-1}\Delta_{0}^{-1}\gamma_{0}=\gamma'^{-1}_{t}\alpha^{-1}\Delta_t^{-1}\gamma_t.$$

But then using \eqref{eq:ts} (for $s=t$) and \eqref{eq:poldiv}, we get a contradiction. This means that such number $\tilde{t}$ does not exist. This implies that for every $t\in [0,T^{1-\delta}]$, we have 
$$
d_G(gh_{t}x', h_{t}y'\gamma'^{-1}_{0}\alpha^{-1}\Delta_{0}^{-1}\gamma_{0})\leq 10 R^{3A_1}\cdot T^{-\frac{\delta}{A_2}}.
$$
Denote $z_{0}=y'\gamma'^{-1}_{0}\alpha^{-1}\Delta_{0}^{-1}\gamma_{0}x'^{-1}$. Using the above for $t=0$  we get $d_G(g,z_{0})\leq  10 R^{3A_1}\cdot T^{-\frac{\delta}{A_2}}$ and then using the above for $t\in [0,T^{1-\delta}]$, $d_G(g,h_tz_{0}h_{-t})\leq 10 R^{3A_1}\cdot T^{-\frac{\delta}{A_2}}$. So by triangle inequality, we get that for every $t\in [0,T^{1-\delta}]$
$$
d_G(z_{0}, h_{t}z_{0}h_{-t})\leq 20 R^{3A_1}\cdot T^{-\frac{\delta}{A_2}}.
$$
This however by Lemma \ref{lem:zz} implies that for some $K\in \R$, $d(z_{0},h_K)\leq T^{-1+\delta}$. Since $d_G(g,z_0)\leq  10 R^{3A_1}\cdot T^{-\frac{\delta}{A_2}}$, we get $d_G(g,h_K)\leq 20 R^{3A_1}\cdot T^{-\frac{\delta}{A_2}}$. This in particular, by the bound on $\|g\|$ implies that $|K|<2R^{A_1}$. Summarizing, by triangle inequality and the assumptions of {\bf CLAIM I},
$$
d_{X\times X}((h_s\times h_s)(x',y'),(id,h_K)H.(x_0,y_0))<30 R^{3A_1}\cdot T^{-\frac{\delta}{A_2}}
$$
This finishes the proof.

\end{proof}

\begin{proof}[Proof of {\bf CLAIM II}] We will start by proving the following: if \eqref{eq:moducent'} holds, then there exists a constant $C>0$ and  $(u,v) \in H.(x_0,y_0)$ such that for all $t\in [0,T^{1-\delta}]$
\begin{equation}\label{eq:onep}
d_{G\times G}\Big((h_t\times h_t)(x',y'),(id,h_K)(h_t\times h_t)(u,v)\Big)\leq C R^{10A_1}\cdot T^{-\frac{\delta}{A_2}},
\end{equation}
 (in the above we mean that there are lifts of $(x,y)$ and $(u,v)$ to $G$ such that the above holds for the flow on $G$.) 
 Nore that by the bound on $K$ in \eqref{eq:moducent'} it follows that  for $s\leq T^{1-\delta}$,
 \begin{equation}\label{eq:moda'}
 d_{X\times X}((h_s\times h_s)(id,h_{-K})(x',y'),H.(x_0,y_0))<60 R^{5A_1}\cdot T^{-\frac{\delta}{A_2}}.
 \end{equation}

 Let $\{X_1, U_1, V_1\}$ and $\{X_2, U_2, V_2\}$ denote the generators of the coordinate subalgebras 
 $\mathfrak{sl}(2, \R) \times \{0\}$ and $\{0\} \times \mathfrak{sl}(2, \R)$ in $\mathfrak{sl}(2, \R)^2$, respectively, satisfying the commutations relations
 $$
 [X_i, U_i]=U_i\,, \quad  [X_i, V_i]=-V_i\,, \quad [U_i, V_i]= 2X_i \,, \quad \text{ for } i=1,2\,.
 $$
  Let us denote
 $$
 X^\pm = X_1\pm X_2\,, \quad U^\pm= U_1\pm U_2\,, \quad V^\pm= V_1\pm V_2\,.
 $$
 Let $\mathfrak{g}^\pm$ denote the subspaces generated by $\{X^\pm, U^\pm, V^\pm\}$ respectively.
 
 \noindent Note that $\mathfrak {g}^+$ is a basis of the Lie algebra $\mathfrak (h) = \text{Lie}(H)$ of the diagonally embedded
 $H \equiv SL(2, \R) < SL(2, \R)^2$ and $U^+$ is the generator of the diagonal unipotent $\{h_t \times h_t\}$.   
 In addition, we have that $[U^+, \mathfrak{g}^-] \subset  \mathfrak{g}^-$. 
 
 By \eqref{eq:moda'} there exists $(u,v) \in H.(x_0,y_0)$ such that 
 $$
 (id,h_{-K})(x',y') = \exp \Big( x^- X^- + u^- U^- + v^- V^-\Big)(u,v) \,,
 $$
 with $ \vert x^- \vert  + \vert u^- \vert  + \vert v^- \vert   < 60R^{5A_1}\cdot T^{-\frac{\delta}{A_2}}$.  We then have
 $$
 (h_t \times h_t) (id,h_{-K}) (x',y')  = \exp ( t U^+)  \exp \Big(x^- X^- + u^- U^- + v^- V^-\Big)  (u,v) \,.
 $$
 Since $[U^+, \mathfrak{g}^-] \subset  \mathfrak{g}^-$, for every $t \in \R$, there exists $ (x^-(t), u^-(t), v^-(t))$ such that 
 $$
  \exp (t U^+)  \exp \Big(x^- X^- + u^- U^- + v^- V^-\Big) \exp(-tU^+)  = \exp \Big(x^-(t) X^- + u^-(t)  U^- + v^-(t)  V^-\Big)
 $$
 and moreover the functions  $(x^-(t), u^-(t), v^-(t))$ are polynomials in $t$ (since $U^+$ is nilpotent). This implies that 
 $$
 (h_t \times h_t)  (id,h_{-K})(x',y') =  \exp \Big(x^-(t) X^- + u^-(t)  U^- + v^-(t)  V^-\Big) (h_t \times h_t)(u,v)\,.
 $$
 Let $0<\tilde{t}<T^{1-\delta}$ be the smallest such that $\max(|x^-(t)|, |u^-(t)|, |v^-(t)|)= C R^{8A_1}\cdot T^{-\frac{\delta}{A_2}}$ (for a constant C to be specified below).  Since the functions  $(x^-(t), u^-(t), v^-(t))$ are polynomials in $t$ it follows that there is a set $V\subset [0,\tilde{t}]$ such that $|V|\geq \tilde{t}/10$
and for $t\in V$, $\max(|x^-(t)|, |u^-(t)|, |v^-(t)|)\geq \frac{C}{100} R^{8A_1}\cdot T^{-\frac{\delta}{A_2}}$. Assume that there exists a $t_0\in V$ for which  $(h_{t_0} \times h_{t_0}) (x',y')\in \tilde{K}$ where $\tilde{K}$ is a fixed compact set of measure $\geq 99/100$. The proof of existence of such $t$ is analogous to the proof of the {\bf Fact} inside the proof of {\bf CLAIM I} and so we skip it here.

Then by the above and \eqref{eq:moda'} it follows that 
 $$
 d_{X\times X}\Big(\exp \Big(x^-(t) X^- + u^-(t)  U^- + v^-(t)  V^-\Big) (h_t \times h_t)(u,v), H.(x_0,y_0)\Big)< 60 R^{5A_1}\cdot T^{-\frac{\delta}{A_2}}.
 $$ 
This however is a contradiction as on the fixed set $\tilde{K}$, $\mathfrak{h}^-$ is uniformly transverse to the tangent space of 
 $ H.(x_0, y_0)$ (which is equal to $\mathfrak{h}^+$). It is here where we choose the constant $C=C_K>0$: by uniform transversality it follows that if $h\in \mathfrak{h}^-$ satisfies $\|h\|\geq \xi$, then 
 $ d_{X\times X}\Big(\exp(h)H.(x_0,y_0), H.(x_0,y_0)\Big)\geq \xi/C_K.$ 
 This means that such $\tilde{t}$ does not exist and so  for all  $0<t<T^{1-\delta}$, $\max(|x^-(t)|, |u^-(t)|, |v^-(t)|)\leq C R^{8A_1}\cdot T^{-\frac{\delta}{A_2}}$. In particular \eqref{eq:onep} holds. We will now show that for every $t_0\leq T$ there exists $K(t_0)\in \R$, $|K(t_0)|\leq T^{2\delta}$ such that 
$$
d_X(h_{t_0}x,h_{K(t_0)+t_0}u)<T^{-1+3\delta},
$$
where $|K(t)-K(t')|<$ for $|t-t'|\leq T^{1-\delta}$. Moreover an analogous statement holds for $y$ and $v$.  Note that \eqref{eq:onep} implies that $d_G(h_t(xu^{-1})h_{-t}, e)\leq CR^{3A_1}\cdot T^{-\delta/A_2}$. This by Lemma \ref{lem:zz} implies that if  $xu^{-1}=\begin{pmatrix}a&0\\c&a^{-1}\end{pmatrix}$, then  $|c|\ll T^{-2+2\delta}$ and $|1-a|\ll T^{-1+\delta}$. Let $K(t):=\frac{(a^{-1}-a)t-ct^2 }{a+ct}$ so that for $t\leq T$, $|K(t)|\leq T^{2\delta}$. Moreover the bound on $1-a$ and $c$ implies that 
$|K(t)-K(t')|<C R^{3A_1}\cdot T^{-\frac{\delta}{A_2}}$ for $|t-t'|\leq T^{1-\delta}$ (by a direct computation). But then using the formula in Lemma \ref{lem:zz} again, it follows that 
$$
h_t(xu^{-1})h_{-t-K(t)}=\begin{pmatrix}a+ct&0\\c&a^{-1}-ct-K(t)c\end{pmatrix},
$$
it is enough to notice that by the bounds on $a-1,c$ the above matrix  is $T^{-1+3\delta}$ close to $id$. This finishes the proof.
 \end{proof} 
 \end{proof}
In fact the last part of the above proof gives the following important statement that we will put as a separate corollary:

\begin{corollary}\label{cor:uni} For every $t_0\leq T$ there exists $K(t_0)\in \R$, $|K(t_0)|\leq T^{2\delta}$ such that 
$$
d_X(h_{t_0}x,h_{K(t_0)+t_0}u)<T^{-1+3\delta},
$$
with an analogous statement for $y$ and $v$.
\end{corollary}

Proposition \ref{l:modular1} has the following crucial corollary:
\begin{corollary}\label{cor:pq} Let $x\in X$ be generic for $\mu_X$. Then there exists $T_x$ such that for $T\geq T_x$ the following holds: let $p,q\leq T^{\delta^2}$  and consider $(x',y')=(a_{-\log p}x,a_{-\log q}x)$. If $(x',y')$ satisfies \eqref{eq:modu'} then it satisfies \eqref{eq:moducent}, with $|K_i|\leq T^{3\delta}$ and $R^{11A_1}T^{-\delta/A_2+4\delta^2}$ on the RHS.
\end{corollary}
\begin{proof} Note that  $(h_{s}\times h_{s})(a_{-\log p}x,a_{-\log q}x)= (a_{-\log q}\times a_{-\log q})(h_{s/q}\times h_{s/q})(a_{\log\frac{p}{q}})x,x)$. By the bound on $q$, \eqref{eq:modu'} it follows that for $s\leq T^{1-\delta}$, 
$$
d_{X\times X}\Big((h_{s/q}\times h_{s/q})(a_{\log p/q}x,x), \Big((a_{\log q}\times a_{\log q})g\Big)H.(x_0,y_0)\Big)<R^{A_1}\cdot T^{-\frac{\delta}{A_2}+4\delta^2}.
$$
 We can apply Proposition \ref{l:modular1}, for $T>T_{x}$ which does not depend on $p,q$. 
  Applying $a_{-\log q}\times a_{-\log q}(\cdot)$ to the LHS and applying renormalization equation, we get 
$$
d_{X\times X}\Big((h_{s}\times h_{s})(a_{-\log p}x,a_{-\log q}x),(a_{-\log q}\times a_{-\log q})(id,h_{K_i})(h_s\times h_s)(u_0,v_0)\Big)<R^{10A_1}T^{-\delta/A_2+4\delta^2}
$$
It remains to notice that $(a_{-\log p}\times a_{-\log q})(id,h_{K_i})(h_s\times h_s)(u_0,v_0)=(id,h_{\tilde{K}_i})(h_{qs}\times h_{qs})(u_0,v_0)$. This finishes the proof.
\end{proof}

\begin{lemma}\label{lem:zz}For every $z=\begin{pmatrix}a&0\\c&a^{-1}\end{pmatrix}\in G$, we have 
$$
h_tzh_{-t}=\begin{pmatrix} a+ct& (a^{-1}-a)t-ct^2 \\ c &a^{-1}-ct  \end{pmatrix}
$$
and
$$
 h_tzh_{-t}z^{-1}=\begin{pmatrix} 1+act+c^2t^2&(1-a^2)t- act^2 \\ c^2 t &1-act  \end{pmatrix}
$$
i.e. the entries are polynomials in $t$ and coefficients of $z$.
\end{lemma}

\subsection{Equidistribution for discrete time}
Notice that we want to apply the SPNT-criterion to the sequence $a_n=\phi(h_nx)$ given by a smooth function $\phi$ on $X$ evaluated along the orbit of the time-$1$ map of the horocycle flow. In this section using the technique of Venkatesh, \cite{Venkatesh}, we show how to pass from results for the flow to the time-$1$ map. The SPNT criterion for the sequence $a_n$ requires us to study $\sum_{n\leq N} (\phi\times \phi)(h_{pn}\times h_{qn}(x,x))$, for primes $p_1,p_2$ in some range. Notice that 
\begin{equation}\label{eq:resc1}
\sum_{n\leq N} (\phi\times \phi)(h_{p_1n}\times h_{p_2n}(x,x))= \sum_{n\leq N}( \phi\circ a_{p_1}\times \phi\circ a_{p_2}) (h_{n}\times h_n)(a_{-p_1}x,a_{-p_2}x)),
\end{equation}
In applications we will have an upper bound on $p_1,p_2\leq N^{\delta^2}$. The following result allows us to pass from quantitative equidistribution for flow to the time-$1$ map.

\begin{proposition}\label{prop:zzz} Let $(x,y)\in X\times X$ satisfy
$$
\Big|\frac1T\int_0^T\varphi((h_r\times h_r)((x,y))\,dr-\int_{X\times X}\varphi\, dm_{X\times X}\Big|\leq S(\varphi)R^{-\kappa},
$$
where $\varphi\in C^{\infty}(X\times X)$ and $R\in[T^{\delta^{3/2}}, T^{\frac{\delta}{100 A_1 A_2}}]$. Then there exists a global constant $\tilde{\kappa}>0$ such that 
$$
\Big|\frac1T\sum_{n=0}^T\varphi((h_n\times h_n)((x,y))-\int_{X\times X}\varphi\, dm_{X\times X}\Big|\leq S(\varphi)R^{-\tilde{\kappa}}.
$$
\end{proposition}

The above proposition can be deduced straightforwardly from  the proof of Theorem 3.1. in \cite{Venkatesh}. The method in \cite{Venkatesh} (to go from the quantitative distribution for the flow to time-$1$ map)  can be applied to any flow which is polynomially mixing, has polynomial equidistribution and has polynomial growth of derivatives. We will explain the main steps here for completeness.
\begin{proof}[Sketch of proof of Proposition \ref{prop:zzz}]
First step is control on twisted ergodic integrals, i.e. orbital integrals twisted by a character. This is Lemma 3.1. in \cite{Venkatesh}. Notice that Lemma 3.1. in \cite{Venkatesh} only uses polynomial mixing of the flow $n(t)$ (it will be $h_t\times h_t$ in our case), and (polynomial) control on the Sobolev norms of $S(\phi\circ (h_{m}\times h_m))$, for $m\leq H=T^{\delta'}$.

In the second step we take a bump function on $\R$ (it is denoted $g_{\delta}(\cdot)$  in \cite{Venkatesh}), which allows to write
$$
\sum_{n\leq N}\phi(h_nx,h_ny)=\int_{0}^N g_{\delta}(t)\phi(h_tx,h_ty)dt +O(\delta N).
$$
One then writes the periodic function $g_{\delta}(t)$ as $\sum_{k\in \Z}a_ke(kt)$, and uses bounds $|a_k|\leq C\delta^{-1}$ and Lemma 3.1 for the twisted integrals $\int_{0}^N e(kt)\phi(h_tx,h_ty)dt$.
\end{proof}

Notice that Proposition \ref{prop:zzz} together with \eqref{eq:resc1} imply the following:
\begin{corollary}\label{cor:adc}Assume $p,q\leq T^{\delta^2}$ and that the point $(x',y')=(a_{-\log p}x,a_{-\log q}x)$ satisfies E1 with $R\in[T^{\delta^{3/2}}, T^{\frac{\delta}{100 A_1 A_2}}]$. Then for every $\phi\in C^{\infty}(X)$ with mean $0$,
$$
|\sum_{n\leq T} (\phi\times \phi)(h_{pn}\times h_{qn}(x,x))|\leq T^{1-\eta},
$$
for some $\eta>0$.
\end{corollary}

Therefore for the rest of the paper we will be studying continuous averages for the flow $(h_t\times h_t)$ keeping in mind that quantitative equidistribution for the flow (i.e. condition E1) implies the corresponding statement for the time-$1$ map and hence also condition E1 for $p,q$ implies that the condition in the SPNT criterion holds.

\subsection{Periodicity and Ratner's theory}

We have the following lemma which works for {\bf any} lattice in $SL(2,\R)$.  Recall that  $$COMM(\Gamma):=\{g\in G\;:\; g\Gamma g^{-1}\cap \Gamma\text{  has finite index in both }g\Gamma g^{-1}\text{ and }\Gamma\}.$$ Using the results of Ratner, \cite{Ratner} we have the following lemma :

\begin{lemma}\label{per:stab} $H.(x_0,y_0)$ is periodic if and only if there exists $\alpha \in COMM(\Gamma)$ such that
\begin{equation}\label{period1}
H.(x_0,y_0)=\{(\xi_1\Gamma,\xi_1\gamma_i\alpha\Gamma):\:\xi_1\in SL_2(\R),\;i=1,\ldots,n \},\end{equation}
where
$$\Gamma_\alpha:=\Gamma/(\Gamma\cap \alpha\Gamma\alpha^{-1}) =\{
\gamma_1(\Gamma\cap \alpha\Gamma\alpha^{-1}),\ldots, \gamma_n(\Gamma\cap \alpha\Gamma\alpha^{-1})\}.$$
\end{lemma}
Recall that the number $n$ is called the {\em index} of $\alpha\in COMM(\Gamma)$ and we denote it $ind(\alpha)$.  
By Corollary~3 in \cite{Ratner}, the flow $h_t \times h_t$  on $H.(x_0,y_0)$  is isomorphic (via an algebraic isomorphism) to the unipotent flow
$h_t$ on $G/\Gamma_\alpha$, hence in particular 
\begin{equation}\label{eq:indvol}
vol(H.(x_0,y_0))=ind(\alpha) vol (G/\Gamma) \,. 
\end{equation}

\subsubsection{Co-compact arithmetic case}\label{sec:comp}

Let $\Gamma$ be a co-compact arithmetic lattice. In this section we establish quantitative properties of $COMM(\Gamma)$. The proof is based on a quantitative version of the analysis in Bourgain, Sarnak and Ziegler \cite{BSZ},~\footnote{We apply the reasoning in the proof of Lemma 2 in \cite{BSZ} which works for all $\beta\in COMM(\Gamma)$, not only for those that stabilize a point for the natural action of $SL_2(\R)$ on the projective line.}  we will need to make the argument quantitative. We will use the same notation as in \cite{BSZ}. 

Since $\Gamma$ is co-compact arithmetic it follows that $\Gamma$ is commensurable with the lattice given by the embedding $\phi (A_1(\Z))$ into $M_2(\R)$ of the integral unit group $A_1(\Z)$ in a quaternion algebra $A(\Q)$ defined over a totally real number field. 
As in \cite{BSZ}, for $\alpha= x_0+x_1\omega+ x_2\Omega+x_3\omega \Omega$ we define $N(\alpha)=x_0^2-ax_1^2-bx_2^2+abx_3^2$ (with $a=\omega^2$, $b=\Omega^2$ being two rationals which are square-free) and $tr(\alpha)=2x_0$. We define (see (3.6) in \cite{BSZ}),
 $$
 \phi(\alpha)= \begin{pmatrix}\bar{\xi}&\eta\\ b\bar{\eta}&\xi \end{pmatrix},
  $$
 where $\xi=x_0-x_1\sqrt{a}$, $\eta=x_2+x_3\sqrt{a}$ (see (3.12) in \cite{BSZ}).  Note that $det(\phi(\alpha))=N(\alpha)$  and $trace(\phi(\alpha))=tr(\alpha)$. We have (see (3.15) in \cite{BSZ})
 $$
 COMM(\Gamma)=\Big\{\frac{\phi(\alpha)}{\sqrt{N(\alpha)}}\;:\; \alpha\in A^+(\Q),\Big\}.
 $$
 where $A^+(\Q)=\{\alpha\in A(\Q)\;:\; N(\alpha)>0\}$. Then (see (3.16) in \cite{BSZ}), up to multiplication by scalars,
 $$
 \beta\in COMM(\Gamma) \text{ iff } \beta=\begin{pmatrix}\bar{\xi}&\eta\\ b\bar{\eta}&\xi \end{pmatrix}\text{ with } \xi+\eta\Omega\in A^+(\Q).
 $$
 \begin{lemma}\label{lem:per'} Let $\beta=\phi(\alpha)/ N(\alpha)^{1/2}$, where $\alpha=x_0+x_1\omega+ x_2\Omega+x_3\omega \Omega\in A(\Q)$.  Then the denominator of $N(\alpha)^{-1}x_0^2$,  $N(\alpha)^{-1}x_1^2$, $N(\alpha)^{-1}x_2^2$ and $N(\alpha)^{-1}x_3^2$ are $\leq C \text{ind}(\beta)$.
 \end{lemma}
 \begin{proof}
 Let $\beta = \phi (\alpha) / N(\alpha)^{1/2}\in COMM(\Gamma)$ and denote $\Gamma'=\Gamma\cap \beta \Gamma \beta^{-1}$ and let $\Gamma/\Gamma'=\{\gamma_1,\ldots, \gamma_n, \gamma_i\in \Gamma\}$. Consider the fundamental solution of the Pell's equation $A^2-aB^2=1$. Such solution exists as $a>0$ is not a square. Let $
R=\begin{pmatrix}A-B\sqrt{a}&0\\ 0&A+B\sqrt{a}\end{pmatrix}\in \phi(A_1(\Z)).
$
Then $R\in \gamma_i {\beta} \Gamma{\beta}^{-1}=\gamma_i {\beta} \Gamma\gamma_i^{-1}{\beta}^{-1}$. We will consider the matrix 
$\gamma_i{\beta}$ instead of ${\beta}$ and to simplify notation we still call it ${\beta}$. Let $\alpha= x_0+x_1\sqrt{a}+x_2\sqrt{b}+x_3\sqrt{ab}$ and let 
$$
\phi(\alpha)=\left[\begin{array}{cc} \overline{\gamma}& {\delta}\\
                     b {\bar \delta}&\gamma\end{array}\right]$$
with $\gamma = x_0-x_1\sqrt{a}$, $\delta=x_2+x_3\sqrt{a}$. It follows that 
$$
\phi(\alpha)R\phi(\alpha)^{-1}\in \phi(A_1(\Z) ).
$$
We recall the formula
$$
\begin{aligned}
&\phi(\alpha){M}\phi(\alpha)^{-1} = \frac{1}{N(\alpha)} \\ & \times \left[\begin{array}{cc}  \overline{\xi} N(\alpha)  + ( \overline{\xi}-\xi) b|\delta|^2+ b (\delta \gamma \overline{\eta}  - \overline{\delta} \overline{\gamma} \eta  ) &  -b \overline{\eta} \delta^2 + \eta \overline{\gamma}^2  + \delta \overline{\gamma} (\xi -\overline{\xi})  \\
                     -b^2\eta \overline{\delta}^2 + b \overline{\eta} \gamma^2  + b \overline{\delta} \gamma (\overline{\xi}-\xi) &
                      \xi N(\alpha)  + ( \xi - \overline{\xi} ) b|\delta|^2+ b ( \overline{\delta} \overline{\gamma} \eta -\delta \gamma \overline{\eta}) \end{array}\right],
\end{aligned}
$$     
where $M=\begin{pmatrix}\bar{\xi}&\eta\\ b\bar{\eta}&\xi \end{pmatrix}{\in \phi(A_1 (\Z))}$. 

Using this for $M=R$, we get that 
$N(\alpha)^{-1} Bb|\delta|^2, N(\alpha)^{-1} B\delta\bar{\gamma}\in  \Z+\Z\sqrt{a}$. 

\noindent Consider now the general formula above, where ${M}$ is any matrix for which $\phi(\alpha){M}\phi(\alpha)^{-1}\in \phi(A_1(\Z))$. Multiplying by $B$ each term of the matrix $\phi(\alpha){M}\phi(\alpha)^{-1}\in \phi(A_1(\Z))$ and using the knowledge $N(\alpha)^{-1} Bb|\delta|^2, N(\alpha)^{-1} B\delta\bar{\gamma}\in  \Z+\Z\sqrt{a}$, we get that in particular
$$
N(\alpha)^{-1} Bb(\eta \bar{\gamma}^2- b\bar{\eta}\delta^2)\in \Z+\Z\sqrt{a}.
$$
Using this for  any $\eta_1$ and $\eta_2$, we get 
\begin{equation}
\label{eq:etas}
N(\alpha)^{-1} B(\eta_1 \bar{\gamma}^2- b\bar{\eta}_1\delta^2),  \quad  N(\alpha)^{-1} B(\eta_2 \bar{\gamma}^2- b\bar{\eta}_2 \delta^2) \in \Z+\Z\sqrt{a}.
\end{equation}
Multiplying the first inclusion by $\eta_2$ and the second one by $\eta_1$ and subtracting from each other, we get 
$$
N(\alpha)^{-1} Bb\delta^2[\bar{\eta}_1 \eta_2- \eta_1\bar{\eta}_2]\in \Z+\Z\sqrt{a}.
$$
If $\bar{\eta}_1 \eta_2= \ell+m\sqrt{a}$, then it follows that 
$$
2mN(\alpha)^{-1}Bb\sqrt{a}\delta^2\in \Z+\Z\sqrt{a}.
$$
Write $\delta=x_2+x_3\sqrt{a}$. Then the above condition and the condition  that $Bb \vert \delta\vert ^2 \in \Z$  give
$$
2mN(\alpha)^{-1}Bb (x_2^2+ax_3^2)\in \Z  \quad \text{ and }  \quad N(\alpha)^{-1} Bb  (x_2^2 -a x_3^2)  \in \Z.
$$
This two conditions in turn imply that 
\begin{equation}
\label{eq:x_2x_3}
2mN(\alpha)^{-1} Bb x_2^2\in \Z\quad \text{ and } \quad 2m N(\alpha)^{-1} B b a x_3^2 \in \Z\,,
\end{equation} 
So the denominator of $x_2^2$ is $\leq 2mB$ and
that of $x_3^2$ is $\leq 2mBa$.  It remains to notice that $B$ is a fixed constant (depending only on $a$ as a solution of the Pell's equation), and $m$ comes from  $\eta_2\bar{\eta}_1$. Since the above reasoning works for any $\eta_1,\eta_2$ it follows that we can pick $\eta_i$ to be $\leq \text{ \rm ind}(\alpha)$. This finishes the proof as far as $x_2$ and $x_3$ are concerned.

Then, multiplying the first inclusion in formula \eqref{eq:etas} by $\bar{\eta}_2$ and the second one by $\bar{\eta}_1$ and subtracting from each other, we get 
$$
N(\alpha)^{-1} B\bar \gamma^2[\eta_1\bar{\eta}_2-\bar{\eta}_1\eta_2]\in \Z+\Z\sqrt{a}\,,
$$
hence, in particular, we derive
\begin{equation}
\label{eq:x_0x_1}
2m N(\alpha)^{-1}B (x_0^2 + a x_1^2) \in \Z  \,.
\end{equation}
Since $N(\alpha)=x_0^2-ax_1^2-bx_2^2+abx_3^2 \in \Q$ by formula \eqref{eq:x_2x_3}
it follows  that we have 
$$
2 m N(\alpha)^{-1}B  (x_0^2-ax_1^2) = 2 mB +  2 m N(\alpha)^{-1} Bb (x_2^2 - a x_3^2)      \in \Z\,,
$$
which together with the condition \eqref{eq:x_0x_1} gives that $2m N(\alpha)^{-1}B  x_0^2 \in \Z$ and $2mN(\alpha)^{-1}B  a x_1^2 \in \Z$, thereby completing the argument.

\end{proof}

\subsubsection{The modular case}
Let now $\Gamma=SL(2,\Z)$. We will be interested in points of the form $(a_{-\log p}x, a_{-\log q}x)$, where $x\in X$. Assume that $x\in X$ is such that 
$(a_{-\log p}x, a_{-\log q}x)$ satisfies E2 with $R=T^\delta$ and $p,q\leq T^{\delta^2}$. Let $H.(x_0,y_0)$ be the corresponding periodic orbit. By Ratner's results \cite{Ratner} it follows that the flow $h_t\times h_t$  on $H.(x_0,y_0)$ is algebraically conjugated with the flow $h_t$ on $SL(2,\R)\slash \Gamma_{p,q}$, where $\Gamma_{p,q}=\beta\Gamma \beta^{-1}\cap \Gamma$ for some $\beta=\beta(p,q,x)\in COMM(SL(2,\Z))$. From now on the lattice $\Gamma_{p,q}$ will always denote the lattice associated with $H.(x_0,y_0)$ with the corresponding element $\beta=\beta_{p,q}$. Recall that 
$$
COMM(SL(2,\Z))=\Big \{\frac{1}{det(A)^{1/2}}A\;:\; A\in GL_2^+(\Q)\Big \}  
 \subset  \{ \beta \in SL(2, \R) \vert  \beta^2 \in GL_2^+(\Q)\} . 
$$

\begin{lemma}\label{lem:alpha} Let $p,q$ be two different integers with $\log^{1/\eta}T<p,q<T^{\delta^2}$. Assume $x\in X$ is such that $(a_{-\log p}x, a_{-\log q }x)$ satisfies E2 with $T^{\delta^{3/2}}<R^{A_1}\leq T^{\frac{\delta}{100A_2}}$. Then 
\begin{enumerate}
\item[(I1)] $vol(H_{\cdot}(x_0,y_0))=ind(\beta)$;
\item[(I2)] $\Gamma_{p,q}$ is a congruence lattice;
\item[(I3)] $vol(H_{\cdot}(x_0,y_0))\geq \min (p^{1/3},q^{1/3})$;
\item[(I4)] all the two-factor products of the denominators of the matrix $\beta$ 
divide an integer $\leq (\text{\rm ind}(\beta) +1)^3$,  hence they are (in absolute value) smaller than $(\text{\rm ind}(\beta)+1)^3$.
\end{enumerate}
\end{lemma}
 \begin{proof}

Property (I1) just follows from the definition of the index. 

\smallskip
Property (I2)  follows  from (I4). In fact, since $\Gamma_{p,q}=\beta\Gamma \beta^{-1}\cap \Gamma$  
and  by  (I4) the two-factor products  of all denominators of the entries of $\beta$ divide an integer $z \leq (\text{\rm ind} (\beta)+1)^3$, then  $\Gamma(z)\subset \Gamma_{p,q}$.

\smallskip 
Finally, property (I3) also follows from (I4) es explained below. By Lemma \ref{per:stab} it follows that if $\bar{\xi}=(\xi_1\Gamma,\xi_2\Gamma)$ is such that $H_{\cdot}(\xi_1\Gamma,\xi_2\Gamma)$ is periodic, then there exists $\beta\in COMM(\Gamma)$ (note that $COMM(\Gamma)$ is a subgroup so $\gamma_i\beta\in COMM(\Gamma)$)
 \begin{equation}\label{repr1'}
 \xi_2= \xi_1 \beta.\end{equation}
 Moreover, the fact that $vol(H_{\cdot}\bar{\xi})<R^{A_1}$ means that the index of $\beta\Gamma \beta^{-1}\cap \Gamma$ is $\leq R^{4A_1}$. By Proposition \ref{l:modular1} and Corollary \ref{cor:pq} it follows that $(a_{-\log p}x, a_{-\log q }x)$ satisfies \eqref{eq:moducent}. Denote $\bar{x}=(a_{-\log p}x, a_{-\log q }x)$. 

By Corollary \ref{cor:uni}  for $t_0=0$,  there exist $(K_1, K_2) \in \R^2$ such that $\vert K_1\vert$, $\vert K_2\vert
\leq T^{2\delta}$ and 
$$
d_{G\slash \Gamma\times G\slash \Gamma}(\bar{x},(h_{K_1},h_{K_2}) \bar{\xi})< T^{-1+3\delta}.
$$
This means that there exist $\gamma_1,\gamma_2\in \Gamma$ such that $$d_{G}(\bar{x}_1,h_{K_1}\xi_1\gamma_1)\leq T^{-1+3\delta}\,, \qquad d_G(\bar{x}_2,h_{K_2}\xi_2\gamma_2)\leq T^{-1+3\delta}\,.$$ 

We have
$$
 \bar{x}_2\bar{x}_1^{-1}=\bar{x}_2 (\xi_2\gamma_2)^{-1}(\xi_2\gamma_2)
(\xi_1\gamma_1)^{-1}(\xi_1\gamma_1) \bar{x}_1^{-1}.
$$
By applying~\eqref{repr1'} to $\overline{\xi}=(\xi_1\gamma_1\Gamma,\xi_2\gamma_2\Gamma)$, we have 
$\xi_2\gamma_2 = \xi_1 \gamma_1 \beta$, hence by the right invariance of the metric $d_G$ allows us to obtain $w_1,w_2\in G$ with $d_G(w_i,e)\leq T^{-1/2}$ for $i=1,2$ and such that
\begin{equation}\label{period20'}
h_{-K_2}\bar{x}_2 \bar{x}_1^{-1}h_{K_1}=w_1\xi_0\beta\xi_0^{-1}w_2,
\end{equation}
where $\xi_0=\xi_1\gamma_1$. We want to apply this reasoning to the point
$\bar{x}=(a_{-\log p}x,a_{-\log q}x)$. Now, \eqref{period20'} reads as
\begin{equation}\label{eq:pr'}
w_1^{-1}h_{-K_2}a_{\log (p/q)} h_{K_1}w_2^{-1}= \xi_0\beta\xi_0^{-1}
\end{equation}
where $w_i$ and $\beta$ are as above. 
From formula \eqref{period20'}, since the matrix  $a_{\log (p/q)}$
is diagonal with eigenvalues
$ \sqrt{p/q}$ and $\sqrt{q/p}$ and the matrices $h_{K_1}$ and $h_{-K_2}$ are upper triangular and
unipotent,   we derive that  $A= h_{-K_2}a_{\log (p/q)} h_{K_1}$ also has eigenvalues $ \sqrt{p/q}$ and $\sqrt{q/p}$,
hence 
$ |\text{\rm trace} (\beta)-  \frac{p+q}{ \sqrt{pq} }| \leq T^{-1/2}$, which in turn implies
hence
$$
|\text{\rm trace} (\beta)^2-  \frac{(p+q)^2}{pq}| \leq T^{-1/2} \Big( | \frac{p+q}{ \sqrt{pq}} | + T^{-1/2} \Big) \leq T^{-1/3}\,.
$$

This however implies that if $\beta =  \begin{pmatrix} a & b \\ c & d \end{pmatrix}$, then by (I4), the denominator of  
$(a+d)^2$ is $\leq (ind(\beta)^3+1)^3$ and so by the above inequality, $(ind(\beta)^3+1)^3\geq \min(p,q)$. This finishes the proof.

 Finally we prove property (I4). Let $\text{\rm ind}(\beta)=n$. Let $\Gamma'=\beta SL(2,\Z)\beta^{-1}\cap SL(2,\Z)$. Let us consider the upper triangular  unipotent matrices 
 $$
 u_i = \begin{pmatrix}1 & i \\ 0 & 1 \end{pmatrix}  \in SL(2, \Z)\,,  \quad \text{ for all } i=1, \dots, n+1\,.
 $$
Since $SL(2,\Z)/\Gamma'$ has $n$ elements $\{\gamma_1,\ldots, \gamma_n\}$ and we are considering $n+1$ unipotents it follows that there exists $i\in \{1, \dots, n+1\}$  such that $u_i\in \Gamma'$, which is equivalent to the condition that  $\beta u_i \beta^{-1}\in SL(2,\Z)$.  Let then 
$$
\beta =  \begin{pmatrix} a & b \\ c & d \end{pmatrix}  \in SL(2, \R) \,.
$$
A direct computation implies that, 
\begin{equation}
\label{eq:prod_in_Z_1}
i a^2 \,, i c^2   \text{ and }  i ac   \in \Z\,,
\end{equation}
 hence the denominators of $a^2$, $c^2$ and $ac$ divide $i$. 
 By considering unipotents 
 $$
 v_j= \begin{pmatrix}1 & 0 \\ j & 1 \end{pmatrix}  \in SL(2, \Z)\,,  \quad \text{ for all } j=1, \dots, n+1\,,
 $$
 we conclude similarly that there exists $j\in \{1, \dots, n+1\}$  such that 
\begin{equation}
\label{eq:prod_in_Z_2}
j b^2 \,,   jd^2    \text{ and }    j bd     \in   \Z\,,
\end{equation}
 hence the denominators of $b^2$, $d^2$ and $bd$ divide $j$. 
 \smallskip
 For the remaining products we reason analogously on the set of matrices
 $$
 r_k = \begin{pmatrix}k & -1 \\ 1 & 0 \end{pmatrix}  \in SL(2, \Z)\,,  \quad \text{ for all } k=1, \dots, n+1\,.
 $$
 Given \eqref{eq:prod_in_Z_1} and \eqref{eq:prod_in_Z_2}, we conclude that there exists $k \in \in \{1, \dots, n+1\}$  such that 
 $$
 (ij k) ab\,,  (ijk) ad \,,   (ijk) bc  \text{ and }   (ijk) cd  \in \Z\,.
 $$
 hence the denominators of $ab$, $ad$ $bd$ and and $cd$ divide $ijk$.
 
 \end{proof}

\section{SPNT  for horocycle flows in cocompact case - proof of Theorem~\ref{th:glowne}}
Since $\Gamma$ is co-compact it follows that for any point $(x,y)$ the condition E3. in Theorem \ref{LMWdrugie} is not satisfied as the injectivity radius is uniformly bounded below. Let $p,q$ be any primes with $p,q\leq T^{\delta^2}$ and fix $R=T^{\frac{\delta}{1000A_1A_2}}$. We have the following:

 \begin{lemma}\label{lem:npo} Assume $\Gamma$ is co-compact  (and arithmetic). For any $x\in SL(2,\R)\slash \Gamma$,  any $p,q\leq T^{\delta^2}$, $p\neq q$, the point  $(a_{-\log p}x,a_{-\log q}x)$ does not satisfy E2.
 \end{lemma}
 
 Notice that the above lemma immediately implies Theorem \ref{th:glowne}. Indeed from this lemma and the fact that E3 never holds 
 it follows that $(a_{-\log p}x,a_{-\log q}x)$ has to satisfy E1 for all $p,q\leq T^{\delta^2}$. But the by Corollary \ref{cor:adc} it follows that the sequence $b_n=\phi(h_nx)$ satisfies the assumption of the SPNT criterion. So it only remains to prove the above lemma:
 \begin{proof}[Proof of Lemma \ref{lem:npo}] By Lemma \ref{per:stab} it follows that if $\bar{\xi}=(\xi_1\Gamma,\xi_2\Gamma)$ is such that $H_{\cdot}(\xi_1\Gamma,\xi_2\Gamma)$ is periodic, then there exists $\beta \in COMM(\Gamma)$ (note that $COMM(\Gamma)$ is a subgroup so $\gamma_i\beta\in COMM(\Gamma)$) such that
 \begin{equation}
 \xi_2= \xi_1 \beta.\end{equation}
 Moreover, the fact that $vol(H_{\cdot}\bar{\xi})<R^{A_1}$ means that the index of $\beta\Gamma \beta^{-1}\cap \Gamma$ is $\leq R^{4A_1}$. By Proposition \ref{l:modular1} and Corollary \ref{cor:pq} it follows that $(a_{-\log p}x, a_{-\log q }x)$ satisfies \eqref{eq:moducent}. Denote $\bar{x}=(a_{-\log p}x, a_{-\log q }x)$. Applying \eqref{eq:moducent} for $s=0$ (keeping in mind Corollary \ref{cor:pq})
$$
d_{G\slash \Gamma\times G\slash \Gamma}(\bar{x},(h_{K_1},h_{K_2}) \bar{\xi})<R^{4A_1}T^{-\delta/A_2-4\delta^2}.
$$

By Corollary \ref{cor:uni}  for $t_0=0$,  there exist $(K_1, K_2) \in \R^2$ such that $\vert K_1\vert$, $\vert K_2\vert
\leq T^{2\delta}$ and 
$$
d_{G\slash \Gamma\times G\slash \Gamma}(\bar{x},(h_{K_1},h_{K_2}) \bar{\xi})< T^{-1+3\delta}.
$$
This means that there exist $\gamma_1,\gamma_2\in \Gamma$ such that $$d_{G}(\bar{x}_1,h_{K_1}\xi_1\gamma_1)\leq T^{-1+3\delta}\,, \qquad d_G(\bar{x}_2,h_{K_2}\xi_2\gamma_2)\leq T^{-1+3\delta}\,.$$ 
We have
$$
 \bar{x}_2\bar{x}_1^{-1}=\bar{x}_2 (\xi_2\gamma_2)^{-1}(\xi_2\gamma_2)
(\xi_1\gamma_1)^{-1}(\xi_1\gamma_1) \bar{x}_1^{-1}.
$$
By applying~\eqref{repr1'} to $\overline{\xi}=(\xi_1\gamma_1\Gamma,\xi_2\gamma_2\Gamma)$, the right invariance of the metric $d_G$ allows us to obtain $w_1,w_2\in G$ with $d_G(w_i,e)\leq T^{-1/2}$ for $i=1,2$ and such that
\begin{equation}\label{period20''}
h_{-K_2}\bar{x}_2 \bar{x}_1^{-1}h_{K_1}=w_1\xi_0\beta\xi_0^{-1}w_2,
\end{equation}
where $\xi_0=\xi_1\gamma_1$. We want to apply this reasoning to the point
$\bar{x}=(a_{-\log p}x,a_{-\log q}x)$. Now, \eqref{period20''} reads as
\begin{equation}\label{eq:pr''}
w_1^{-1}h_{-K_2}a_{\log (p/q)} h_{K_1}w_2^{-1}= \xi_0\beta\xi_0^{-1}
\end{equation}
where $w_i$ and $\beta$ are as above. Denote $A=h_{-K_2}a_{\log (p/q)} h_{K_1}$. 
From formula \eqref{eq:pr''}, since the matrix $A$ has eigenvalues
$\sqrt{p/q}$ and $\sqrt{q/p}$,  we derive that
$$| \text{\rm trace}(\beta)-  \frac{p+q}{\sqrt{pq} }| = | \text{\rm trace}(\beta)- \text{\rm trace} \big(a_{\log(p/q)} \big)|\leq T^{-1/2}.$$
Let $\beta = \phi(\alpha) / N(\alpha)^{1/2}$ with $\alpha \in A^+(\Q)$. 
Since $\text{\rm trace}({\beta})=2x_0N(\alpha)^{-1/2}$ the above formula reads   $|\frac{2x_0}{N(\alpha)^{1/2}}- \frac{p+q}{(pq)^{1/2}}|\leq T^{-\bar\kappa}$ with $\bar{\kappa}\geq \frac{23\delta}{25A_2}$. In particular, it also implies that 
$$|\frac{4x_0^2}{N(\alpha)}- \frac{{(p+q)^2}}{pq}|\leq T^{-1/4}\,.$$ 
This however, by the bound on $p,q\leq T^{\delta^2}$ and the bound on the denominator of $\frac{4x_0^2}{N(\alpha)}$ (see Lemma \ref{lem:per'})  implies that $\frac{x_0}{N(\alpha)^{1/2}}=\frac{{p+q}}{2(pq)^{1/2}}$. Note that 
$$
\sqrt{p/q}\cdot \sqrt{q/p}=1=\Big(\Big(\frac{x_0}{N(\alpha)^{1/2}}\Big)^2-a\Big(\frac{x_1}{N(\alpha)^{1/2}}\Big)^2-b\Big(\frac{x_2}{N(\alpha)^{1/2}}\Big)^2+ab\Big(\frac{x_3}{N(\alpha)^{1/2}}\Big)^2\Big).
$$
Using the formula {$\frac{x_0}{N(\alpha)^{1/2}}=\frac{p+q}{2(pq)^{1/2}} =  \frac{1}{2}(\sqrt{p/q} +\sqrt{q/p})$, we get that 
$$
\Big(\frac{x_0}{N(\alpha)^{1/2}}\Big)^2 - \sqrt{p/q}\cdot \sqrt{q/p} = \frac{1}{4} (\sqrt{p/q} +\sqrt{q/p})^2  - \sqrt{p/q}\cdot \sqrt{q/p}
= \frac{1}{4} (\sqrt{p/q} -\sqrt{q/p})^2
$$
hence
$$
 \Big( \frac{1}{2}   N(\alpha)^{1/2} (\sqrt{p/q}-\sqrt{q/p}) \Big)^2-ax_1^2-bx_2^2+ab x_3^2=0
$$
}
Note that {since $2x_0/ N(\alpha)^{1/2} = (p+q)/\sqrt{pq}$, we have 
$$N(\alpha)^{1/2} (\sqrt{p/q}-\sqrt{q/p})  = (p-q) N(\alpha)^{1/2} / \sqrt{pq} =   2 x_0 (p-q)/(p+q) \in \Q$$
}
which, since $A(\Q)$ is a division algebra, implies that 
$$
N(\alpha)^{1/2}(\sqrt{p/q}-\sqrt{q/p}) = x_1 = x_2 = x_3 = 0\,,
$$
in particular $p=q$. A contradiction.


\end{proof}

\section{SPNT for horocycle flows (the modular case) - proof of Theorem~\ref{th:glowne'}}

In this section we will use the notation $A_i$ to denote positive constant that depend on $X$ only. Let $x\in X$ be generic for the Haar measure $\mu_X$. We will use Theorem \ref{LMWdrugie} for $R\sim T^\delta$, for some small $\delta>0$ and for sufficiently large $T\geq T_x$. Recall that analogously to the  co-compact case that we are interested in the behavior of the point $(a_{-\log p}x, a_{-\log q }x)$, where $p,q\leq T^{\delta^2}$ and $x\in X$ is generic for Haar measure\footnote{If $x$ is an  $(h_t)$-periodic point then the SPNT for $x$ follows from Vinogradov's theorem (if the period is irrational) and semiprimes in arithmetic progressions (if the period is rational).}.
 Notice that if $x\in X$ is such that E1. holds for $(a_{-\log p}x, a_{-\log q }x)$ and all $p\neq q$, $p,q\leq T^{\delta^2}$ then analogously to the cocompact case we show that SPNT holds, using the semiprime criterion from Section \ref{sec:conspnt}. The analysis below deals with points $x\in X$ (generic for Haar) such that one can find $p\neq q$ with $p,q\leq T^{\delta^2}$ for which $(a_{-\log p}x, a_{-\log q }x)$ does not satisfy E1. The following results describe the behavior of $x\in X$ for which $(a_{-\log p}x, a_{-\log q }x)$ satisfies E2. or E3. \\
 
The proposition below holds for any sub-polynomial function $\psi: \R^+ \to \R^+$, that is any
 function such that for all $\epsilon >0$ we have
 $$
 \lim_{T\to +\infty}  \frac{\psi(T)}{T^{\epsilon} }=0\,,
 $$
 in particular for $\psi(T)=\log T$ and also $\psi(T)=\log \log T$.
 \begin{proposition}\label{prop:E2}  Let $p,q$ be two different primes with $\psi(T)^{1/\eta}<p\neq q<T^{\delta^2}$ and $\frac{1}{5}\leq p/q\leq 5$. Assume $x\in X$ is such that $(a_{-\log p}x, a_{-\log q}x)$ satisfies E2 with $R=T^\delta$. There exists $\eta_0>0$ such that for all $\eta <\eta_0$ at least one of the following holds. 
 \begin{enumerate}
 \item[$E2_1$] there exist a Sobolev norm $S_d$ and a constant $C_d>0$ such that for every $\varphi\in C_c^{\infty}(X\times X)$ with $\mu_{X\times X}(\varphi)=0$, we have
$$
\Big|\frac1T\int_0^T\varphi\circ (a_{\log p}\times a_{\log q})\circ (h_r\times h_r)(a_{-\log p}x,a_{-\log q}x)\,dr\Big|\ll  S_d(\varphi) \frac{C_d}{\psi(T)^{101}}.$$

\item[$E2_2$]  
  there exists a periodic point $w\in X$ with $per(w)<T^{A_3\delta}$  and $t_0\in [0,T]$ such that 
 $$
 d_X( h_{t_0} a_{-\log p}x,w)\leq T^{-1+A_4\delta}.
 $$
 \end{enumerate}
 \end{proposition}
 
 \begin{proposition}\label{prop:E3} Assume $x\in X$   is such that $(a_{-\log p}x, a_{-\log q}x)$ satisfies E3 for some $p,q\leq T^{\delta^2}$. Then there exists a periodic point $w\in X$ with $per(w)<T^{A_4\delta}$ and $t_0\in [0,T]$ and such that 
 $$
 d_X(h_{t_0}x,w)\leq T^{-1+A_4\delta}.
 $$
 \end{proposition}
 
 Finally we have the following result describing points which are closed to a periodic point:
 \begin{theorem}\label{prop:peri} Assume $x\in X$ is such that  there exists a $t_0\in [0,T^{1+\delta^2}]$ and a periodic point $w\in X$ with $per(w)<T^{A_4\delta}$  such that $d_X(h_{t_0}x,w)\leq T^{-1+A_4\delta}$. Then for $T\geq T_x$,
 $$
 |\sum_{p_1\cdot p_2\leq T} f(h_{p_1\cdot p_2}x)- \pi_2(T) \int_X f d\mu_X|=o(\pi_2(T)).
 $$
 \end{theorem} 
 
 Let us  first show how the above propositions imply Theorem \ref{th:glowne'}. We will then prove the propositions in separate subsections. 
 \begin{proof}[Proof of Theorem \ref{th:glowne'}] Take  $x\in X$ generic for the Haar measure.  If for all $p,q\leq T^{\delta^2}$ at least one of the following holds: (c1).  $(a_{-\log p}x,a_{-\log q}x)$ satisfies E1. or (c2). $(a_{-\log p}x,a_{-\log q}x)$ satisfies E2 and also $E2_1$, then analogously to the co-compact case we use the criterion in Section \ref{sec:conspnt}. Therefore we are left with the case in which there are $p,q\leq T^{\delta^2}$ such that $(a_{-\log p}x,a_{-\log q}x)$ satisfies $E2_2$ or $E3$. If $E2_2$ holds then 
 $$
   d_X( h_{pt_0} x,a_{\log p}w)\ll p^2 d_X( a_{-\log p}h_{pt_0}x,w)=p^2 d_X(h_{t_0}a_{-\log p}x,w)\leq T^{-1+A_4\delta+2\delta^2}.
 $$
 It remains to notice that $pt_0\leq T^{1+\delta^2}$ and $a_{\log p} w$ is a periodic point of period $\leq T^{A_3\delta+\delta^2}$. We can then apply Theorem \ref{prop:peri}. If E3 holds then by Proposition \ref{prop:E3} we can apply Theorem \ref{prop:peri} directly to get that Theorem \ref{th:glowne'} holds.  
  \end{proof}

 So it remains to prove the above results. We will do it in separate subsections below. 
 
 \subsection{Proposition \ref{prop:E2}} 
Let $\Gamma< SL(2, \R)$ be any lattice and let $M_\Gamma = M_{\rm thin} \cup M_{\rm thick}$ denote Margulis thin-thick decomposition of $M_\Gamma = SO(2, \R) \backslash SL(2, \R) / \Gamma  $. We recall that the Margulis decomposition is
 defined as follows. Let $\epsilon_0 >0$ be any fixed number strictly less than the Margulis constant of the Poincar\'e plane
 (which is a universal number). Let $\rho_\Gamma: M_\Gamma \to \R^+$ denote the injectivity radius function. Then
 $$
 M_{\rm thin} :=\{x \in M_\Gamma \vert  \rho(x) <\epsilon_0\} \quad \text{ and } \quad M_{\rm thick} :=\{x \in M_\Gamma \vert  \rho(x) \geq \epsilon_0\} \,.
 $$
 Since $M_\Gamma$ has finite volume, the thin part $M_{\rm thin}$ is a union of cusps (unbounded components)  and  {\it Margulis
 tubes} (boundaries of closed geodesics of length less than $\epsilon_0$. 
 
 \begin{definition}  The cuspidal part $M_{\rm cusp}$ of $M_\Gamma$ is defined as the subset of the thin part $M_{\rm thin}$
 which is a finite union of cusps. The compact part $M_{\rm cpt}$ of $M_\Gamma$ is defined as the union of subset of the thick 
 part $M_{\rm thick}$ with all Margulis tubes. By definition we have a decomposition
 $$
 M_\Gamma = M_{\rm cpt}  \cup  M_{\rm cusp} \,.
 $$
  \end{definition} 
 \noindent Let $\mu_\Gamma>0$ denote the bottom of the spectrum of the (positive) Laplace operator for the hyperbolic metric on $M$ on the orthogonal complement of constant functions and let $\nu_\Gamma\in (0,1)$ denote the number
 $$
 \nu_{\Gamma} := \text{Re}  \sqrt {1-4 \mu_\Gamma} \,.
 $$
 Let $\text{\rm inj}_\Gamma$ denote the injectivity radius of the compact  part $M_{\rm cpt}$
 and let $d_\Gamma: S_\Gamma \to \R^+$ denote the hyperbolic distance function on  $S_\Gamma= 
 SL(2, \R)/\Gamma$ from the closed subset $S_\Gamma \vert M_{\rm cpt}$.  
 We then have (see \cite{FlaFor}, Theorem 5.14, and \cite{Strom13}, Theorem 1):
 
 \begin{theorem}  
 \label{thm:SL2_equi}
 For every $s\geq 4$ there exists a constant  $C_s>0$ such that for every function  $f \in W^s(S_\Gamma)$ 
 and for all $(x,T) \in S_\Gamma \times [1, +\infty)$, 
 $$
\Big \vert \frac{1}{T} \int_0^T   f \circ h_t (x) dt  -  \int_{S_\Gamma}  f d\text{\rm vol}_\Gamma \Big\vert  \leq  
C_s \Vert f  \Vert_{W^s(S_\Gamma)}   \max\{\text{\rm inj}_\Gamma^{-1}, e^{ \frac{(1-\nu_\Gamma)}{2}d_\Gamma(a_{\log T} (x))}\}    T^{\frac{-(1- \nu_\Gamma)}{2}} \,.
 $$
 \end{theorem} 
 In the above Theorem \ref{thm:SL2_equi} the volume $d\text{\rm vol}_\Gamma$ is normalized, while the Sobolev spaces
 are defined with respect to the constant curvature metric (whose volume is not normalized).

We proceed to the proof of Theorem \ref{thm:SL2_equi}. 
 \begin{lemma} 
 \label{lem:per}
 (\cite{Stre}, Lemma 1.3)
 Let $x \in S_\Gamma$  and $T >0$. Let $\eta>0$ and $1\leq K \leq T$.
  There is an interval $I_0 \subset {[}0, T{]}$  of size $\vert I_0\vert \leq \eta^{-1} K^2$ such that
for all $s_0 \in {[}0, T{]} \setminus I_0$ there is a segment $\{ h_s (\xi) \vert 0 \leq s \leq K\}$  
  of a closed horocycle approximating $\{  h_{s_0+s} (x) \vert 0 \leq t \leq K\}$ of order in the
sense that
  $$
  d_S  (h_{s_0+s} (x) , h_s(\xi) )  \leq \eta\,, \quad \text{ for all }  0 \leq s \leq K \,.
  $$
There exists $C>1$ such that period $P:= P(s_0,x)$  of this closed horocycle is at most  
$$
 P \leq C T  \exp \Big(-d_\Gamma( a_{-\log T}(x)) \Big ) \,.
$$  
Moreover, one can assure $P \geq C^{-1} \zeta^2 T  \exp \Big( - d_\Gamma( a_{-\log T}(x) )\Big )$  for some $\zeta$ by weakening the
bound on $I_0$ to the bound $\vert I_0\vert \leq \max \{\eta^{-1} K^2, \zeta T\}$.
\end{lemma} 
\begin{proof} The proof in \cite{Stre}, Lemma 1.3, is given  for $\Gamma= SL(2, \Z)$ and the argument can be applied
without modifications to the cusps. In the general case we may proceed as follows. Let $x \in S_\Gamma$ and let 
$$
t_1:=  \max \{t\geq 0 \vert   a_{-t} x \in S_\Gamma \vert M_{thick} \}\,.
$$
Let $x' = a_{-t_1} x$ denote the point at the boundary of the thick
part. Let $t_2>0$ denote the time spent by the orbit in a cusp. By the result of  \cite{Stre}, Lemma 1.3,  given $\eta>0$
and $K\leq  e^{t}$ there exists an interval $I_0' \subset [0,e^{t-t_1}]$ of length $\leq  (\eta e^{-t_1})^{-1} (K e^{-t_1})^2$  and a
periodic point $\xi'$ such that for some $r_0 \in [0,1]$
$$
d(h_{r_0 +r} x', h_r \xi')  \leq   e^{-t_1} \eta \,,\quad \text{ for all } r \in [0,K e^{-t_1}] \setminus I'_0\,,
$$
and there exists $C>1$ such that the period $P'$ of $\xi'$ satisfies
$$
P'=1  \leq  e^{t_2} \exp (-d_\Gamma (g_{t_2} x') ) \,.
$$
Let $\xi = a_{-t_1} \xi'$ and let $t= t_1+t_2$. By definition we have $d_\Gamma (a_{-t_2} x) )= d_\Gamma (a_{-t} x)$ and 
the period $P$ of $\xi$ is at most
$$
P \leq e^{t_1} = e^{t_1+t_2} e^{-t_2}  \leq  e^t \exp (-d_\Gamma (a_{-t_2} x') )  =  e^t \exp (-d_\Gamma (a_{-t} x) ) \,.
$$
In addition we have that 
$$
\begin{aligned}
d(h_{s_0 +s} x, h_s \xi) &= d(h_{s_0 +s} a_{t_1} x' , h_s a_{t_1} \xi') = d (a_{t_1} h_{e^{-t_1} (s_0 +s)} x' , a_{t_1} h_{e^{-t_1}s}  \xi')
\\ &\leq e^{t_1}  d (h_{e^{-t_1} (s_0 +s)} x' ,  h_{e^{-t_1}s}  \xi')    \leq   e^{t_1}  e^{-t_1} \eta    =\eta\,,
\end{aligned} 
$$
hence $d(h_{s_0 +s} x, h_s \xi)$ with $s_0 = e^{t_1} r_0$ for all $s \in [0, K] \setminus e^{t_1} I'_0$,  and the interval
$I_0 = e^{t-1} I'_0$ has length $\leq e^{t_1} \vert I_0' \vert \leq \eta K^2$.

\end{proof} 

We state below an equidistribution result which can be derived from M.~Einsiedler, G.~Margulis and A.~Venkatesh \cite{EMV} and 
.M.~Einsiedler, G.~Margulis, A. Mohammadi and and A.~Venkatesh \cite{EMMV}.

Let $G$ be  a semisimple $\Q$-group so that $G= G(\R)$ and $\Gamma$ is a congruence subgroup of 
$G(\Q)$;  let $H\subset G$ be a subgroup such that  $H^+=H$, i.e. $H$ has no compact factors and is simply connected
and such that the centralizer of $\mathfrak h = \text{Lie} (H)$ in $\mathfrak g = \text{Lie} (G)$ is trivial.

Below we will apply the theorem with $G= SL(2, \R)^2$, $\Gamma= SL(2, \Z)^2$ and $H=SL(2, \R)$ embedded
diagonally in $G$. Notice that $H$ is a maximal subgroup of $G$. The theorem below is a special case of the more general Theorem 1.5. in \cite{EMMV}. Since in  our case $G$ is simply connected, it follows that $\pi^f(x)=\int_{X\times X}f d\mu_{X\times X}$. Moreover, the set $Y_{\mathscr{D}}$ is a maximal algebraic semisimple homogeneous set as the diagonally embedded subgroup  $H$ is maximal in $G$ ($\iota$ is the diagonal embedding).\footnote{The authors would like to thank M. Einsiedler for his feedback on Theorem 1.5. in \cite{EMMV}.}
\begin{theorem} [\cite{EMMV}, Theorem 1.5, see also \cite{EMV}, Theorem 1.3]
\label{thm:EMMV}
Let $\Gamma$, $H \subset G$  be as above. For any $g\in G$, let $H_g:=  gHg^{-1}$ and let $\mu_g$ be the $H_g$-invariant probability measure on a closed $H_g$-orbit $H_g  . g(x_0,y_0)$ inside $X= \Gamma \backslash G$.  There exists $\sigma$, $d>0$ and a constant $C_d>0$  (depending only on $G$, $H$) such that  $\mu_g$ is $V^{-\sigma}$-close to   $\mu_{X\times X}$, i.e. for any $f \in C^\infty(X)$ we have  
$$
\Big\vert \int_{X\times X}  f d\mu_g - \int_{X\times X}  f d\mu_{X\times X} \Big\vert  < C_d \text{vol}(H_g . g(x_0,y_0))^{-\sigma} S_d(f) 
$$
where $S_d(f)$ denotes an $L^2$-Sobolev norm of degree $d$.
\end{theorem}
In our setting we will apply it for elements $g=Id\times h_L$. In this case it follows that $vol(H_g . g(x_0,y_0))\geq vol(H_\cdot(x_0,y_0))$ and in fact we have a polynomial gain in $|L|$ (which is especially powerful for large $|L|$).

We are now ready to give the proof of Proposition \ref{prop:E2}.

\begin{proof} [Proof Proposition \ref{prop:E2}]
Let $(x_0, y_0)\in X\times X$ and $H$ be such that  that $(a_{-\log p}x, a_{-\log q }x)$ satisfies E2 with 
$R=T^\delta$. By  assumption $H.  (x_0, y_0)$ is a closed submanifold of $X\times X$ hence there exists a lattice
$\Gamma_{p,q}$ such that $H. (x_0, y_0)$ is isomorphic to the quotient $S_{p,q}:=SL(2, \R)/ \Gamma_{p,q}$. 
By (I3) in Lemma \ref{lem:alpha}, we have 
$$
\psi(T)^{1/(3\eta)} \leq \min\{p^{1/3}, q^{1/3}\} \leq \text{vol} (\Gamma_{p,q}) = \text{vol} ( H. (x_0, y_0))  \leq R' \leq T^{A_1\delta}\,.
$$
We also have that 
$$
 \nu_{\Gamma_{p,q}}  \leq  1-\rho    \quad \text{and} \quad   \text{\rm inj}_{\Gamma_{p,q} } \geq   \rho  \,.
$$  
Indeed the upper bound on $ \nu_{\Gamma_{p,q}} $ follows from (I2) in Lemma \ref{lem:alpha} and Selberg's bound on the spectral gap for congruence lattices, \cite{Selberg}.  Moreover  since $\Gamma_{p,q}  < SL(2, \Z)$ and it has finite index, the quotient $SL(2, \R)/\Gamma_{p,q}$ is a finite cover of the modular quotient $SL(2, \R)/SL(2,\Z)$.  It follows than any Margulis tube 
in $SL(2, \R)/\Gamma_{p,q}$ projects (by a locally isometric map) to a Margulis tube in $SL(2, \R)/SL(2,\Z)$, hence, for all 
$(p, q)\in \Z \times \N\setminus\{0\}$, we have  
$$
\text{\rm inj}_{\Gamma_{p,q} }  \geq  \inj_{SL(2, \Z)}\,.
$$

Let $\tilde{T}>0$ and assume that  for some $(u,v)\in H.(x_0,y_0)$, 
$$
d_{H.(x_0,y_0)} \Big( (a_{\log \tilde{T}} \times a_{\log \tilde{T}})(u, v) \Big)  \leq    (1-A\delta) \log \tilde{T}\,.
$$
By Theorem \ref{thm:SL2_equi}  for every  $f \in W^s( H.(x_0,y_0) )$,
 $$
\Big \vert  \frac{1}{
\tilde{T}} \int_0^{\tilde{T}}   f \circ (h_t \times h_t) (u, v) dt  -   \int_{H.(x_0,y_0)}  f d\text{\rm vol}_{x_0,y_0} \Big\vert  \leq  
C_s(\rho)  \Vert f  \Vert_{W^s(H.(x_0,y_0))}   \tilde{T}^{-A \delta  \frac{1-\rho}{2}} \,.
 $$
In the above formula the measure $d \text{vol}_{x_0,y_0}$ is the normalized volume. We note that if $\varphi$ is the restriction to 
$H.(x_0,y_0)$ of the function $\varphi \circ (a_{\log p} \times a_{\log q})$ with $\varphi \in C^{\infty} _c(X\times X)$, then taking into account that by assumption $p, q < T^{\delta^2}$,  we have
$$
\begin{aligned}\label{all.pq}
&\Big \vert  \frac{1}{\tilde{T}} \int_0^{\tilde{T}}   \varphi \circ (a_{\log p} \times a_{\log q})  \circ (h_t \times h_t) (u, v) dt  -   \int_{H.(x_0,y_0)}  \varphi \circ (a_{\log p} \times a_{\log q}) d\text{\rm vol}_{x_0,y_0} \Big\vert \\ & \quad  \leq  
C_s(\rho)  \Vert \varphi \circ  (a_{\log p} \times a_{\log q})\Vert_{C^s(X\times X)} \text{vol}(H.(x_0, y_0))^{1/2}    \tilde{T}^{-A \delta  \frac{1-\rho}{2}} \\
& \quad  \quad \leq C_s(\rho)  \Vert \varphi \Vert_{C^s(X\times X)}   \tilde{T}^{s \delta^2}  \tilde{T}^{-  (A   (1-\rho) -A_1) \delta/2}  \,.
\end{aligned}
$$

If additionally $\tilde{T}=T^{1-\delta}$, then since $\text{vol} (H.(x_0,y_0)) = \text{vol} (\Gamma_{p,q})  \geq \log^{1/(3\eta)} T$, it follows from Theorem~\ref{thm:EMMV} that, for $T$ sufficiently large, we have
\begin{equation}
\label{eq:EMV}
\begin{aligned}
\Big \vert   \int_{H.(x_0,y_0)}  &\varphi \circ  (a_{\log p} \times a_{\log q})   d\text{\rm vol}_{x_0,y_0}   \\ &-  \int_{X\times X} \varphi \circ (a_{\log p} \times a_{\log q})  d \mu_{X\times X} \Big \vert 
\leq  S_d(\varphi)\psi(T)^{- \sigma/(3\eta)}  \,.
\end{aligned}
\end{equation}
 In addition, since the measures $\text{\rm vol}_{x_0,y_0}$ and $\mu_{X\times X}$ are invariant under the diagonal geodesic flow $\{a_t \times a_t\}$ it follows that 
$$
\begin{aligned}
\int_{H.(x_0,y_0)}  \varphi \circ  (a_{\log p} \times a_{\log q})   d\text{\rm vol}_{x_0,y_0}  -  \int_{X\times X} \varphi \circ (a_{\log p} \times a_{\log q})  d \mu_{X\times X} \\
=\int_{H.(x_0,y_0)}  \varphi \circ  (a_{\log (p/q)} \times \text{\rm Id})   d\text{\rm vol}_{x_0,y_0}  -  \int_{X\times X} \varphi  \circ(a_{\log (p/q)} \times \text{\rm Id}) d \mu_{X\times X} \,,
\end{aligned}
$$
thus Theorem \ref{thm:EMMV}  can be applied to the function $ \varphi \circ  (a_{\log (p/q)} \times \text{\rm Id})$ and thanks
to the hypothesis that $1/5 \leq p/q \leq 5$, there exists a constant $C_d >0$ such that
$$
S_d(\varphi \circ   (a_{\log (p/q)} \times \text{\rm Id})  )  \leq C_d S_d(\varphi) \,.
$$

We have thus concluded that in this case, there exist a Sobolev norm $S_d$ and a constant $C_d>0$ such that, for $\delta>0$ sufficiently small,  
$$
\begin{aligned}
\Big \vert  \frac{1}{\tilde{T}} \int_0^{\tilde{T}}   &\varphi  \circ  (a_{\log p} \times a_{\log q})  \circ (h_t \times h_t) (u, v) dt   \\ &-  \int_{X\times X} \varphi  \circ  (a_{\log p} \times a_{\log q}) d \mu_{X\times X} \Big \vert 
\leq  C_d S_d(\varphi)  \psi(T)^{- \sigma/(3\eta)}\,.
\end{aligned}
$$
It remains to estimate the deviation of the ergodic average for the orbit of $(a_{-\log p} x, a_{-\log q} x)$. By Proposition \ref{l:modular1} and Corollary \ref{cor:pq} it follows that for $(u_i,v_i)=(h_{iT^{1-\delta}}\times h_{iT^{1-\delta}})(u,v)\in H.(x_0,y_0)$ we have that 
$$
d_{X\times X}\Big((h_s\times h_s)(a_{-\log p}x,a_{-\log p}x), (h_{K_i}\times h_{K'_i})(h_{s-iT^{1-\delta}}\times h_{s-iT^{1-\delta}})(u_i,v_i)\Big)\leq T^{-\eta},
$$
for every $s\in I_i=[iT^{1-\delta},(i+1)T^{1-\delta}]$ and $i\leq [T^\delta]$. Using this we get (denoting $\varphi_i=\varphi \circ  (a_{\log p} \times a_{\log q})  \circ (h_{K_i} \times h_{K'_i})$)
\begin{equation}
\label{eq:int_approx}
\begin{aligned}
\Big\vert  \int_0^{T} &\varphi \circ  (a_{\log p} \times a_{\log q})  \circ (h_{t}  \times h_{t}) (a_{-\log p} x, a_{-\log q} x)dt 
\\ &- \sum_{i=0}^{[T^{\delta}]}\int_0^{T^{1-\delta} } \varphi_{i} \circ (h_s\times h_s) ( u_i,v_i)  ds  \Big\vert  \leq  \Vert \varphi\Vert_0 T^{1-\eta} \,.
\end{aligned}
\end{equation}
Let $A>A_1/(1-\rho)$. Assume first that  there exists $i\in \{0, \dots, [T^{\delta}]\}$ such that
$$
d_{H.(x_0,y_0)} \Big( (a_{\log T^{1-\delta}} )(u_i, v_i) \Big)  \geq    (1-A\delta)(1-\delta) \log T\,.
$$

In  this case by Lemma~\ref{lem:per}  with $K=1$ and $\eta = T^{-1+ A \delta}$ there exists $(x',y') \in H. (x_0, y_0)$ 
such that  $(x', y')$ is periodic of period $P \leq  C T^{A \delta}$ and $\tau_0 \in [0,   T^{1-A \delta}]$ such that
$$
d_{H.(x_0,y_0)} \Big(  h_{\tau_0} \times h_{\tau_0} (u_i, v_i), (x',y') \Big) \leq  T^{-1+ A \delta}\,.
$$
It follows that $x' \in X$ is a periodic point of period $P\leq  C T^{A \delta}$  such that 
$$
d _X\Big(  h_{\tau_0}   u_i  , x'  \Big) \leq  T^{-1+ A \delta}, 
$$
which implies that there exists $t_0\leq T$ such that 
$$
d _X\Big(  h_{t_0}   u  , x'  \Big) \leq  T^{-1+ A \delta}.
$$
However Corollary \ref{cor:uni} and the bound on $K(t_0)$ then imply that 
$$
d _X\Big(  h_{t_0}   x , \bar{x}'  \Big) \leq  T^{-1+ A \delta+3\delta},
$$
where $\bar{x}'=h_{-K(t_0)}x'$ is periodic of the same period as $x'$. This finishes the proof in this case.

On the other hand, if for all  $i\in \{0, \dots, [T^{\delta}]\}$ 
$$
d_{H.(x_0,y_0)} \Big( (a_{\log T^{1-\delta}} )(u_i, v_i) \Big)  \leq    (1-A\delta)(1-\delta) \log T\
$$
we proceed as follows. Notice first that by Theorem \ref{thm:SL2_equi} we have
$$
\begin{aligned}
\Big \vert  \frac{1}{T^{1-\delta}} \int_0^{T^{1-\delta}}    \varphi_{i} \circ (h_t\times h_t)(u_i, v_i) dt  -   \int_{H.(x_0,y_0)}  \varphi_{i}d \text{\rm vol}_{x_0,y_0} \Big\vert \\ \leq
 C_s(\rho)  \Vert \varphi  \Vert_{C^s(X\times X)}    T^{-  [ (A(1-\delta)-8)   (1-\rho) -A_1 ] \frac{\delta}{2} + s\delta^2}  \,.
\end{aligned}
$$
So we only need to estimate $\int_{H.(x_0,y_0)}  \varphi_{i}d \text{\rm vol}_{x_0,y_0}$; which by the definition of $\varphi_i$ is equal to
$$
\begin{aligned} 
\int_{H.(x_0,y_0)} \varphi\circ(id \times a_{\log p/q})\circ (id\times h_{q(K_i'-K_i)})(a_{\log q} \times a_{\log q})  \circ (h_{K_i} \times h_{K_i})   d \text{\rm vol}_{x_0,y_0}\\ =\int_{H.(x_0,y_0)} \varphi\circ(id,a_{\log p/q})\circ (id\times h_{q(K_i'-K_i)})d \text{\rm vol}_{x_0,y_0},
\end{aligned}
$$
where in the last equality we use that  $H.(x_0, y_0)$ is invariant under the diagonal subgroup. By Theorem \ref{thm:EMMV} it follows that (for some constant $C'_d>0$), 
$$
\begin{aligned}
\Big\vert \int_{H.(x_0,y_0)} \varphi\circ(id,a_{\log p/q})\circ (id\times h_{q(K_i'-K_i)})d \text{\rm vol}_{x_0,y_0}-\int_{X\times X} \varphi  d \mu_{X\times X} \Big\vert  \\ \ll C_d S_d\big (\varphi\circ (id\times a_{\log_{p/q}}) \big) vol(H_{\cdot}x_0,y_0)^{-\sigma} \leq C'_d S_d(\varphi)\psi(T)^{-\sigma/3\eta.}
\end{aligned}
$$
So finally we derive that (for some constant $C''_d>0$)
$$
\Big \vert   \int_0^{T^{1-\delta}}   \varphi_{i} \circ h_t\times h_t (u_i, v_i) dt  - T^{1-\delta}   \int_{X\times X} \varphi  d \mu_{X\times X} \Big \vert 
\leq  C''_d S_d(\varphi) T^{1-\delta} \psi(T)^{-\sigma/(3\eta) }\,.
$$
By the approximation estimate in formula \eqref{eq:int_approx} we conclude that 
$$
\begin{aligned}
\Big \vert  \frac{1}{T}  \int_0^{T}   \varphi \circ (a_{\log p} \times a_{\log q})  \circ (h_t \times h_t) &(a_{-\log p} x, a_{-\log q} x)    dt  
\\ & -   \int_{X\times X} \varphi d \mu_{X\times X} \Big \vert 
\leq  2C''_dS_d(\varphi)    \psi(T)^{-\sigma/(3\eta) }\,.
\end{aligned}
$$
The argument is therefore complete. This finishes the proof.

\end{proof}

 \subsection{Proposition \ref{prop:E3}}
Note that E3 for the point $(a_{-\log p} x,a_{-\log q }x)$ implies that the injectivity radius of $a_{-\log T}x$ is at most $R'^2 e^{-T}$. It is then enough to use  Lemma \ref{lem:per}, applied to the lattice $SL(2,\Z)$.

 \section{Proof of Theorem \ref{prop:peri}}
One of the main tools in proving Theorem \ref{prop:peri}  is the following approximation of a point $x\in X$ by a union of periodic orbits:
\begin{lemma}\label{l:modular3} There exists $A_5>4A_4$ such that the followoing holds:
Assume $x\in X$ satisfies $d_{X}(x,w)<T^{-1+2A_4\delta}$ for some $w\in X$ periodic with period ${\rm per}(w)<T^{2A_4\delta}$.
Let $\inf_{\gamma\in SL(2,\Z) }xw^{-1}\gamma^{-1}=\left[\begin{array}{cc}a&b\\c&d\end{array}\right]$.
Then, there exist some periodic $w_i$, $ i\in [-T^{A_5\delta},T^{A_5\delta}]\cap \Z$, with period $<T^{A_5\delta}$ and disjoint intervals $J_1,K,J_2$ such that $[-T,T]=J_1\cup K\cup J_2$, $|K|=O(T^{1-\delta})$ and
$$
d_{X}\Big(h_tx,h_{\frac{(a+ct_i)^2at}{a+ct}}w_i\Big)<\frac1{\log T}$$
for each $t\in I_i \cap J_s$, where $I_i=[iT^{1-A_5\delta},(i+1)T^{1-A_5\delta})$ and any $t_i\in I_i$ for all $i$.
\end{lemma}
\begin{proof} We will give the proof for positive $t\in [0,T]$, the part for negative $t\in [-T,0]$ follows the same lines. Using the right invariance of the metric $d_X$ and replacing $w$  by $w\gamma$ (for some $\gamma\in SL_2(\Z)$) if needed, we have
\begin{equation}\label{peq0}
\max(|a-1|,|d-1|, |b|,|c|)<\frac2{T^{1-2A_4\delta}}.\end{equation}

For simplicity denote $A(t)=a+ct$. Then,  set $K:=\{0\leq t\leq T:\:|A(t)|<T^{-\delta}\}$. Clearly, $K$ is an interval, and $[0,T]=J_1\cup K\cup J_2$ for some other disjoint intervals $J_1,J_2$. If $K\neq\emptyset$ then  $c<0$ (since $a$ is close to 1). Moreover, the initial point $t_0$ of $K$ satisfy $a+ct_0=T^{-\delta}$, so $t_0=\frac1c(T^{-\delta}-a)$ and also $t_0\leq T$. Whence $|c|>\frac1{2T}$. Furthermore, $|K|$ is at most $2T^{-\delta}|c|^{-1}$, so finally, $|K|=O(T^{1-\delta})$.

Observe that if $0<t,t'<T$ and $|t-t'|<T^{1-5A_4\delta}$ then
$$
|A(t)-A(t')|=c|t-t'|<2T^{-3A_4\delta}.$$
Therefore, $$
|A({t})^2-A({t'})^2|<8\cdot T^{-A_4\delta}$$
because (in view of~\eqref{peq0}), $|A(t)|\leq\frac2{T^{1-2A_4\delta}}T=2T^{2A_4\delta}$.

We write
\begin{equation}\label{peq1}
h_tx=(h_txw^{-1}h_{-t})(h_tw)\end{equation}
and denote $M(t)=-ct^2+t(d-a)+b$.
Then (tacitly assuming that $A(t)\neq0$)
$$
h_txw^{-1}h_{-t}=\left[\begin{array}{cc}1&t\\0&1\end{array}
\right]\left[\begin{array}{cc}a&b\\c&d\end{array}\right]
\left[\begin{array}{cc}1&-t\\0&1\end{array}\right]=$$
$$
\left[\begin{array}{cc}A(t)&M(t)\\c&d-ct\end{array}\right]=
\left[\begin{array}{cc}1&0\\\frac{c}{A(t)}&1\end{array}\right]
\left[\begin{array}{cc}A(t)&M(t)\\0&d-ct-\frac{cM(t)}{A(t)}\end{array}\right]=$$
$$
\left[\begin{array}{cc}1&0\\cA(t)^{-1}&1\end{array}\right]
\left[\begin{array}{cc}1&M(t)A(t)\\0&1\end{array}\right]
\left[\begin{array}{cc}A(t)&0\\0&A(t)^{-1}\end{array}\right]$$
(since $d-ct-\frac{cM(t)}{A(t)}=A(t)^{-1}$).

Now, let $A_5:=10A_4$, fix $1\leq i\leq T^{A_5\delta}$ and let $t\in I_i\cap K$, i.e.
\begin{equation}\label{przyp1}
|A(t)|>\frac1{T^{A_4\delta}}.
\end{equation}

Note that by the renormalization property,
$$
h_{M(t)A(t)}a_{2\log A(t)}(h_tw)=h_{M(t)A(t)+A(t)^2t}(a_{2\log A(t)}w).$$  By~\eqref{przyp1}  and~\eqref{peq0}, we obtain that $|cA(t)^{-1}|<2T^{-1+2A_4\delta}$. 

Returning to \eqref{peq1} and using that $d_G(\begin{pmatrix}1&0\\cA(t)^{-1}&1\end{pmatrix},e)\ll |cA(t)|^{-1}$ 
,  it follows that
$$d_{X}(h_tx,h_{M(t)A(t)+A(t)^2t}(a_{2\log A(t)}w))\ll 2T^{-1+2A_4\delta}$$
(note that $a_sw$ is also periodic with ${\rm per}(a_sw)=e^s{\rm per}(w)$).  Let $t_i\in I_i\cap K$. We set
$$
\widetilde{w}_i:=a_{2\log A({t_i})}w.$$
 Note that
$$
{\rm per}(\widetilde{w}_i)=A({t_i})^2{\rm per}(w)\leq 4T^{2A_4\delta}\cdot T^{A_4\delta}=4T^{3A_4\delta}.$$
We have,
$$
h_{M_tA_t+A_t^2t}a_{2\log A_t}w=a_{2\log A_t}(h_{\frac{M_t}{A_t}+t}w).$$

Moreover, as $t,t_i\in I_i\cap K$, $|\frac{A(t)}{A(t_i)}-1|\leq T^\delta c |t-t_i|\ll T^{(A_4+1-A_5)\delta}$. Therefore,
and by right-invariance
$$
d_{SL(2,\R)}(a_{2\log A(t)}(h_{\frac{M(t)}{A(t)}+t}w), a_{2\log A({t_i})}(h_{\frac{M(t)}{A(t}+t}w))=d_{SL(2,\R)}(a_{2\log (A(t)/A({t_i}))},e) \ll T^{A_4+1-A_5)\delta}.
$$
But (again by renormalization)
$$
a_{2\log A_{t_i}}(h_{\frac{M_t}{A_t}+t}w)=
h_{\frac{M_t}{A_t}A_{t_i}^2+tA^2_{t_i}}\widetilde{w}_i.
$$

Putting the above estimates together we get 
\begin{equation}\label{eq:mart}
d_X(h_tx, h_{\frac{M(t)A(t_i)^2}{A(t)}+tA(t_i)^2}\widetilde{w}_i)\ll T^{A_4+1-A_5)\delta}
\end{equation}

Note that by the definition of $M(t)$ and $A(t)$, $A(t_i)^2\frac{M(t)+tA(t)}{A(t)}=A(t_i)^2\frac{b+dt}{A(t)}$. Moreover, 

$$
\frac{b+dt}{A(t)}-\frac{b}{A(t_i)}-\frac{at}{A(t)}-\frac{(d-a)t_i}{A(t_i)}= 
$$$$
b\Big(\frac{1}{A(t)}-\frac{1}{A(t_i)}\Big)+ \Big(\frac{(d-a)t}{A(t)}-\frac{(d-a)t_i}{A(t)}\Big)+ \Big(\frac{(d-a)t_i}{A(t)}-\frac{(d-a)t_i}{A(t_i)}\Big).
$$
Note that since $t,t_i\in K$ it follows that $|\frac{1}{A(t)}- \frac{1}{A(t_i)}|<2T^\delta$ and by the bound on $b$ (see \eqref{peq0}) it follows that the first term above is (in absolute value) $\ll T^{-1/2}$. Similarly and using additionally that $|t-t_i|<T^{1-A_5\delta}$, we get that the second term is (in absolute value) $\ll T^\delta \cdot 2T^{-1+A_4\delta}\cdot T^{1-A_5\delta}\leq 2 T^{(A_4+1-A_5)\delta}$. Finally the third term is, using \eqref{peq0}, $t,t_i\in K\cap I_i$ and $A(t)-A(t_i)=c(t-t_i)$, is $\ll T^{(2A_4+2-A_5)\delta}$. Using also that $A(t_i)\leq 2T^{A_4\delta}$. We get that 
$$
\Big|A(t_i)^2\frac{b+dt}{A(t)}-A(t_i)^2\frac{at}{A(t)}-A(t_i)^2\frac{b}{A(t_i)}-A(t_i)^2\frac{(d-a)t_i}{A(t_i)}\Big|\ll T^{(4A_4+2-A_5)\delta}\leq T^{-\delta}
$$
as $A_5=10A_4$. Note that term $A(t_i)^2\frac{b}{A(t_i)}+A(t_i)^2\frac{(d-a)t_i}{A(t_i)}$ does not depend on $t$. Consequently so defining $\bar{w}_i=h_{A(t_i)^2\frac{b}{A(t_i)}+A(t_i)^2\frac{(d-a)t_i}{A(t_i)}}\widetilde{w}_i$ (which is also periodic of the same period as $\widetilde{w}_i$, we get that by \eqref{eq:mart},
$$
d_X(h_tx, h_{ A(t_i)^2\frac{at}{A(t)}}\bar{w}_i)\ll T^{-\delta/3}.
$$
This finishes the proof.
\end{proof}

To prove Theorem \ref{prop:peri}, we will show that there exists a global constant $\bar{C}>0$ such that for any $\delta>0$ and any $T\geq T_{x,\delta}$, 
 \begin{equation}\label{eq:zzzz}
 \Big|\frac{1}{\pi_2(T)}\sum_{pq\leq T} f(h_{pq}x)\Big|=\bar{C}\delta+ {\rm o}_{T\to \infty}(1),
\end{equation}
since $\delta$ can be taken arbitrary small, we get that $\frac{1}{\pi_2(T)}\sum_{pq\leq T} f(h_{pq}x)=o(1)$ for any $x\in X$ generic for the Haar measure. Denote $\cA_{a,c,I_i}(t):=\frac{(a+ck_i)^2at}{a+ct}$ on $I_i=[k_i,l_i]$. Note that by the assumptions of Theorem \ref{prop:peri} we can use Lemma \ref{l:modular3} for the time $\bar{T}=T^{1+\delta^2}$ the point $=h_{t_0x}$, $t_0\leq T^{1+\delta^2}$ and  the periodic point $w$. Let $I_i=[i\bar{T}^{1-A_5\delta},(i+1)\bar{T}^{1-A_5\delta}]$. The intervals $\{I_i\}$ cover the interval $[-\bar{T}, \bar{T}]$ up to possible  the exceptional set $K$ with $|K|={\rm O}(T^{(1+\delta^2)(1-\delta)})$. Let $\{I'_i\}_{i=1}^z$ consist of those intervals in the collection $\{I_i\}$ for which $I'_i\subset [0,T]$. Then by Lemma \ref{l:modular3},
$$
\sum_{pq\leq T}f(h_{pq}x)=\sum_{pq\leq T}f(h_{pq-t_0}h_{t_0}x)=
\sum_{i=1}^{z}\sum_{pq\in I'_i} f(h_{pq-t_0}h_{t_0}x)+ {\rm O} (T^{1-\delta/2})=
$$
\begin{equation}\label{eq:zzz}
\sum_{i=1}^{z}\sum_{pq\in I'_i} f(h_{\cA_{a,c,I'_i-t_0}(pq-t_0)}w_i)+{\rm O}(\epsilon\pi_2(T)).
\end{equation}
Note that $I'_i-t_0\subset [-\bar{T},\bar{T}]$. The following proposition describes distribution of periodic points for the function $\cA(\cdot)$: 

\begin{proposition}\label{prop:vin} There exists a constant $A_6>0$ such that the following holds: for $\delta>0$ let  $I=[M,N]\subset [-T,T]$,  $|I|\geq T^{1-2A_5\delta}$. Let $\cA_{a,c,I}(t)=(a+cM)^2\cdot \frac{at}{a+ct}$ with 
$$
|a-1|,|c|\leq 2T^{-1+2A_4\delta},
$$
Then for any $f\in C_c^\infty(X)$ with $\mu_X(f)=0$ any periodic point $w$ with $per(w)\leq T^{A_5\delta}$ and any $t_0\in [0,T]$
 $$
\frac{1}{\pi_2(I)} \Big|\sum_{pq\in I}f(h_{\cA_{a,c,I}(pq-t_0)}w)\Big|\leq A_6\Big[\delta+K(\frac{1}{per(w)})\Big],
$$
where $K:\R_+\to \R_+$, $K(\delta)\to 0$ as $\delta \to 0$.
\end{proposition}
We will prove the above proposition in Section \ref{sec:prop:vin'}.  Notice however that by \eqref{eq:zzz} and the above proposition we immediately get that \eqref{eq:zzzz} holds (since $\epsilon$ is arbitrary small).

Note that using Proposition \ref{prop:vin} and \eqref{eq:zzz}, the proof of Theorem \ref{prop:peri} will be finished if we show that for most $i\leq T^{A_5\delta}$ the period of $w_i$ is large (and grows to $\infty$ with $T$). For this we will use that $x\in X$ is generic for Haar. We have:

\begin{lemma}\label{lem:disj} Let $P_W=\bigcup_{per(w)\leq W} supp(\mu_w)$, where $\mu_w$ is the probability measure on the periodic orbit $w$. Then for every $\epsilon>0$ there exists a function $g\in C(X)$, $0\leq g\leq 1$ of compact support such that  $g\equiv 0$ on $P_{\epsilon^{-1}}$, $\int_X g d\mu_{X} \geq 1-\epsilon$.
\end{lemma}
\begin{proof}
This  follows from e.g.  Strombergsson \cite{Strom04},  Lemma 3.2.
\end{proof}

Let now $x\in X$ be as in Theorem \ref{prop:peri}. Then Lemmas \ref{l:modular3} and \ref{lem:disj} imply the following:  For every $\epsilon$ there exists $T_{\epsilon,x}$ such that for $T\geq T_{\epsilon,x}$,
$$
\bigcup_{i\in Z}|I_i|\geq (1-\epsilon)T, \text{ where } Z=\{i\leq T^{A_5\delta}\;:\; per(w_i)\geq \epsilon^{-1}\}.
$$
Indeed, note that by Lemma \ref{lem:disj} for $\epsilon^2$ it follows that $g$ vanishes an all periodic orbits of period $\leq \epsilon^{-2}$.Therefore, 
$$
(1-\epsilon^2)T\leq T\mu_X(g)\leq \sum_{n\leq T}g(h_nx)+\epsilon^2=\sum_{i} \sum_{n\in I_i}g(h_{\cA{a,c,t_i}(n-t_0)}w_i)+ 2\epsilon^2 \leq \bigcup_{i\in Z}|I_i|.
$$

The above shows that for most $i\leq T^{A_4\delta}$, $per(w_i)\geq \epsilon^{-1}$ and so Proposition \ref{prop:vin} and \eqref{eq:zzz} finish the proof.

It remains to prove Proposition \ref{prop:vin}.

\subsection{Proof of Proposition \ref{prop:vin}}\label{sec:prop:vin'}
In this section we will prove Proposition \ref{prop:vin}. Let $\R\ni R=per(w)<T^{A_5\delta}$. We will consider the function $f$ restricted to the periodic orbit $\{h_tw:0\leq t\leq R\}$. Recall that $\cA_{a,c,I}(t)=(a+cM)^2\frac{at}{a+ct}$. Let $\beta=1/R$.
Let $1\leq b\leq r<|I|$, $(b,r)=1$, be such that
\begin{equation}\label{eq:bet}
\Big\|(a+cM)^2\beta-\frac{b}{r}\Big\|<\frac{1}{r}|I|^{-1}.
\end{equation}

We will consider two cases:

{\bf CASE I (major arc case):} $r\leq T^{1000A_5\delta}$ and $|c|\leq T^{100A_5\delta}|I|^{-2}$. In this case we split the interval $I$ into disjoint  intervals $\{J_s=[z'_s,z'_{s+1}]\}$ such that $|J_s|\sim T^{-998A_5\delta}|I|\geq T^{1-999A_5\delta}$ (we WLOG assume that $z'_s$ are integers). Note that for $t\in J_s$, by Taylor's formula (up to order 2)
 $$ 
 \cA_{a,c,I}(t)= \cA_{a,c,I}(z'_s)+(a+cM)^2 \frac{a^2}{(a+cz'_s)^2}(t-z'_s)+ \rm{O}(T^{-A_5\delta}), 
$$
as $|\partial_{tt}\cA_{a,c,I}(\theta)(t-z'_s)^2|\leq |J_s|^2 (a+cM)^2\frac{2a^2|c|}{(a+c\theta)^3}=\rm{O}(T^{-A_5\delta})$ by the bound on $c$. We have
$$
|(a+cM)^2 \frac{a^2}{(a+cz'_s)^2}-(a+cM)^2|\ll \frac{2acz'_s+(cz'_s)^2}{(a+cz'_s)^2}\ll T^{\delta} T^{100A_5\delta} |I|^{-2}T \ll T^{-A_5\delta}|J_s|^{-1},
$$
as $|I|\geq T^{1-A_5\delta}$. This shows that $|(a+cM)^2 \frac{a^2}{(a+cz'_s)^2}(t-z'_s)-(a+cM)^2(t-z'_s)|\ll T^{-A_5\delta}$. Moreover, by the bounds on $a,c$, 
$$|(a+cM)^2(t-z'_s)-(t-z'_s)|\leq |J_s||a+cM-1||a+cM+1|\ll $$$$T^{-998A_5\delta}|I|(T^{-1+2A_4\delta}+T^{100A_5\delta}|I|^{-2}T)\leq T^{1-998A_5\delta}(T^{-1+2A_4\delta}+T^{100A_5\delta}T^{-2+2A_5\delta}T)\leq T^{-A_5\delta}.
$$
Summarizing,  $\cA_{a,c,I}(t-t_0)= \cA_{a,c,I}(z'_s-t_0)+(t-z'_s-t_0)+ \rm{O}(T^{-A_5\delta})$.
Therefore for $w_s=h_{\cA_{a,c,I}(z'_s-t_0)}w$ and $z_s=z'_s-t_0$, we have
$$
\sum_{pq\in I}f(h_{\cA_{a,c,I}(pq-t_0)}w)=\sum_{s}\sum_{pq\in J_s}f(h_{pq-z_s}w_s) + {\rm O}(T^{-A_5\delta}\pi_2(I))
$$
Note that by \eqref{eq:bet}, 
$$\Big\|\beta-\frac{b}{r}\Big\|\leq |(a+cM)^2-1|+ \Big\|(a+cM)^2\beta-\frac{b}{r}\Big\|<
3 T^{-A_5\delta}|J_s|+\frac{1}{r}|I|^{-1}\ll T^{-2A_5\delta}|J_s|
$$
Fix $s$ and let $0<v<r$. If $(pq-z_s)b\equiv v \text{ mod }r$, i.e. $(pq-z_s)\frac{b}{r}=k+v/r$, then 
$$
\frac{pq-z_s}{R}=(pq-z_s)(\beta-\frac{b}{r})+k+\frac{v}{r}=k+\frac{v}{r}+ {\rm O}(T^{-2A_5\delta}).
$$
Therefore, $pq-z_s=kR+\frac{vR}{r}+ {\rm O}(T^{-A_5\delta})$, as $R\leq T^{A_5\delta}$.
Let $M_{v,s}:=\{pq\in J_s,\;:\; (pq-z_s)b\equiv v-z_sb\mod r \}$.  Let $w_s'= h_{-z_sbR/r}w_s$. Since $w_s$ is periodic of period $R$, we get 

$$
\sum_{pq\in J_s}f(h_{pq-z_s}w_s)=\sum_{v<r} \sum_{pq\in M_{v,s}}f(h_{pq-z_s}w_s)=\sum_{v<r} |M_{v,s}|f(h_{\frac{vR}{r}}w'_s)+ {\rm O}(T^{-A_5\delta} |J_s|).
$$
By Proposition \ref{thm:Maks2} (for $\kappa=999A_5\delta$), using that $r<T^{1000A_5\delta}$
$$
\sum_{v<r} |M_{v,s}|f(h_{\frac{vR}{r}}w'_s)=\frac{\pi_2(J)}{\phi(r)}  \sum_{v<r}1_{(v,r)=1}f(h_{\frac{vR}{r}}w'_s)+ $$$$\frac{F_{J,r,T}}{\phi(r)}  \sum_{v<r}\chi(v) f(h_{\frac{vR}{r}}w'_s)+{\rm O}(\delta\pi_2(J)).
$$
We will consider two cases depending on how large $r$ is with respect to $R$ (trivially $r\geq R$). 

{\bf Subcase A:} $r\leq R^{20}$.
In this case the proof is finished by Proposition \ref{prop:mult} below (using it for $\nu=1_{(\cdot,r)}$ and $\nu=\chi$), as $\phi(r)\gg r [\log \log r]^{-1}$ and $r\leq R^{20}$.

{\bf Subcase B:} $r\geq R^{20}$. In this case we split the interval $[0,R]$ into intervals of $\{U_i=[u_i,u_{i+1}]\}$ of size $U\sim \frac{r^{1-\xi}}{R}$, $\xi=1/100$. Note that for $v\in U_i$, $|\frac{vR}{r}- \frac{u_iR}{r}|\ll r^{-\xi}$. Therefore 
$$
\sum_{v<r}1_{(v,r)=1} f(h_{\frac{vR}{r}}w'_s)=\sum_{i\leq r/U }\Big(\sum_{v\in U_i}1_{(v,r)=1}\Big)f(h_{\frac{u_iR}{r}}w'_s) +{\rm O}\Big(r^{1-\xi}\Big).
$$
Moreover, by Lemma \ref{lem:taPV}, 
$$
\sum_{i}\Big(\sum_{v\in U_i}1_{(v,r)=1}\Big)f(h_{\frac{u_iR}{r}}w'_s)=\phi(r)\frac{U}{r}\sum_{i\leq r/U}f(h_{\frac{u_iR}{r}}w'_s)+ o(\phi(r))
$$
Note that the points $h_{\frac{u_iR}{r}}w'_s$ are $r^{-\xi}$ dense on the periodic orbit of period $R\to \infty$. Therefore $Ur^{-1}\sum_{i\leq r/U}f(h_{\frac{u_iR}{r}}w'_s)$ is close to the integral of $f$ on this periodic orbit. Using that measures on long periodic orbits distribute towards the Haar measure, \cite{Sarn}, and $\mu_X(f)=0$, it follows that $Ur^{-1}\sum_{i\leq r/U}f(h_{\frac{u_iR}{r}}w'_s)$ is $o(1)$. Therefore, $\sum_{v<r}1_{(v,r)=1} f(h_{\frac{vR}{r}}w'_s)=o(\phi(r))$. Analogous reasoning for $\chi$ in place of $1_{(v,r)=1}$ together with Lemma  \ref{lem:taPV} shows that $\sum_{v<r}\chi(v) f(h_{\frac{vR}{r}}w'_s)=o(\phi(r))$. This finishes the proof in this case.

\begin{proposition}\label{prop:mult} Let $\nu$ be any multiplicative function, $|\nu|\leq 1$ and let $w$ be a periodic point of period $R$. Then for every $r\in [R,R^{20}]$
$$
|\sum_{n\leq r} \nu(n)f(h_{nR/r}w)|\ll \frac{r}{[\log \log(R)]^2},
$$
where the implied constant depends only on $X$, $\nu$ and $f$ (but not on $w$).
\end{proposition}
\begin{proof}  We will show this by using the condition implying \eqref{eq:mob'} with $X=r$. For this for $M\in [r^{9/10},r]$ and $p,q$ primes with  $p,q\in \Big[ e^{(\log \log \log r)^3},e^{(\log \log r)^{10}}\Big]$ with $1/5\leq p/q\leq 5$ we need to show that 
\begin{equation}\label{eq:hp}
|\sum_{n\leq M}f(h_{pnR/r}w)f(h_{qnR/r}w)|\leq  \frac{M}{(\log\log M)^{10}}.
\end{equation}
Note that the LHS above can, using renormalization, be written as 
$$
\sum_{n\leq M}f(h_{pnR/r}w)f(h_{qnR/r}w)=
$$$$
\sum_{n\leq M}(f\times f)\circ(a_{\log(pR/r)},a_{\log(qR/r)}) (h_n\times h_n)(a_{-\log(pR/r)}w,a_{-\log(qR/r)}w).
$$
Since $\log(pR/r)=\log p+\log R/r$, the above can be written as 
$$
\sum_{n\leq M}(\bar{f}\times \bar{f})\circ(a_{\log(p)},a_{\log(q)}) (h_n\times h_n)(a_{-\log(p)}\bar{w},a_{-\log(q)}\bar{w}),
$$
where $\bar{f}=f\circ a_{\log R/r}$ and $\bar{w}=a_{-\log(R/r)}w$. Since $w$ is periodic of period $R$,  $h_{r}\bar{w}=h_ra_{-\log R/r}w=a_{-\log(R/r)}h_{R}w=\bar{w}$ and so $\bar{w}$ is periodic of period $r\in [R,R^{20}]$.

Note that we can analyze the above expression using again Theorem \ref{LMWdrugie} for the point  $(a_{-\log(p)}\bar{w},a_{-\log(q)}\bar{w})$ and different parameters: $T=M$ and $R=R'=e^{(\log \log r)^{20}}$.  Assume first that E1 or E2 holds for all $p,q\in  \Big[ e^{(\log \log \log r)^3},e^{(\log \log r)^{10}}\Big]$ and assume additionally that for those $p,q$ for which E2 holds we have that $E2_1$ in Proposition \ref{prop:E2} holds for $\Psi(T)=\log \log T$. Then notice that  \eqref{eq:hp} holds for all $p,q$. So by Proposition \ref{prop:E2} we only need to consider the case in which there are exist $p,q\in  \Big[ e^{(\log \log \log r)^3},e^{(\log \log r)^{10}}\Big]$ for which either E3. holds or $E2_2$ holds for the point $\bar{w}$. But  this in both these cases (using also Proposition \ref{prop:E3}) means that there exists a point $\tilde{w}\in X$ of period $per(\bar{w})\leq M^{A_4\delta}$ and $t_0\leq M$ such that $d_X(h_{t_0}\bar{w},\tilde{w})<M^{-1+A_4\delta}$. By the definition of $\tilde{w}$ it follows that  $\tilde{w}=a_{\log \bar{R}}h_{t'}(e) \Gamma$, for some $t'\leq 1$, $\bar{R}\leq M^{A_4\delta}$. This by the bound on $\bar{R}$ then implies that 
 $d_X\Big(a_{-\log \bar{R}}h_{t_0}\bar{w},h_{t'}(e)\Big)<M^{-1+10A_4\delta}$. Similarly, since $h_{t_0}\bar{w}\in X$ is a periodic point of period $r$, $h_{t_0}\bar{w}=a_{\log r} h_{t''}(e) \Gamma$, for some $t''\leq 1$. This implies that 
 
 $$
 d_X\Big(a_{\log (r\bar{R}^{-1})}h_{t''}(e),h_{t'}(e)\Big)<M^{-1+10A_4\delta}
 $$
 Notice that the point $u=a_{\log (r\bar{R}^{-1})}h_{t''}(e)$ is periodic of period $r\bar{R}^{-1}$. Moroever the above assumption implies in particular that 
 $$
d_X\Big( \{h_t u \}_{t\leq M^{2/3}}, \{h_te\}_{t\leq 1}\Big)<(\log M)^{-1}.
 $$
Note that $M\geq r^{9/10}\geq (r\bar{R}^{-1})^{3/5}$. However the above implies that the orbit of length $\geq (r\bar{R}^{-1})^{3/5} $ of a periodic point of period $r\bar{R}^{-1}$ is not equidistributed (it is trapped in the neighborhood of the orbit of $e$). This contradicts the fact that for any $\epsilon>0$ and any $T>0$ pieces of periodic horocycles of period $T$ of length $\geq T^{1/2+\epsilon}$ become equidistributed, \cite{Strom04}. This implies that $E2_2$ or E3 are never satisfied. The proof is finished.
\end{proof}

{\bf CASE II (minor arc case):} We have either $r\geq T^{1000A_5\delta}$ or  $|c|\geq T^{100A_5\delta}|I|^{-2}$.

 Let $L_t(\theta)=\theta+t \mod 1$ be the linear flow on the circle $S^1$. Let $\Delta:\{h_tw:0\leq t\leq R\}\to S^1$ be given by $\Delta(w)=0$ and $\Delta(h_tw)=L_{t/R}0$. 
and $\tilde{f}:S^1\to \R$,
$\tilde{f}(x)=f(\Delta^{-1}x)$. Then since the map $\Delta$ is equivariant,
$$
\sum_{pq\in I}f(h_{\cA_{a,c,I}(pq-t_0)}w)=\sum_{pq\in I}\tilde{f}(L_{\cA_{a,c,I}(pq-t_0)\beta}0)=\sum_{pq\in I}\tilde{f}(\cA_{a,c,I}(pq-t_0)\cdot \beta).
$$
Since $\tilde{f}$ is a function on the circle it follows that $\tilde{f}(x)=\sum_{n\in \Z} a_ne_n(x)$, where $a_n=\int_{S^1}\tilde{f}(x)e_n(x)dx$. Note that $a_0=\int_{S^1}\tilde{f}(x)dx=\frac{1}{R}\int_{\{h_tw\;:\; 0\leq t\leq R\}}f(x)d\mu_w\leq A_6 K(\frac{1}{per(w)})$, for some function $K:\R_+\to \R_+$ with $K(\delta)\to 0$ as $\delta\to 0$. This follows from the well known fact, see e.g. \cite{Sarn}, that if $per(w)\to \infty$, then $\mu_{per(w)}\to \mu_{X}$ and we know that $\mu_X(f)=0$.

 Since $\|\tilde{f}''\|_{C^2}= {\rm O}(\|f''\|_{C^2})\cdot R^2$ it follows by integration by parts that $|a_n|={\rm O}(R^2/n^2)$. Therefore 
$$
\Big|\sum_{|n|\geq R^{2.01}}a_n (\sum_{pq\in I}e_n(\cA_{a,c,I}(pq)\cdot \beta))\Big|\ll \pi_2(I)\cdot \frac{1}{R^{0.01}}.
$$
Moreover, 
$$
\Big|\sum_{|n|\leq R^{2.01}, n\neq 0}a_n (\sum_{pq\in I}e_n(\cA_{a,c,I}(pq-t_0)\cdot \beta))\Big|\ll R^{2.01}\cdot  \max_{|n|\leq R^{2.01}}\Big|\sum_{pq\in I}e_n(A_{a,c,I}(pq)\cdot \beta)\Big|= o(\pi_2(I)),
$$
the last inequality by Theorem \ref{thm:maks1}. Indeed, note that $\cA_{a,c,I}(pq-t_0)n\beta=m_{a,c}(pq-t_0)\cdot n(a+cM)^2\beta$. (see \eqref{eq:muv}). Note that by \eqref{eq:bet} and the bound $n\leq R^{2.01}$, 
$$\|n(a+cM)^2\beta- \frac{nb}{r}\|\leq \frac{R^{2.01}}{r|I|},$$
 So we can indeed use Proposition \ref{thm:maks1} with $\tilde{\beta}=n(a+cM)^2\beta$ remembering that $R\leq T^{A_5\delta}$. This finishes the proof in this case.



\section{Distribution of semi-primes in short intervals}
\begin{lemma}\label{lem:taPV} Let $r\in \N$ and let $I\subset [0,r]$, $|I|\geq r^{1-1/100}$. Then 
$$
\sum_{v\in I} 1_{(v,r)=1}= \frac{\phi(r)|I|}{r}+ o\Big(\frac{\phi(r)|I|}{r}\Big).
$$
Moreover, if $\chi$ is a real quadratic non-principal character, then 
$$
\sum_{v\in I} \chi(v)= o\Big(\frac{\phi(r)|I|}{r}\Big).
$$
\end{lemma}
\begin{proof} 
  
  By Polya-Vinogradov, 
  $$
  \sum_{v \in I} \chi(v) \ll \sqrt{r} \log r
  $$
  thus giving us the second part of the Lemma. 
  For the first part of the Lemma, we notice that, by the formula for the sum of the M\"obius function $\mu$ over divisors,
  $$
  \sum_{v \in I} 1_{(v,r) = 1} = \sum_{v \in I} \sum_{\substack{d | v \\ d | r}} \mu(d) = \sum_{d | r} \mu(d) \sum_{u: d u\in I} 1 =
  \sum_{d | r} \mu(d) \cdot \Big ( \frac{|I|}{d} + O(1) \Big )
  $$
  By M\"obius inversion formula, this is
  $$
  \frac{\phi(r)}{r} \cdot |I| + O(d(r))
  $$
  and the result follows, since $d(r) \ll_{\varepsilon} r^{\varepsilon}$ for every $\varepsilon > 0$. 

\end{proof}

\subsection{Siegel-Walfisz to large moduli}
\begin{proposition}\label{thm:Maks2} Let $J\subset [0,T]$ be an interval of length $|J|\geq T^{1-\kappa}$ and let $r<T^{\kappa}$. Then
$$|\{pq\in J\;:\; pq\equiv a \mod r\}|=1_{(a,r)=1}\frac{\pi_2(J)}{\varphi(r)}+\chi(a) \cdot \frac{F_{J, r, T}}{\varphi(r)}+ {\rm O}\Big(\frac{1}{\varphi(r)}\kappa \pi_2(J)\Big),$$
where $\chi$ is a quadratic real character and $F_{J, r, T}$ is such that $|F_{J, r, T}| \leq \pi_2(J)$ for all large enough $T$. 
\end{proposition}
\begin{proof}

  First assume without loss of generality that $J \subset [T^{1 - 2 \kappa}, T]$. Indeed dropping the integers in $J \cap [0, T^{1 - 2 \kappa}]$ incurs a loss of at most $T^{1 - 2 \kappa}$ integers which is completely acceptable.

  We also notice that we can assume 
  that one of the primes in $p q$ is $< T_0 := \exp(\log T / \log\log T)$.
  Indeed, by Brun-Titchmarsh (see e.g \cite{MontgomeryVaughan}), the contribution of integers
  $p q$ with $p, q > T_0$ is 
  $$
  \ll \sum_{\substack{p q \in J \\ T_0 < p \leq \sqrt{T} < q \\ p q \equiv a \pmod{r}}} 1 \ll \frac{|J|}{\varphi(r)\log T} \sum_{T_0 < p < \sqrt{T}} \frac{1}{p},
  $$
  and this is 
  $$
   \ll \frac{|J|}{\varphi(r)} \cdot \frac{\log\log\log T}{\log T}
  $$
  and is therefore negligible. 

  We now evaluate the contribution of the remaining integers, which is
  $$
  \sum_{\substack{ p< T_0 \\ p \nmid r}} \sum_{\substack{p q  \in J \\ q \equiv \overline{p} a \pmod{r}}} 1. 
  $$
  We now notice that, since $p <T_0$ and so $q > T^{1-4\kappa}$,
  $$
  \sum_{\substack{ q \in J / p \\ q \equiv \overline{p} a \pmod{r}}} 1
  = \frac{(1 + O(\kappa))}{\log T} \sum_{\substack{q \in J / p \\ q \equiv \overline{p} a \pmod{r}}} \log q
  $$
  By Brun-Titchmash the error term from $O(\kappa)$ is acceptable and absorbed by the final error term. We now open the sum over $q$ into characters,
  $$
  \frac{1}{\log T}\sum_{\substack{q \in J / p \\ q \equiv \overline{p}a \pmod{r}}} \log q = \frac{1}{\varphi(r) \log T} \sum_{\chi \pmod{r}} \sum_{q \in J / p} \log q \cdot \chi(q)  \chi(p \overline{a}).
    $$
    We separate the contribution of the principal and quadratic character,
  \begin{align} \label{hoh}
    \frac{1}{\varphi(r) \log T} \cdot \sum_{q \in J / p} \log q & +
  \frac{\psi(p a)}{\varphi(r) \log T} \cdot \sum_{q \in J / p} \psi(q) \log q  \\ \label{hoh2} & + \frac{1}{\varphi(r) \log T} \sum_{\chi^2 \neq \chi_0 \pmod{r}} \chi(p)\overline{\chi(a)}\sum_{q \in J / p}  \chi(q) \log q
  \end{align}
  where $\psi$ is the quadratic character $\pmod{r}$. By Huxley's theorem \cite{Huxley} (see also \cite{H-B88}),
  $$
  \sum_{q \in J / p} \log q \sim \frac{|J|}{p}.
  $$
  Therefore the total contribution of the first term is 
  $$
  (1 + O(\kappa)) \frac{\pi_2(J)}{\varphi(r)}
  $$
  as expected. 
  
  By Gallagher's theorem \cite[Theorem 7]{Gallagher} the term in \eqref{hoh2} is
  \begin{equation} \label{siegel}
\ll \sum_{\chi^2 \neq \chi_0} \Big | \sum_{q \in J / p} \chi(q) \log q \Big | \ll \frac{|J|}{p} \exp \Big ( - c \cdot \frac{\log T}{\log r} \Big )
  \end{equation}
  and therefore its total contribution is
  $$
    \ll \frac{\pi_2(J)}{\varphi(r)} \cdot \exp \Big ( - \frac{c}{\kappa} \Big ).
  $$
  It thus remains to deal with the contribution of the quadratic character. If there is no Siegel-zero then by Gallagher's theorem the contribution of the quadratic character to \eqref{hoh} is bounded by the right-hand side of \eqref{siegel} and therefore gives also an acceptable total contribution. Meanwhile if there is a Siegel zero, then the contribution of the quadratic character to \eqref{hoh} is,
  $$
  - \frac{\psi(p a)}{\varphi(r) \log T} \cdot \frac{|J|}{p} \xi_{J / p}^{- \delta_r} + O \Big ( \frac{|J|}{\varphi(r) \log T} \cdot \frac{1}{p} \cdot \exp \Big (- \frac{c}{\kappa} \Big ) \Big ).
  $$
  where $\xi_{J / p} \in J / p$ and $\delta_{r}$ is the smallest positive real number such that $L(1 - \delta_r, \psi) = 0$. The
  total contribution of the error term is once again acceptable, while the contribution of the main term is, 
  $$
  - \frac{\psi(a)}{\varphi(r)} \cdot \frac{|J|}{\log T} \sum_{p < T_0} {{\frac{\psi(p)}{p}}} \cdot \xi_{J / p}^{-\delta_r}.
  $$
  it then suffices to note that we can write, 
  $$
F_{J,r,T} := \frac{|J|}{\log T} \sum_{p < T_0}{{\frac{\psi(p)}{p}}}  \cdot \xi_{J / p}^{-\delta_r}
  $$
  and this has the property that 
  $$
  |F_{J, r, T}| \leq \frac{\pi_2(J)}{\varphi(r)}
  $$
  provided that $T$ is sufficiently large. 
\end{proof}

\subsection{Minor arc estimates}
We will also need to following result:
\begin{proposition}\label{thm:maks1}
  Let $\delta \in (0, \tfrac{1}{100})$. Let $T$ be given, $R \leq T^{10 A_5 \delta}$ with $A_{5} > 0$ an absolute constant and $I \subset [0, T]$ an interval of length $> T^{1 - A_5 \delta}$.
  Let $\mu, \beta, \nu$ be real numbers with
  $$
  |\mu - 1| \leq 2 T^{-1 + A_5 \delta} \ , \ \beta = \frac{1}{R} \ , \ |\nu| \leq 2 T^{-1 + A_5 \delta}.  $$
  Let
  \begin{equation}\label{eq:muv}
  m_{\mu, \nu}(t) := \frac{\mu t}{\mu + \nu t}. 
  \end{equation}
  Suppose that either
  \begin{enumerate}
  \item $|\nu| > T^{100 A_5 \delta} \cdot |I|^{-2}$, 
  \item or, there exists an $T^{1000 A_5 \delta} \leq r \leq |I| T^{-1000 A_5 \delta}$ and $(b,r) = 1$, such that,
    $$
    \Big \vert \beta - \frac{b}{r} \Big \vert \leq \frac{T^{1000 A_5 \delta}}{r |I|}. 
    $$
  \end{enumerate}
  Then, for all $A_5 \delta$ sufficiently small, and all $t_0\in [0,T]$
  $$
  \Big\vert \sum_{p q \in I} e (m_{\mu, \nu}(p q-t_0) \beta) \Big\vert   \ll |I| T^{- 15 A_5 \delta}. 
  $$
\end{proposition}

We split according to the two possible cases.
\subsection{ Case $|\nu| > |I|^{-2} T^{100 A_5 \delta}$}

Our main tool will be the following two Lemmas. The first Lemma amounts to an application of Poisson summation followed by a bound on the oscillatory integrals. 
\begin{lemma} \label{le:Robert1}
  Let $\alpha \geq 1$ be a real number. There exists a constant $C(\alpha)$  such that for any interval $I$, any real number $\lambda_{2} > 0$ and twice differentiable function $f : I \mapsto \mathbb{R}$ such that,
  $$
  \lambda_{2} \leq |f''(x)| \leq \alpha \lambda_{2} \ , \ x \in I
  $$
  one has,
  $$
  \Big | \sum_{m \in I} e(f(m)) \Big | \leq C(\alpha) \Big ( |I| \lambda_2^{1/2} + \lambda_{2}^{-1/2} \Big ) \,.
  $$
\end{lemma}
\begin{proof}
  See \cite[Theorem 1]{Robert}
\end{proof}

The second Lemma amounts to an application of van der Corput differencing followed by an application of Poisson summation and the estimation of oscillatory integrals.
\begin{lemma} \label{le:Robert2}
  Let $\alpha \geq 1$ be a real number. There exists a constant $C(\alpha)$  such that for any interval $I$, any real number $\lambda_{2} > 0$ and twice differentiable function $f : I \mapsto \mathbb{R}$ such that,
  $$
  \lambda_{3} \leq |f'''(x)| \leq \alpha \lambda_{3} \ , \ x \in I
  $$
  one has,
  $$
  \Big | \sum_{m \in I} e(f(m)) \Big | \leq C(\alpha) \Big ( |I| \lambda_3^{1/6} + |I|^{1/2} \lambda_{3}^{-1/6} \Big )\,.
  $$
\end{lemma}
\begin{proof}
  See \cite[Theorem 2]{Robert}.
\end{proof}

We also recall the standard method of Type-I/Type-II sums for estimating sums over primes.
\begin{lemma} \label{le:Type}
  Suppose that $f : \mathbb{N} \rightarrow \mathbb{C}$ is an arbitrary sequence supported on $[0, T]$. Let $\alpha_a$ and $\beta_b$ denote two arbitrary sequences with $|\alpha_a| \leq 1$ and $|\beta_b| \leq 1$ supported respectively in $[A, 2A)$ and $[B, 2B)$ with $A B \asymp T$. Suppose that $\delta > 0$ is such that for all $T^{1/10} \leq A \leq T^{1/2} \leq B$ we have,
      $$
       \Big\vert  \sum_{a,b} \alpha_a \beta_b f(a b)  \Big\vert  \ll T^{1 - \delta}.
      $$
      Suppose also that for all $A \leq T^{1/10}$ we have,
      $$
      \Big\vert  \sum_{a,b} \alpha_a f(a b)  \Big\vert  \ll T^{1 - \delta}. 
      $$
      Then,
      $$
       \Big\vert  \sum_{p} f(p)  \Big\vert  \ll T^{1 - \delta / 2}.  
      $$
\end{lemma}
\begin{proof}
  This is a standard Type-I/Type-II estimate, see for example \cite{H-B82} or \cite[Lemma 3.1, Lemma 3.3]{Polymath}.
\end{proof}

We will naturally choose $f(n) := e(m_{\mu, \nu}(n - t_{0})) \mathbf{1}_{n \in I}$. We will use the lemmas below to obtain the required type-I and type-II information.
\begin{lemma} \label{le:TypeII}
  Let $T$ be given, $R \leq T^{10 A_5 \delta}$ and $I \subset [0, T]$ an interval of length $> T^{1 - A_5 \delta}$.
  Let $\alpha_a$ and $\beta_b$ be arbitrary coefficients with $|\alpha_a| \leq 1$ and $|\beta_b| \leq 1$ and supported respectively in $[A, 2A]$ and $[B, 2B]$ with $ A B \asymp T$ and $T^{40 A_5 \delta} \leq A \leq B$. Suppose that $\mu, \nu$ and $\beta$ are such that,
  $$|\mu - 1| \leq 2 T^{-1 + A_5 \delta} \ , \ |I|^{-2} \cdot T^{1000 A_5 \delta} \leq |\nu| \leq 2 T^{-1 + A_5 \delta} \ , \ \beta = \frac{v}{R}$$
  with $1 \leq v \leq R^3$. 
  Let 
  $$
  m_{\mu, \nu}(t) = \frac{\mu t}{\mu + \nu t}. 
  $$
  Then, for $1 \leq \ell \leq T^{40 A_5 \delta}$ and $t_{0} \in [0, T]$
  $$
 \Big\vert   \sum_{a b \in I} \alpha_a \beta_b e(m_{\mu,\nu}(\ell a b - t_{0}) \beta)  \Big\vert  \ll T^{1 - 20 A_5 \delta} \,.
  $$
\end{lemma}
\begin{proof}
  We start by first cutting the interval $I$ into shorter (disjoint) intervals $\{ I' \}$ of length $T^{1 - 100 A_5 \delta}$, and a remaining shorter interval that we simply ignore by bounding its contribution trivially.
  We also bound trivially the contribution of any interval $I'$ for which there exists $a b \in I'$ such that
  $$
  |\mu + \nu (\ell a b - t_{0}) | \leq T^{-20A_{5} \delta}.
  $$
  This removes at most $T^{1 - 20 A_{5} \delta}$ integers and is therefore acceptable. 

  On the remaining intervals $I'$ we can therefore assume that for all $ a b \in I'$ we have
  $|\mu + \nu (\ell a b - t_{0})| > T^{-20 A_{5} \delta}$.
  We fix such an interval $I'$. We aim to obtain a bound $|I'| T^{-20 A_5 \delta}$ for
  $$
 \Big\vert   \sum_{a b \in I'} \alpha_a \beta_b e(m_{\mu, \nu}(\ell a b - t_{0}) \beta)  \Big\vert \,.
  $$
  By Cauchy-Schwarz the above is
  $$
  \leq B^{1/2} \Big ( \sum_{B \leq b < 2B} \Big | \sum_{a b \in I'} \alpha_a e(m_{\mu, \nu}(\ell a b - t_{0}) \beta) \Big |^2 \Big )^{1/2}
  $$
  Expanding the square we get, 
  $$
  \sum_{a_1, a_2} \alpha_{a_1} \overline{\alpha_{a_2}} \sum_{\substack{B \leq b < 2B \\ a_1 b , a_2 b \in I'}} e\Big(m_{\mu, \nu}(\ell a_1 b - t_{0}) \beta - m_{\mu, \nu}(\ell a_2 b - t_{0}) \beta \Big)
  $$
  We notice that the condition $a_1 b, a_2 b \in I'$ implies $|a_1 - a_2| \ll |I'| / B \asymp A T^{-100 A_5 \delta}$. We further separate this sum into terms with $a_1 = a_2$ and terms with $a_1 \neq a_2$. Thus we get,
  $$
  \ll |I'| + \sum_{\substack{0 < |a_1 - a_2| \ll A T^{-100 A_5 \delta}}} \Big | \sum_{\substack{b \in J}} e \Big ( m_{\mu, \nu}(\ell a_1 b - t_{0}) \beta - m_{\mu, \nu}(\ell a_2 b - t_{0}) \beta \Big ) \Big |
  $$
  where $J := (I' / a_1) \cap (I' / a_2) \cap [B, 2B]$.
  Let
  $$
  f_{\mu,\nu,\beta, \ell, a_1, a_2, t_{0}}(x) :=  m_{\mu, \nu}(\ell a_1 x - t_{0}) \beta - m_{\mu, \nu}(\ell a_2 x - t_{0}) \beta.
  $$
  We now differentiate the function $f$. We notice that
  $$
  m_{\mu, \nu}(t) = \frac{\mu t}{\mu + \nu t} = \frac{1}{\nu} \cdot \Big ( {{\mu}} - \frac{\mu^2}{\mu + \nu t} \Big )
  $$
  Therefore, for $k\geq 1$,
  $$
  m_{\mu, \nu}^{(k)}(t) = \frac{c_k \mu^2 \nu^{{{k-1}}} }{(\mu + \nu t)^{k + 1}}
  $$
  for some coefficients $c_k \neq 0$. 
  Therefore
  $$
  f^{(k)}(x) = c_k \mu^2 \beta \ell^k \nu^{k - 1} \cdot \Big (  \frac{a_1^k}{(\mu + \nu \ell a_1 x - \nu t_{0})^{k + 1}} - \frac{a_2^k}{(\mu + \nu \ell a_2 x - \nu t_{0})^{k + 1}} \Big ).
  $$
  To understand this quantity write $a_2 = a_1 + h$ and note that $|h| \ll A T^{-100 A_5 \delta}$. Therefore expanding in a Taylor series, we find, 
  \begin{align*}
  & \frac{(a_1 + h)^{k}}{(\mu + \nu \ell (a_1 + h) x - \nu t_{0})^{k + 1}} \\ & = \frac{a_1^k}{(\mu + \nu \ell a_1 x - \nu t_{0})^{k + 1}}\cdot \Big (1 +  \frac{h}{a_1} \cdot \Big ( k - \frac{(k + 1) \nu \ell x a_1}{\mu + \nu \ell a_1 x - \nu t_{0}} \Big ) + O \Big ( \frac{T^{40 A_{5} \delta} |h|^2}{A^2} \Big ) \Big ).
  \end{align*}
  since by assumption for all $x \in J$ we have $| \mu + \nu ( \ell a_{1} x - t_{0} ) | > T^{-20A_{5} \delta}$.
  We can cut the interval $J$ into a union of $\ll \log T$ intervals $J^{(k)}_{U, \mu, \nu \ell a_1}$ on which
  $$
  \Big | k - \frac{(k + 1) \nu \ell x a_1}{\mu + \nu \ell x a_1 - \nu t_{0}} \Big | \asymp  e^{- U}
  $$
  with $|U| \leq 50 \delta A_5 \log T$ and an interval $J^{(k)}_{\infty, \mu , \nu \ell a_1}$ on which the left-hand side in the above formula is less than $T^{-50 A_5 \delta}$. Notice that because of the triangle inequality we don't care if the above intevals are disjoint, we will also not care about their lengths except for the interval $J^{{(k)}}_{\infty, \mu, \nu \ell a_{1}}$ and we don't care if the same $U$ repeats for several of the intervals. We only care that there is not too many such intervals (e.g $\ll \log T$) and that $J^{(k)}_{\infty}$ is reasonably short.  By abuse of notation we will also write,
  $$
  J^{(k)}_{U} := J^{(k)}_{U, \mu, \nu \ell a_{1}} \ , \ J^{(k)}_{\infty} := J^{(k)}_{\infty, \mu, \nu \ell a_{1}}.
  $$
  All subsequent $\ll$ and $\asymp$ symbols are allowed to depend on $k$.
  Notice that,
  $$
  |J^{(k)}_{\infty}| \ll B T^{-49 A_{5} \delta},
  $$
  since on such an interval necessarily $\nu \ell x a_{1} = \mu - \nu t_{0} + o(1)$ and $\mu - \nu t_{0} \asymp 1$.

Notice furthermore that,
  $$
  \vert f^{(k)}(x) \vert \asymp \lambda_k:= e^{-U} \cdot \ell^{k} \vert \nu\vert ^{k - 1} |a_1|^{k-1}  |a_1 - a_2| |\beta| \ , \ x \in J^{(k)}_{U, \mu, \nu \ell a_1}\,.
  $$
     We further separate into subcases according to the size of $\vert \nu\vert$. To avoid further cluttering the notation we will drop extraneous subscripts from the intervals $J$, which is permissible since when summing over $b$ we treat $a_1, h, \mu, \nu, \ell$ as fixed.

  \subsubsection{The subcase $|\nu| > T^{-1 - 1/10}$}

  First we trivially bound the contribution of $b \in J^{(3)}_{\infty}$. The total contribution is
  $$
  B^{1/2} \cdot \Big ( B \cdot T^{-49 A_5 \delta} \cdot \frac{|I'|^2}{B^2} \Big )^{1/2} \ll |I'| \cdot T^{-24 A_5 \delta}.
  $$
  Fix an interval $J_{U}^{(3)}$. 
  Notice that for $x \in J_{U}^{(3)}$, since by hypothesis $|\nu| \leq  2T^{-1 + A_5 \delta }$, we have
  $$ \lambda_3 \ll \ell^3 |\beta| \cdot T^{-2 + 2 A_5 \delta} A^3 \ll T^{-1/2 + 200 A_5 \delta} $$ while
  $$ \lambda_3  \gg |a_1|^{2}  |a_1 - a_2| T^{-2 - 1/5 - 1000 A_5 \delta}.$$ Therefore by Lemma \ref{le:Robert2},
  $$
  \begin{aligned}
 \Big\vert  \sum_{b \in J_{U}^{(3)}} e (f_{\mu, \nu, \beta, a_1, a_2, t_{0}}(b)) \Big\vert\ll  & |J_{U}^{(3)}| T^{-1/12 + 2000 A_5 \delta} \\ &+ |J_U^{(3)}|^{1/2} T^{1/3 + 1/30 + 4000 A_5 \delta} (|a_1|^{2}  |a_1 - a_2| )^{-1/6}
  \end{aligned}
  $$
  We bound the length of each interval trivially by $B$.
  The contribution of the first term, summed over $a_1 \neq a_2$ is
  $$
  \ll A T^{1-1/20}
  $$
  provided that $ A_5 \delta$ is sufficiently small. On the other hand the contribution of the second term summed over $a_1 \neq a_2$ is (since $B \asymp T / A$)
  $$
    \Big(\frac{T}{A} \Big )^{1/2} \cdot T^{1/3 + 1/30 + 4000 A_5 \delta} \sum_{\substack{a_1 \neq a_2 \\ |a_1|, |a_2| \ll A}} \frac{1}{   |a_1|^ {1/3}  |a_1- a_2|^{1/6}}c  \ll A T^{7/8}
  $$
  for $A_5 \delta$ sufficiently small, and this is also sufficient.

  \subsubsection{The subcase $T^{1000 A_5 \delta} \cdot |I|^{-2} < |\nu| < T^{-1 - 1/10}$}

  We bound trivially the contribution of $b \in J_{\infty}^{(2)}$, and again this yields,
  $$
  \ll |I'| \cdot T^{-24 A_5 \delta}.
  $$
  We fix our attention on an interval $J = J_{U}^{(2)}$ for some $U \in [-50 A_5 \delta \log T, 50 A_{5} \delta \log T]$.
  We notice that on the entire interval, 
  $$
 \lambda_2 \ll T^{-1/20}
  $$
  provided that $A_5 \delta$ is sufficiently small, 
  while also
  $$
  \lambda_2 \gg |I'|^{-2}  |a_1 |  |a_1- a_2| T^{100 A_5 \delta}
  $$
  (Notice that we use here that $|I'| \gg |I| T^{-100 A_5 \delta}$.)
  Therefore, by Lemma \ref{le:Robert1} we have,
  $$
    \Big\vert \sum_{b \in J} e(f_{\mu, \nu, \beta, a_1, a_2, t_{0}}(b))   \Big\vert \ll |J| T^{-1/40} + |I'|  \big(|a_1 |  |a_1- a_2| \big)^{-1/2} T^{-100 A_5 \delta}.
  $$
  Summing the main term over $a_1 \neq a_2$ we obtain a bound of
  $$
  \ll A T^{1-1/40}.
  $$
  Summing the off-diagonal term over $a_1 \neq a_2$ we obtain a final bound of
  $$
  \ll |I'| A T^{-40 A_5 \delta}. 
  $$
  Putting all these bounds together gives the final result. 
\end{proof}

We also need an easy type-I estimate, stated below.

\begin{lemma} \label{le:TypeI}
  Let $\delta \in (0, \tfrac{1}{100})$. Let $T$ be given, $R \leq T^{10 A_5 \delta}$ with $A_{5} > 0$ an absolute constant and $I \subset [0, T]$ an interval of length $> T^{1 - A_5 \delta}$.
  Let $\alpha_a$ be arbitrary coefficients with $|\alpha_a| \leq 1$ and supported respectively in $[A, 2A]$ with $ A \ll T^{1/10}$.
  Suppose that $\mu, \nu$ and $\beta$ are such that, $$|\mu - 1| \leq 2 T^{-1 + A_5 \delta} \ , \ |I|^{-2} T^{1000 A_5 \delta} \leq |\nu| \leq 2 T^{-1 + A_5 \delta} \ , \ \beta = \frac{1}{R}\,.$$ 
  Let
  $$
  m_{\mu, \nu}(t) = \frac{\mu t}{\mu + \nu t}. 
  $$
  Then, for all $1 \leq \ell \leq T^{40 A_5 \delta}$, $t_{0} \in [0, T]$ and all $A_5 \delta$ sufficiently small
  $$
   \Big\vert \sum_{a b \in I} \alpha_a e(m_{\mu,\nu}(\ell a b - t_{0}) \beta)   \Big\vert \ll T^{1 - 20 A_5 \delta} \,.
  $$
\end{lemma}
\begin{proof}
  We split again into short intervals of length $T^{1 - 20 A_{5} \delta}$. Discarding $\ll 1$ intervals we can assume that we are located on an interval on which $|\mu + \nu(\ell a b - t_{0})| > T^{-20 A_{5} \delta}$ for all $ a b$ in the interval.
  We then cover this interval by a union of $\ll \log T$ intervals $I_{U}$ on which
  \begin{equation} \label{eq:holds}
    |\mu + \nu ( \ell a b - t_{0} ) | \asymp U
  \end{equation}
  with $U$ running over powers of two between $T^{-50 A_{5} \delta}$ and $T^{50 A_{5} \delta}$.
  As in the proof of the previous lemma we don't care about the intervals $I_{U}$ being disjoint, about their lengths, or possible reappearance of the same value $U$. It only matters that there is not too many of them. Let $J$ be one such interval on which $|\mu + \nu(\ell a b - t_{0})| \asymp U$ for all $ a b \in J$ for some $U \in [T^{-50A_{5} \delta}, T^{50 A_{5} \delta}]$.
  Let
  $$
  f(x) := m_{\mu, \nu}( \ell a x - t_{0}) \beta .
  $$
  Notice that,
  $$
  |f''(x)| \asymp  \lambda_2:=\ell^2 a^2 \mu^{2} |\nu| \vert \beta\vert\ U^{-3},.
  $$
  Therefore $\lambda_2  \leq T^{-1/4}$ on the entire interval, and moreover
  $  \lambda_2 \geq a^2 |I|^{-2} T^{100 A_5 \delta}$ on the entire interval.
  Therefore by Lemma \ref{le:Robert1} the sum over $b \in J / a$ above is bounded by
  $$
  \ll \sum_{A \leq a < 2A} \frac{|I|}{a} \Big ( T^{-50 A_5 \delta} +  a T^{-1/8} \Big )  \ll |I| T^{-21 A_5 \delta}
  $$
  as needed. This is sufficient since we sum this over at most $\ll \log T$ intervals $I_{U}$.
\end{proof}

We are now ready to prove Proposition \ref{thm:maks1} in the case when $|\nu| > T^{100 A_5 \delta} \cdot |I|^{-2}$.
\begin{proof}[Proof of Proposition \ref{thm:maks1} for $|\nu| > T^{1000 A_5 \delta} \cdot |I|^{-2}$]
  First by Lemma \ref{le:TypeII} we can assume that $p \leq T^{40 A_5 \delta}$. Now by Lemma \ref{le:Type} we are reduced to bounding type-I and type-II sums which are handled by Lemma \ref{le:TypeI} and Lemma \ref{le:TypeII}. This gives the claim. 
\end{proof}

\subsection{The case when $|\nu| \leq |I|^{-2} T^{1000 A_5 \delta}$}

In this case we can assume without loss of generality that $r > T^{10000 A_5 \delta}$. The result then following easily from the following three Lemmas.
\begin{lemma} \label{le:VinogradovTypeII}
  Let $\alpha$ be a real number such that,
  $$
  \Big | \alpha - \frac{a}{q} \Big | \leq \frac{1}{q Q}
  $$
  with $1 \leq q \leq Q$ and $(a,q) = 1$. Then, for any sequence $\alpha_a$ and $\beta_b$ with $|\alpha_a| \leq 1$ and $|\beta_b | \leq 1$, 
  $$
 \Big\vert  \sum_{\substack{m n \leq x \\ m > M, n > N}} \alpha_m \beta_n e(\alpha m n)  \Big\vert\ll \Big ( \frac{x}{M} + \frac{x}{N} + \frac{x}{q} + q \Big )^{1/2} \sqrt{x} (\log x)^2.
  $$
\end{lemma}
\begin{proof}
  See \cite[Lemma 13.8]{IwaniecKowalski}
\end{proof}

\begin{lemma} \label{le:Vinogradov}
  Let $\alpha$ be a real number such that
  $$
  \Big | \alpha - \frac{a}{q} \Big | \leq \frac{1}{q Q}
  $$
  with $(a,q) = 1$ and $1 \leq q \leq Q$. Then,
  $$
  \Big\vert  \sum_{p \leq x} e(\alpha p)  \Big\vert \ll \Big ( \sqrt{q x} + q^{-1/2} x  + x^{4/5} \Big ) (\log x)^2. 
  $$
\end{lemma}
\begin{proof}
See \cite[Theorem 13.6]{IwaniecKowalski}
\end{proof}

We are now ready to prove the result in the case when $r$ is large. 

\begin{proof}[Proof of Proposition \ref{thm:maks1} for $r$ large]
    This follows from a minor variant of Vinogradov's work on $e(\alpha p)$. Indeed, one short intervals of length $T^{1 - 2000 A_5 \delta}$ the function $m_{a,c}(t)$ can be replaced simply by the identity, that is $m_{a,c}(t) \approx t$,  with an acceptable error. By Lemma\ref{le:VinogradovTypeII} we can assume that $p \leq T^{500 A_5 \delta}$. Since $r > T^{10000 A_5 \delta}$ it follows that even in the case where $r$ is divisible by $p$, the denominator $r / p$ is still larger than $T^{500 A_5 \delta}$. Therefore applying
    Lemma~\ref{le:Vinogradov} we end up with a satisfactory saving in the sum over $q$. 
\end{proof}




\section{Appendix: Deviation of ergodic averages for $SL(2, \R)$ unipotent flows}

\subsection{Spectral decomposition of horocycle orbits}
Since~${\mathcal I}^s(S_\Gamma)\subset W^{-s}(S_\Gamma)$ is closed, there is an
orthogonal splitting
\begin{equation}
  \label{eq:orthsplitone}
  W^{-s}(S_\Gamma)={\mathcal I}^s(S_\Gamma) \oplus^{\perp} {\mathcal I}^s(S_\Gamma)^\perp\,.
\end{equation}
Although the space~${\mathcal I}^s(S_\Gamma)$ is $\{\phi^X_t\}$-invariant, the
action of the geodesic one-parameter group~$\{\phi^X_t\}$ on~$W^s(S_\Gamma)$ is 
{\it not }unitary and the orthogonal splitting~\eqref{eq:orthsplitone} is 
{\it not }$\{\phi^X_t\}$-invariant.

According to Theorems 1.1 and 1.4 of  \cite{FlaFor}, the one-parameter group
$\{\phi^X_t\}$ has a (generalized) spectral representation on the space~${\mathcal
I}^s(S_\Gamma)$. In fact, for all $s>0$, there is a $\{\phi_t^X\}$-invariant
orthogonal splitting 
\begin{equation}
  \label{eq:orthsplittwo}
  {\mathcal I}^s(S_\Gamma)={\mathcal I}^s_d \oplus^{\perp} {\mathcal I}^s_{\mathcal C}
\end{equation}
and the spectrum of $\phi_t^X$ is discrete on the subspace~${\mathcal I}^s_d:=
{\mathcal I}_d\cap {\mathcal I}^s(S_\Gamma)$ and Lebesgue of finite multiplicity with
spectral radius equal to $e^{-t/2}$ on ${\mathcal I}^s_{\mathcal C}$, for all $t\in
\R$. 

\smallskip
Let ${\mathcal B}\subset {\mathcal I}_d$ be a basis of (generalized) eigenvectors
for~$\{\phi_t^X\}\,|\,{\mathcal I}_d$ such that ${\mathcal B}\cap \bigl({\mathcal I}_d
\ominus {\mathcal I }_{1/4}\bigr)$ is a basis of eigenvectors for $\{\phi_t^X\}
\,|\,\bigl({\mathcal I}_d\ominus {\mathcal I}_{1/4}\bigr)$ and, if $1/4\in\sigma_{pp}(\square)$,
the spectrum of the Casimir operator $\square$,
the set ${\mathcal B}_{1/4}:={\mathcal B}\cap {\mathcal I}_{1/4}$ is a basis which
brings $\{\phi_t^X\}\,|\,{\mathcal I}_{1/4}$ into its Jordan normal form.

\smallskip
For any~${\mathcal D}\in {\mathcal B}\setminus {\mathcal B}_{1/4}$ of Sobolev
order~$S_{\mathcal D}>0$, there exists a complex exponent~$\lambda_{\mathcal D}\in \C$ with~$\Re
(\lambda_{\mathcal D})=-S_{\mathcal D}<0$ such that, for all~$t\in\R$,
\begin{equation}
 \label{eq:exponents}
 \phi^X_t({\mathcal D})= e^{\lambda_{\mathcal D} t}\,{\mathcal D}\,;
\end{equation}
in fact, for any Casimir parameter  $\mu= 1-\nu^2/4  \in \R^+\setminus \{1/4\}$, with $\nu \in (0, 1) \cup i \R$, 
there exists a  distributional  basis $\mathcal B_\mu = \mathcal B \cap \mathcal E'(H_\mu)= \{ \mathcal D^+_\mu, \mathcal D^-_\mu\}$ such that
$$
 \phi^X_t({\mathcal D}^\pm_\mu)= e^ { - \frac{1\pm \nu}{2}  t}\,{\mathcal D}^\pm_\mu \,;
$$
or any Casimir parameter  $\mu = 1-\nu^2/4= -n^2+n <0$, with $\nu =2n-1$ ($n\in \N\setminus \{0\}$)
there exists a  distributional  basis ${\mathcal B}_n = \mathcal{B} \cap \mathcal E'(H_\mu)= \{ {\mathcal D }_n\}$ such that
$$
 \phi^X_t({\mathcal D}^+_n)= e^ {- \frac{1+ \nu}{2}  t}\,{\mathcal D}^+_n = e^ {- n  t}\,{\mathcal D}^+_n \,;
$$
if~$1/4\in\sigma _{pp}(\square)$, the subset ${\mathcal B}_{1/4}\subset {\mathcal B}$ is
the union of a finite number of pairs $\{{\mathcal D}^+,{\mathcal D}^-\}$ such that
the distributions~${\mathcal D}^{\pm}\in {\mathcal B}^{\pm}_{1/4}={\mathcal B}\cap {\mathcal
I}^{\pm}_{1/4}$ have the same Sobolev order equal to $1/2$ and the following formula holds:
\begin{equation}
\label{eq:JBtwo}
\phi^X_t \begin{pmatrix} {\mathcal D}^+ \\ {\mathcal D}^- \end{pmatrix} \,=\,
e^{-t/2}\,\begin{pmatrix} 1 & 0\\ -\frac{t}{2} & 1 \end{pmatrix}\, 
\begin{pmatrix} {\mathcal D}^+ \\ {\mathcal D}^- \end{pmatrix}\,\,.
\end{equation} 

The set~${\mathcal B}^s:={\mathcal B}\cap {\mathcal I}^s_d$ is a basis of (generalized)
eigenvectors for the action of~$\{\phi^X_t\}$ on~${\mathcal I}^s_d$. By Theorem
1.1 of \cite{FlaFor}, for all $s>1$, there is a decomposition  
\begin{equation}
{\mathcal B}^s= \bigcup_{\mu\in \sigma_{pp}(\square)} {\mathcal B}_{\mu} \cup
\bigcup_{ 1\le n < s} {\mathcal B}_n\,\,.
\end{equation}
The operator $\phi_t^X\,|\,{\mathcal I}^s_{\mathcal C}$ has Lebesgue spectrum of finite
multiplicity supported on the circle of radius~$e^{-t/2}$ in the complex plane,
for all $t\in \R$. Its norm satisfies the following bound.

\begin{lemma} (\cite{FlaFor}, Lemma 5.1) 
\label{lemma:normbound}
For every $s>1$, there exists a constant $C_1:=C_1(s)>0$ such that, for all $t\in \R$,
\begin{equation}
\label{eq:normbound}
\| \phi^X_t \bigm\vert \,{\mathcal I}^s_{\mathcal C}\|_{-s}
\le C_1\, (1+|t|) \,e^{-t/2}\,\,.
\end{equation} 
\end{lemma}

According to~\eqref{eq:orthsplitone} and~\eqref{eq:orthsplittwo}, every
$\gamma \in W^{-s}(S_\Gamma)$ can be written as

\begin{equation}
\label{eq:gammasplit}
  \gamma= \sum_{{\mathcal D}\in {\mathcal B}^s} c_{{\mathcal D}}(\gamma)
{\mathcal D} + {\mathcal C}(\gamma) +{\mathcal R}(\gamma)
\end{equation}
with~${\mathcal C}(\gamma)\in {\mathcal I}^s_{\mathcal C}$ and~${\mathcal R}
(\gamma)\in {\mathcal I}^s(S_\Gamma)^{\perp}$. The real number~$c_{\mathcal D}
(\gamma)$ will be called the $\mathcal D${\it -component} of~$\gamma$
along~${\mathcal D}\in {\mathcal B}^s$ and the distribution~${\mathcal C}(\gamma)$ the
($U$-invariant) {\it continuous component} of~$\gamma$. We recall that
the continuous component vanishes for all~$\gamma\in W^{-s}(S_\Gamma)$ if
$S_\Gamma$ is compact. The following Lemma tells us that bounds on the norms
of distributions in $W^{-s}(S_\Gamma)$ are equivalents to bounds on their
coefficients.
\begin{lemma} (\cite{FlaFor}, Lemma 5.2)
\label{lemma:bounddist}
There exists a constant $C_2:=C_2(s)>0$ such that, for evert Casimir parameter $\mu >0$,
$$
C_2^{-2} \|\gamma \vert W^{-s}(H_\mu)\|^2 \le  \sum_{{\mathcal D}\in {\mathcal B}^s_\mu}
|c_{\mathcal D}(\gamma)|^2 \,+\,  \|{\mathcal R}(\gamma) \vert  W^{-s}(H_\mu) \|^2 \le C_2^2\, \|\gamma\vert W^{-s}(H_\mu)\|^2\,,
$$
hence in particular 
\begin{equation}
\label{eq:estcoeff}
C_2^{-2} \|\gamma\|^2_{-s} \le \sum_{{\mathcal D}\in {\mathcal B}^s}
|c_{\mathcal D}(\gamma)|^2  \,+\, \|{\mathcal C}(\gamma)\|^2_{-s}\,+\,
\|{\mathcal R}(\gamma)\|^2_{-s} \le C_2^2\, \|\gamma\|^2_{-s} \,.
\end{equation}
\end{lemma}
\begin{proof}
  The splittings \eqref{eq:orthsplitone} and~\eqref{eq:orthsplittwo} are
  orthogonal with respect to the Hilbert structure of $W^{-s}(S_\Gamma)$.
  The basis~${\mathcal B}^s$ is not orthogonal, however we claim that its
  distortion is uniformly bounded. In fact, vectors of the basis
  supported on different irreducible representations are orthogonal;
  if~${\mathcal D}^+_{\mu}$, ${\mathcal D}^-_{\mu} \in {\mathcal B}^s_\mu$ are normalized
  eigenvectors supported on the same irreducible representation of
  Casimir parameter~$\mu\in\R^+$ (principal or complementary series),
  a calculation shows that the function $\langle{\mathcal D}^+_{\mu}, {\mathcal
    D}^-_{\mu}\rangle_{-s}$ is continuous on the open set $\R^+\setminus \{1/4\}$,
  it converges to~$0$ as $\mu \to +\infty$ and to~$1$ as~$\mu\to 1/4$.
  Since~${\mathcal I}^s_d$ is contained in the pure point component of the
  the spectral representation of the Casimir operator, the angle
  between ${\mathcal D}^+_{\mu}$ and~${\mathcal D}^-_{\mu}$ has a strictly
  positive uniform lower bound for ~$\mu\in \sigma(\square)$.
  \relax
\end{proof}

\subsection{Horocycle orbits}
For $x\in S_\Gamma$ and $T\in \R^+$, let~$\gamma_{x,T}$ be the probability
measure uniformly distributed on the horocycle orbit of length $T$
starting at $x$. More precisely, for any continuous function~$f$
on~$S_\Gamma$, we define
$$
\gamma_{x,T}(f)= \frac{1}{T} \int_0^T f(h_t(x))\,dt
$$

By the Sobolev embedding Theorem (see \cite{Ad}), for~$s>3/2$,
the measures $\gamma_{x,T} $ are continuous functionals on~$W^s(S_\Gamma)$
(which depend weakly-continuously on~$x\in S_\Gamma$ and~$T\in\R^+$).  Thus
the splitting~\eqref{eq:gammasplit} can be applied to horocycle
orbits.  We set
$$c_{\mathcal D}(x,T):=c_{\mathcal D}(\gamma_{x,T}),\qquad
{\mathcal C}(x,T):={\mathcal C}(\gamma_{x,T})\,,\qquad
{\mathcal R}(x,T):={\mathcal R}(\gamma_{x,T})\,.
$$
so that
\begin{equation}
\label{eq:gammasplit2}
\gamma_{x,T}  = \sum_{{\mathcal D}\in {\mathcal B}^s} c_{\mathcal D}(x,T)
{\mathcal D} + {\mathcal C}(x,T) +{\mathcal R}(x,T)\,.
\end{equation}
Following~\cite{FlaFor}, the proof of Theorem~\ref{thm:SL2_equi} will be reduced to
estimates on the norms of the three parts of this splitting. We start
by showing in next section that since the parts of this splitting
invariant by the horocycle flow, namely $ \sum_{{\mathcal D}\in {\mathcal B}^s}
c_{\mathcal D}(x,T) {\mathcal D}$ and ${\mathcal C}(x,T)$, vanish on coboundaries,
the remainder part ${\mathcal R}(x,T)$ must be of the order of $1/T$;
furthermore the individual coefficients $c_{\mathcal D}(x,T)$ and the
continuous component ${\mathcal C}(x,T)$ cannot be too small.

The uniform norm of functions on a compact manifold can be bounded in terms 
of a Sobolev norm by the Sobolev embedding theorem. In the case of
a non-compact hyperbolic surfaces $M$ of finite area, since the injectivity
radius is not bounded away from zero, the Sobolev embedding theorem
holds only locally. We therefore prove a version of the Sobolev
embedding theorem on compact subsets of the unit tangent bundle $S_\Gamma$, 
with an explicit bound on the constant.

\subsection{Sobolev embedding}

The following Lemma is a version of Lemma 5.3 of
\cite{FlaFor} (see also Lemma 2.1 in \cite{Strom13} or Prop. B.2 in \cite{BR})  rewritten to carefully keep track of the dependence of the constants on the lattice.

\begin{lemma} 
  \label{lemma:Sobembed_1}
  There exists a universal constant $C>0$ such that for any
  function $F\in W^2(S_\Gamma)$, we have that $F$ is continuous and, for all $x \in S_\Gamma$,
  $$
  \begin{cases}
    |F(x)| \le C \,  \text{\rm inj}_\Gamma^{-1}   \| F\|_{W^2(S_\Gamma)} \,,  \quad &\text{ if }  \pi(x) \in M_{cpt}\,, \\ 
  |F(x)| \le C \, e^{d_\Gamma(x)/2}  \| F\|_{W^2(S_\Gamma)} \,,  \quad &\text{ if }  \pi(x) \in M_{cusp} \,.
  \end{cases} 
  $$
\end{lemma}

\begin{proof} 
  Recall that if $G$ is a locally $W^2$ function on Poincar\'e's plane
  $H$ then $G$ is continuous and there exists a universal constant  $C>0$ such
  that, for all $\epsilon \in (0, 1)$ we have 
  $$
  |G(z)|^2 <  \frac{C}{\epsilon} \int_{B(z,\epsilon)} ( |G(w)|^2 + |dG(w)|^2 +
  |\Delta G(w)|^2 ) \, dw
  $$
  for any $z\in H$ (\cite{He} page 63).  Indeed, the dependence of the constant 
  on the radius of the ball can be determined by a scaling argument or by examining the 
  proof of Theorem 3.4 in \cite{He}.   Indeed, for every function $G$ on the Poincar\'e disk
  let $G_\epsilon (z) :=  G(\epsilon z)$.  There exists a universal constant $C'>0$ such that
  $$
  |G(0)|^2 = |G_\epsilon(0)|^2 <  C' \int_{B(0,1)} ( |G_\epsilon(w)|^2 + |dG_\epsilon(w)|^2 +
  |\Delta G_\epsilon(w)|^2 ) \, dy
  $$
  An immediate computation by change of variables establishes that there exists
  a universal constant $C'' >0$ such that
  $$
  \int_{B(0,1)} ( |G_\epsilon(w)|^2 + |dG_\epsilon(w)|^2 +
  |\Delta G_\epsilon(w)|^2 ) \, dy \leq   \epsilon^{-2} \int_{B(0,1)} ( |G(w)|^2 + |dG(w)|^2 +
  |\Delta G(w)|^2 ) \, dy \,.
  $$
  Let $SB(z, \epsilon)$ the unit tangent bundle of $B(z, \epsilon)$. A similar argument
  gives that for any $G$ locally $W^2$ function on the unit tangent bundle $SH$ of the Poincar\'e
  plane $H$, $G$ is continuous and there exists a universal constant  $C>0$ such
  that, for all $\epsilon \in (0, 1)$ and  for any $x\in SH$, we have 
  $$
  |G(x)|^2 <  \frac{C}{\epsilon} \int_{B(z,\epsilon)} ( |G(y)|^2 + |dG(y)|^2 +
  |\Delta G(y)|^2 ) \, dy\,.
  $$
  For~$p$ in~$M_\Gamma$ denote by~$\rho_\Gamma(p)$ the radius of injectivity
  of~$M_\Gamma$ at~$p$. Let $\pi: S_\Gamma \to M_\Gamma$ the projection
  defined as $\pi(x) = SO(2, \R) x  \in   SO(2, \R)\backslash SL(2, \R)/\Gamma$.
  
 \noindent  Let $\epsilon_0>0$ denote the Margulis constant of the Poincar\'e plane. 
   Let $\epsilon:= \epsilon_0/2$ and set $A_0$ the open set of points $x \in S_\Gamma$
  where $\rho(\pi(x)\geq \epsilon$. Hence the complement $A_0^c$ consists of the union of 
  of~$k$ connected components $V_i$ each contained in $SA_i$, the tangent unit bundle
  of disjoint open cusps $A_i\approx S^1 \times \R^+$ whose boundary horocycle has length 
  $2\epsilon=\epsilon_0$ and of~$h$ connected components $T_j$ which are tubular neighborhoods of geodesic
  of length less than $\epsilon$ (Margulis tubes).

 \noindent    By the Sobolev embedding theorem mentioned above there exists
  $C >0$ such that for any $x\in A_0$ we have
  $$
  |F(x)|^2 <  \frac{2C}{\epsilon_0} \int_{ S_\Gamma \vert B(\pi(x),\epsilon)} ( |F|^2 + |dF|^2 +
  |\Delta F|^2 ) \, d \hat y  \leq  \frac{2C}{\epsilon_0}  \Vert F \Vert^2_{W^2(S_\Gamma)}.
  $$ 
  For $x \in  T_j$, let $\epsilon := \inj_\Gamma$ be the injectivity radius of the compact part. By the
  Sobolev embedding theorem we then have, we then have 
  $$
  |F(x)|^2 <  \frac{C}{\epsilon} \int_{ S_\Gamma \vert B(\pi(x),\epsilon)} ( |F|^2 + |dF|^2 +
  |\Delta F|^2 ) \, d \hat y  \leq  \frac{C}{\inj_\Gamma}  \Vert F \Vert^2_{W^2(S_\Gamma)}.
  $$ 
For $x\in V_i$ let $d$ be the distance of $x$ from $\partial A_i$. We note that
$d = d_\Gamma(x)$, the distance of $x$ from the compact part $S_\Gamma\vert M_{\rm cpt}$
of $S_\Gamma$. It's easy to see that $\epsilon_0 e^{-d} \le \rho(x) \le 2 \epsilon_0
  e^{-d}$.  Let $\tilde F$ denote the lift of $F$ to Poincare's
  half-plane~$H$ and let $\tilde x$ be a point $SH$ projecting to $x$.
  Then, by the same embedding theorem, with $\epsilon = \epsilon_0/2$, 
  $$
  |\tilde F(\tilde x)|^2 <  \frac{C}{\epsilon} \int_{B(\tilde x,\epsilon)} (
  |\tilde F|^2 + |d \tilde F|^2 + |\Delta \tilde F|^2 ) \, d\tilde y
  $$
  and, since, the ball $B(\tilde x,\epsilon) \subset SH$ covers the
  ball $B( x, \epsilon)\subset S_\Gamma$ at most $[e^{d}/2]+1$ times, we get
  $$
  |F(x)|^2=|\tilde F(\tilde x)|^2 < e^{d} \Big(\frac{C}{\epsilon} \Big)\int_{B(x,\epsilon)} ( |F|^2 +
  |dF|^2 + |\Delta F|^2 ) \, dy \leq  \frac{C}{\epsilon} e^{d} \Vert F \Vert^2_{W^2(S_\Gamma)}\,.
  $$
The proof is complete.
\end{proof}

\noindent A function on $S_\Gamma$ is called {\em cuspidal}  if it has zero average along all translate of (cuspidal) closed (periodic) horocycle orbits. 

\begin{lemma} 
  \label{lemma:Sobembed_2}
  There exists a universal constant $C>0$ such that for any
 cuspidal function $F\in W^3(S_\Gamma)$, we have that $F$ is continuous and, for all $x \in S_\Gamma$
 such that $\pi(x) \in M_{\rm cusp}$, 
  $$
  |F(x)| \le C \, e^{-d_\Gamma(x)/2} \| F\|_{W^3(S_\Gamma)}\,.
  $$
\end{lemma}
\begin{proof} This a version of Lemma 2.2 in \cite{Strom13} (which in turn follows Prop. 4.1 in 
\cite{BR}).  Each cusp $\mathcal A$ is diffeomorphic to a semi-infinite cylinder $ S^1 \times \R^+$ with boundary a closed
horocycle of length $\epsilon_0>0$.  After a conjugation we may assume that the cusp is in canonical form,
that is,  $\mathcal A = \{ z \in \C \vert   \text{Im}(z) \geq  \epsilon_0^{-1}\} / \Gamma_\infty$ with $\Gamma_\infty <SL(2, \Z)$
the cyclic group generated by a upper triangular Jordan block. 

\smallskip
\noindent  Let $X$, $U$ and $\Theta$ denote the generators of the geodesic flow, of the stable horocycle flow and 
 of the one-parameter rotation group $SO(2, \R)$ respectively.  The unit tangent bundle $S\mathcal A$ over $\mathcal A$ can then
 be parametrized by a map
$$
(t,u, \theta) \to   \exp (\theta \Theta) \exp (t X) \exp (u U)  \Gamma_\infty \,,   \quad (t,u,\theta) \in [\epsilon_0^{-1}, +\infty) \times \R \times S^1\,.
$$
The condition that $F$ is cuspidal means
\begin{equation}
\label{eq:cuspidal}
\int_0^1 F (\exp (\theta \Theta) \exp (t X) \exp (u U) \Gamma_\infty) du  =0  \,, \quad \text{ for all }  (t, \theta) \in \R^+ \times S^1\,.
\end{equation}
Let $x = \exp (\theta_0 \Theta) \exp (t_0 X) \exp (u_0 U) \Gamma_\infty$ with $u_0 \in [0, e^{-t}]$.  Then there exists 
$u_* \in [0, 1]$ such that  $F(\exp (\theta \Theta) \exp (t X) \exp (u_* U) \Gamma_\infty)=0$, hence
$$
F(x) = \int_{u_*}^{u_0} \frac{d}{du} F \Big(\exp (\theta_0 \Theta) \exp (t_0 X) \exp (u U) \Gamma_\infty\Big) du \,.
$$
Since
$$
 \frac{d}{du} F \Big(\exp (\theta_0 \Theta) \exp (t_0 X) \exp (u U) \Gamma_\infty \Big)  = $$$$ e^{-t_0}  
 \text{Ad}_{\exp(\theta_0 \Theta)} (U)F \Big( \exp (\theta_0 \Theta)   \exp (t_0 X)  \exp (u U) \Gamma_\infty\Big) 
$$
and, by Lemma  ~\ref{lemma:Sobembed_1}, 
$$
\vert  \text{Ad}_{\exp(\theta_0 \Theta)} (U) F \Big( \exp (\theta_0 \Theta)   \exp (t_0 X) \Gamma_\infty\Big)  \vert \leq
C e^{d_\Gamma(x)/2}  \Vert  \text{Ad}_{\exp(\theta_0 \Theta)} (U) F \Vert_{W^2(S_\Gamma)}\,,
$$
it follows that, since $d_\Gamma(x) \leq  t_0$, 
$$
\vert F(x) \vert \leq C e^{-d_\Gamma(x)/2} \Vert  F  \Vert_{W^3(S_\Gamma)} \,.
$$
\end{proof} 

\begin{lemma} 
  \label{lemma:Sobembed_3}
  There exists a universal constant $C>0$ such that for any function $F\in W^3(S_\Gamma)$, which belong to a complementary series component of Casimir parameter $\mu(\nu):=(1-\nu^2)/4$ with $\nu\in (0,1)$,  we have that $F$ is continuous and, for all $x \in S_\Gamma$ such that $\pi(x) \in M_{\rm cusp}$, 
  $$
  |F(x)| \le C \, e^{ \frac{1-\nu}{2} d_\Gamma(x)} \| F\|_{W^3(S_\Gamma)}\,.
  $$
\end{lemma}
\begin{proof} This a version of Lemma 2.3 in \cite{Strom13}.  As in the proof of Lemma~\ref{lemma:Sobembed_2} the argument is based on the remark that each cusp $\mathcal A$ is diffeomorphic to a semi-infinite cylinder $S^1 \times \R^+$ with boundary a closed
horocycle of length $\epsilon_0>0$ and that after a conjugation we may assume that the cusp is in canonical form. 
Each complementary series component $H_\mu$ has a orthonormal basis $\{u_n\}$ of eigenfunctions of the action of the circle group $SO(2,\R)$.  Since the Casimir operator $\square = -X^2  +X - \Theta U + U^2$ we have
$$
\square u_n =  \mu(\nu)  u_n \quad  \text {and} \quad  \Theta u_n = i n u_n \,, \quad \text{ for all } n \in \Z\,.
$$
We consider the functions
$$
\phi_n(g\Gamma_\infty) = \int_0^1   u_n ( g \exp (u U) \Gamma_\infty) du \,, \quad \text{ for all } n \in \Z.
$$
Since $\phi_n$ is by definition invariant under the horocycle flow,  and since
$$
\begin{aligned}
\phi_n( \exp(\theta \Theta) g\Gamma_\infty) &=  \int_0^1   u_n (  \exp(\theta \Theta) g \exp (u U) \Gamma_\infty) du  
\\ &=   i n \int_0^1   u_n ( g \exp (u U) \Gamma_\infty) du   = i n \phi_n  (g\Gamma_\infty) \,,
\end{aligned}
$$
it follows that
$$
\begin{aligned}
(\square \phi_n) (\exp(\theta\Theta) &\exp (t X) \exp (uU) \Gamma_\infty ) 
\\ &= (- \frac{d^2}{dt^2} + \frac{d}{dt})    \phi_n (\exp(\theta\Theta) \exp (t X) \exp (uU) \Gamma_\infty ) \\ &= \mu(\nu) \phi_n (\exp(\theta\Theta) \exp (t X) \exp (uU) \Gamma_\infty ) \,,
\end{aligned}
$$
which in turn, since the Casimir parameter is given by the identity $\mu(\nu):= (1-\nu^2)/4$, implies that there exist constants $C_n, C'_n \in \R$ such that
$$
\phi_n (\exp(\theta\Theta) \exp (t X) \exp (uU) \Gamma_\infty) = C_n e^{\frac{1+\nu}{2} t} + C'_n e^{\frac{1-\nu}{2}t}
$$
Since the basis $\{u_n\}$ is orthonormal we have
$$
\begin{aligned}
1&\geq \int_{S \mathcal A} \vert  u_n \vert ^2 d\text{vol} = \int_{0}^{2 \pi} \int_{\epsilon_0^{-1}} ^{+\infty}   \vert \phi_n (\exp(\theta\Theta) \exp (t X) \exp (uU) \Gamma_\infty)  \vert ^2  e^{-t} d\theta dt \\ & =  2\pi \int_{\epsilon_0^{-1}} ^{+\infty}   \vert \phi_n (\exp(\theta\Theta) \exp (t X) \exp (uU) \Gamma_\infty)  \vert ^2 e^{-t} dt = 2\pi  \int_{\epsilon_0^{-1}} ^{+\infty} ( C_n e^{\frac{1+\nu}{2} t} + C'_n e^{\frac{1-\nu}{2} t})^2  e^{-t} dt\,,
\end{aligned}
$$
hence, by taking into account that $\nu\in (0, 1)$, it follows that $C_n=0$ and $C'_n$ is bounded above, uniformly with respect to
$n\in \Z$, by a universal constant  $C>0$.

\smallskip
\noindent For any smooth function $F$ on $S\mathcal A$ we can write
$$
\int_0^1 F (\exp (\theta \Theta) \exp (t X) \exp (u U) \Gamma_\infty) du  =  \sum_{n\in\Z}  \langle F, u_n\rangle  \phi_n(\exp (\theta \Theta) \exp (t X) \Gamma_\infty) du \,,
$$
hence we conclude that there exists a universal constant $C''>0$ such that 
\begin{equation}
\label{eq:cuspidal_bound}
\begin{aligned}
\big\vert  &\int_0^1 F (\exp (\theta \Theta) \exp (t X) \exp (u U) \Gamma_\infty) du   \big\vert \\ & \leq  C e^{\frac{1-\nu}{2}t} 
\Big(\sum_{n\in \Z} (1+n^2)^{-1} \Big)^{1/2} \Big(\sum_{n\in \Z}   (1+n^2) \vert \langle F, u_n\rangle  \vert^2  \Big)^{1/2} 
\leq   C'' C e^{\frac{1-\nu}{2} t}  \Vert F \Vert_{W^1(S\mathcal A)} \,.
\end{aligned}
\end{equation}
At this point the argument proceeds exactly as in the the proof the previous Lemma  with the bound in formula
\eqref{eq:cuspidal_bound} in place of that in formula \eqref{eq:cuspidal}.  Since there exists a universal constant
$C^{(3)} >0$ such that 
$$
t \leq  C^{(3)} d_\Gamma (\exp (\theta \Theta) \exp (t X) \exp (u U) \Gamma_\infty) 
$$
the result follows.
\end{proof} 

Lemmas~\ref{lemma:Sobembed_1}, ~\ref{lemma:Sobembed_2} and ~\ref{lemma:Sobembed_3} allow 
us to derive the following upper bound for the uniform norm of components and remainder terms of horocycle
arcs. For every Casimir parameter $\mu \in \R$, let $H_\mu$ denote the isotypical 
component of  $L^2(S_\Gamma)$ and $W^s(H_\mu)$ for $s>0$ the corresponding
weighted Sobolv spaces.  Let $H_o$ be the component given by cuspidal functions, 
and $W^s(H_o)$ for $s>0$ the corresponding weighted Sobolev spaces

 Let 
$$
 {\mathcal B}^s _o := {\mathcal B}^s \cap W^s(H_o)   \quad \text{ and }   \quad   {\mathcal B}^s_\mu  :=   {\mathcal B}^s \cap W^s(H_\mu)  \,.
$$

\begin{corollary} (see \cite{FlaFor}, Corollary 5.4)
\label{cor:unifupperb}
For all $s\ge 2$, there exists a constant $C_4:=C_4(s)>0$ such that the following holds. Let~$\gamma_{x,T}$ denote
the horocycle arc with endpoints $x$ and $h_T(x)$, then 
\begin{equation}
\label{eq:unifupperb_1}
\begin{aligned}
\sum_{{\mathcal D}\in {\mathcal B}^s}
|c_{\mathcal D}(x,T)|^2  & \,+\, \|{\mathcal C}(x,T)\|^2_{-s}\,+\,
\|{\mathcal R}(x,T)\|^2_{-s}  \\  & \le C_4^2\, \max \{ \inj_\Gamma^{-1}, \max_{y\in \gamma_{x,T}} e^{d_\Gamma (y)/2} \}^2\,.
\end{aligned}
\end{equation}
For all $s \geq 2$ and for all Casimir parameters $\mu:=\mu(\nu) \in (0, 1/4)$ and for cuspidal components, we have 
\begin{equation}
\label{eq:unifupperb_2}
\begin{aligned}
&\sum_{{\mathcal D}\in {\mathcal B}^s_\mu}
|c_{\mathcal D}(x,T)|^2   \le C_4^2\, \max \{ \inj_\Gamma^{-1}, \max_{y\in \gamma_{x,T}} e^{\frac{1-\nu}{2} d_\Gamma (y)} \} ^2\,; \\
&\sum_{{\mathcal D}\in {\mathcal B}^s_o}
|c_{\mathcal D}(x,T)|^2   \le C_4^2\, \max \{ \inj_\Gamma^{-1}, \max_{y\in \gamma_{x,T}} e^{-d_\Gamma (y)/2} \}^2 \,; \\
\end{aligned}
\end{equation} 
\end{corollary}

\begin{proof}

By definition for $s\geq 2$ and  for any function~$f\in  W^s(S_\Gamma)$
$$
|\gamma_{x,T}(f)| \leq  \max \{ |f(y)|\,:\, y  \in \gamma_{x,T}(f)\}  \,,
$$
hence by Lemmas ~\ref{lemma:Sobembed_1}, for all $s\geq 2$, 
$$
\Vert \gamma_{x,T}  \Vert_{-s}  \le C_3\, \max \{ \inj_\Gamma^{-1}, \max_{y\in \gamma_{x,T}} e^{d_\Gamma (y)/2} \}\,.
$$
The estimate~\eqref{eq:unifupperb_1} then follows from Lemma~\ref
{lemma:bounddist}.  

\smallskip
The estimates in formula \eqref{eq:unifupperb_2} follows from Lemmas ~\ref{lemma:Sobembed_2} 
and ~\ref{lemma:Sobembed_3} since for every $s\geq 2$ and  for every Casimir parameter $\mu \in \sigma(\square)$,
$$
\sum_{{\mathcal D}\in {\mathcal B}^s_\mu}
|c_{\mathcal D}(x,T)|^2 \leq   \Vert  \gamma_{x,T}  \vert  W^{-s}(H_\mu) \Vert^2  = \sup_{f\in W^s(H_\mu)\setminus\{0\}} \Big( \frac{\vert \gamma_{x,T}  (f) \vert} { \Vert f \Vert_s } \Big)^2
$$
and, similarly,
$$
\sum_{{\mathcal D}\in {\mathcal B}^s_o}
|c_{\mathcal D}(x,T)|^2 \leq   \Vert   \gamma_{x,T}  \vert  W^{-s}(H_o) \Vert^2 = \sup_{f\in W^s(H_o)\setminus\{0\}} \Big( \frac{\vert \gamma_{x,T} (f) \vert} { \Vert f \Vert_s } \Big)^2 \,.
$$
The lemma is therefore proved.
\end{proof}

\subsection{Coboundaries}
Let $\{\phi_t\}$ be a measure preserving ergodic flow on a probability
space. We recall that a function $g$ is a {\it coboundary } for 
$\{\phi_t\}$ if it is a derivative of a function $f$ along this flow. The
Gottschalk-Hedlund Theorem, or rather its proof, yields upper bounds
for the uniform or the $L^2$ norm of ergodic averages of a coboundary
$g$ in terms of the uniform, or respectively, the $L^2$ norm of its
primitive $f$. A key consequence is that the uniform bound for the
remainder term ${\mathcal R}(x,T)$ proved in Corollary~\ref{cor:unifupperb}
can be significantly improved.

\begin{lemma} (see \cite{FlaFor}, Lemma 5.5)
  \label{lemma:remainder}
For every~$s>3$, there exists a constant~$C_5:=C_5(s)$ such that the following holds. Let~$\gamma_{x,T}$ denote
the horocycle arc with endpoints $x$ and $h_T(x)$,then 
  \begin{equation}
    \label{eq:unifRbound}
    \| {\mathcal R}(x,T)\|_{-s} \le \frac{C_5}{T} \,  \max \{ \inj_\Gamma^{-1}, e^{d_\Gamma (x)/2},  e^{d_\Gamma (h_T(x))/2} \}\,.
  \end{equation}
For every $s >4$ and for all Casimir parameters $\mu:=\mu(\nu) \in (0, 1)$ and for cuspidal components, we have 
\begin{equation}
\label{eq:unifupperb_2}
\begin{aligned}
&   \| {\mathcal R}(x,T) \vert W^{-s}(H_\mu)\|  \le \frac{C_5}{T} \, \max \{ \inj_\Gamma^{-1}, e^{\frac{1-\nu}{2} d_\Gamma (x)},
e^{\frac{1-\nu}{2} d_\Gamma (h_T(x))}  \} \,; \\
& \| {\mathcal R}(x,T) \vert W^{-s}(H_o)\|   \le  \frac{C_5}{T}\, \max \{ \inj_\Gamma^{-1}, e^{-d_\Gamma (x)/2},  e^{-d_\Gamma (h_T(x))/2}  \} \,.
\end{aligned}
\end{equation}

\end{lemma}

\begin{proof}
  Let~${\mathcal I}:={\mathcal I}^s(S_\Gamma)$. The orthogonal
  splitting~\eqref{eq:orthsplitone} induces a dual orthogonal splitting
  \begin{equation}
    \label{eq:dualsplit}
    W^s(S_\Gamma)= \text{Ann}(\mathcal I) \oplus \text{Ann}({\mathcal I}^\perp)\,.
  \end{equation}
  Hence, any function~$g\in W^s(S_\Gamma)$ has a unique (orthogonal)
  decomposition~$g = g_1 + g_2$, where~$g_1\in \text{Ann}(\mathcal I)$
  and~$g_2 \in \text{Ann}({\mathcal I}^\perp)$. Since ${\mathcal R}(x,T)\in
  {\mathcal I}^\perp $, the function~$g_2 \in \text{Ann}({\mathcal I}^\perp)$
  and~$g_1\in \text{Ann}(\mathcal I)$, we have:
  \begin{equation}
    \label{eq:remainder}
    {\mathcal R}(x,T)(g)= {\mathcal R}(x,T)(g_1+g_2)= {\mathcal R}(x,T)(g_1)
    = \gamma_{x,T}(g_1)\,.
  \end{equation}
  
  The function~$g_1$ is a coboundary for the horocycle flow. In fact,
  it belongs to the kernel of all invariant distributions for the horocycle flow $h_\R$ of
  order~$\le s$; hence, by Theorem 1.2 of \cite{FlaFor},  there exists a
  function~$f_1\in W^t (S_\Gamma)$, with~$2<t<s-1$, solution of the cohomological
  equation 
  \begin{equation}
  \label{eq:CE_remainder}
   \frac{d}{dt}  (f_1 \circ h_t)  =g_1\circ h_t \,,
  \end{equation}
 such that~$\|f_1\|_t\le C \|g_1\|_s$.  Let~$d>0$ be such that $x$,
  $h_T(x)\in {\overline B}(x_0,d)$. By the Sobolev embedding
  Theorem, the function~$f_1$ is continuous and by
  Lemma~\ref{lemma:Sobembed_1}
  \begin{equation}
    \label{eq:Sobembed_1}
    \max \{|f_1(x)|, |f_1\big( h_T(x) \big)|  \}   \leq   C \max \{ \inj_\Gamma^{-1}, e^{d_\Gamma (x)/2}, e^{d_\Gamma (h_T(x))/2} \}  \| g_1\|_s\,.
  \end{equation}
  By the Gottschalk-Hedlund argument and the
  inequality~\eqref{eq:Sobembed_1},
  \begin{equation}
    \label{eq:cobest}
    |\gamma_{x,T}(g_1)|= \frac{1}{T}|f_1\circ h_T(x)-f_1(x)| \le 
    \frac{2C'_3}{T} \, \max \{ \inj_\Gamma^{-1}, e^{d_\Gamma (x)/2}, e^{d_\Gamma (h_T(x))/2} \}  \,\| g_1\|_s\,.
  \end{equation}
  Since the dual splitting~\eqref{eq:dualsplit} is orthogonal, by the
  estimates~\eqref{eq:remainder} and~\eqref{eq:cobest}, we get
  $$
  \begin{aligned}
  |{\mathcal R}(x,T)(g)|&\le \,\frac{2C'_3 }{T}\,  \max \{ \inj_\Gamma^{-1}, e^{d_\Gamma (x)/2}, e^{d_\Gamma (h_T(x))/2} \}   \| g_1\|_s
  \\ & \le
  \frac{2C'} {T}\,\max \{ \inj_\Gamma^{-1}, e^{d_\Gamma (x)/2}, e^{d_\Gamma (h_T(x))/2} \} \,\| g\|_s\,.
  \end{aligned} 
  $$
  An analogous argument holds for the projections of the distribution ${\mathcal R}(x,T)$ on cuspidal components or on irreducible sub-representations of the complementary series.  In fact, whenever  $g \in W^s(H_o)$ for $s>4$,  then the solution $f_1$ of the cohomological equation in formula \eqref{eq:CE_remainder}  is such that $f_1 \in W^3(H_o)$, hence by Lemma \ref{lemma:Sobembed_2}, we have
  \begin{equation}
    \label{eq:Sobembed_2}
    \max \{|f_1(x)|, |f_1\big( h_T(x) \big)|  \}   \leq   C \max \{ \inj_\Gamma^{-1}, e^{-d_\Gamma (x)/2}, e^{-d_\Gamma (h_T(x))/2} \}  \| g_1\|_s\,;
  \end{equation}
  whenever $g \in W^s(H_\mu)$ for $s>4$ and $\mu:=\mu(\nu) \in (0,1/4)$,  then the solution $f_1$ of the cohomological equation in formula \eqref{eq:CE_remainder}  is such that $f_1 \in W^3(H_\mu)$, hence by Lemma \ref{lemma:Sobembed_3}, we have 
  \begin{equation}
    \label{eq:Sobembed_2}
    \max \{|f_1(x)|, |f_1\big( h_T(x) \big)|  \}   \leq   C \max \{ \inj_\Gamma^{-1}, e^{\frac{1-\nu}{2}d_\Gamma (x)}, 
    e^{ \frac{1-\nu}{2} d_\Gamma (h_T(x))/2} \}  \| g_1\|_s\,.
  \end{equation}
The proof in the latter cases can then be completed by the Gottschalk-Hedlund argument and orthogonality as above. 
The lemma is therefore proved.
\end{proof}

\subsection{Iterative estimates}
Let $\{X, U, V\}$  denote the generators of the Lie algebra 
 $\mathfrak{sl}(2, \R)$, respectively, satisfying the commutations relations
 $$
 [X, U]=U\,, \quad  [X, V]=-V\,, \quad [U, V]= 2X \,.
 $$
By the commutation relations,  the geodesic
flow~$\{\phi^X\}$ expands the orbits of unstable horocycle
flow~$\{\phi^V_s\}$ by a factor~$e^t$ and it contracts the orbits of
stable horocycle flow~$\{h_s\}:= \{\phi^U_s\}$ by a factor~$e^{-t}$:
\begin{equation}\label{eq:gaone}
  \phi^X_t\circ\phi^U_s= \phi^U_{se^{-t}}\circ\phi^X_t,\qquad
   \phi^X_t\circ\phi^V_s=\phi^V_{se^t}\circ\phi^X_t\,.
\end{equation}
It follows that, in the distributional sense,
\begin{equation}
\label{eq:gatwo}
\phi^X_t(\gamma_{x,T})= \gamma_{\phi^X_{-t}(x), e^t\,T}\,.
\end{equation}

Let $x\in S_\Gamma$, $T>0$. It will be convenient to discretize the geodesic
flow time~$t\ge 1$ and to consider the push-forwards of the arc $\gamma_{x,T}$ 
by $\phi^X_{\ell h}$, where $h\in [1,2]$ and~$\ell\in \N$. Then the 
distribution (measure)~$\phi^X_{\ell h}(\gamma_{x,T})$ has a splitting
\begin{equation}
\label{eq:pushforwsplit}
  \phi^X_{\ell h}(\gamma_{x,T})=\sum_{{\mathcal D}\in {\mathcal B}^s}
  c_{\mathcal D}(x,T,\ell){\mathcal D} \,+ \,{\mathcal C}(x,T,\ell)\,
  +\,{\mathcal R}(x,T,\ell)\,.
\end{equation}
We prove below pointwise upper bounds on the sequences of
functions~$c_{\mathcal D}(\cdot,T,\ell)$, ${\mathcal C} (\cdot,T,\ell)$ and
${\mathcal R}_(\cdot,T,\ell)$.  By the identity~\eqref{eq:gatwo} and the
definition~\eqref{eq:gammasplit2}, we have:
\begin{equation}
\label{eq:coeffell}
\begin{aligned}
  c_{\mathcal D}(x,T,\ell)&=c_{\mathcal D}\left(\phi^X_{-\ell h}(x), e^{\ell
      h}\,T\right)\,,
  \\
  {\mathcal C}(x,T,\ell)&={\mathcal C} \left(\phi^X_{-\ell h}(x),e^{\ell
      h}\,T\right)\,,
  \\
  {\mathcal R}(x,T,\ell)&={\mathcal R}\left(\phi^X_{-\ell h}(x), e^{\ell
      h}\,T\right)\,.
\end{aligned}
\end{equation}
Uniform upper bounds on the functions~$c_{\mathcal D}(\cdot,
T,\ell)$, ${\mathcal C}(\cdot,T,\ell)$ and ${\mathcal R}(\cdot, T,\ell)$ are
clearly equivalent to uniform bounds on~$c_{\mathcal D}(\cdot,
e^{\ell h}\,T)$, ${\mathcal C}(\cdot,e^{\ell h}\,T)$ and~${\mathcal R}(\cdot,
e^{\ell h}\,T)$ respectively. Let
\begin{equation}
\label{eq:remainderdef}
\begin{aligned}
  {r}_{\mathcal D}(x,T,\ell)&:= c_{\mathcal D}\left(\phi^X_h{\mathcal R}
    (x,T,\ell)\right) \,\in\R\,\,,
  \\
  {\mathcal R}_{\mathcal C}(x,T,\ell)&:= {\mathcal C}\left(\phi^X_h{\mathcal R}
    (x,T,\ell)\right)\,\in  \,{\mathcal I}^s_{\mathcal C}\,\,.
\end{aligned} 
\end{equation}

By the identity~$\phi^X_{(\ell+1)h}= \phi^X_h\circ\phi^X_{\ell h}$,
since the distributions ${\mathcal D}\in{\mathcal B}^s\setminus{\mathcal B}_{1/4}$
are eigenvectors of the geodesic flow $\{\phi^X_t\}$ (see 
\eqref{eq:exponents}) and the space~${\mathcal I}^s_{\mathcal C}$ is
$\{\phi^X_t\}$-invariant, we obtain by projecting on 
${\mathcal D}$-components and on the continuous component:
\begin{equation}
\label{eq:diffeq}
\begin{aligned}
  c_{\mathcal D}(x,T,\ell+1) &= c_{\mathcal D}(x,T,\ell)\,e^{\lambda_{\mathcal D}
    h}\, + \, {r}_{\mathcal D}(x,T,\ell)\,;
  \\
  {\mathcal C}(x,T,\ell+1) &= \phi^X_h {\mathcal C}(x,T,\ell) \,+ \,
  {\mathcal R}_{\mathcal C}(x,T,\ell)\,.
\end{aligned}   
\end{equation}
If~$1/4 \in \sigma_{pp}$, for all pairs~$\{{\mathcal D}^+,{\mathcal D}^-\}\subset
{\mathcal B}_{1/4}$ we obtain by~\eqref{eq:JBtwo}:
\begin{equation}
\label{eq:Jdiffeq}
\begin{aligned}
  c_{{\mathcal D}^+}(x,T,\ell+1) &= [c_{{\mathcal D}^+}(x,T,\ell)-\frac{h}{2}\,
  c_{{\mathcal D}^-}(x,T,\ell)]\,e^{-h/2}\,+\,{r}_{{\mathcal D}^+}(x,T,\ell)\,;
  \\
  c_{{\mathcal D}^-}(x,T,\ell+1) &= c_{{\mathcal D}^-}(x,T,\ell)\,e^{-h/2}\,+\,
  {r}_{{\mathcal D}^-}(x,T,\ell)\,.
\end{aligned}
\end{equation}

Bounds on the solutions of the difference equations~\eqref{eq:diffeq}
and~\eqref{eq:Jdiffeq} can be derived from the following trivial lemma.

\begin{lemma} (see \cite{FlaFor}, Lemma 5.9)
  \label{lemma:opeq}
  Let $\Phi\in {\mathcal L}(E)$ be a bounded linear operator on a normed
  space~$E$.  Let~$\{ R_{\ell} \}$, $\ell\in \N$, be a sequence of
  elements of $E$. The solution~$\{x_{\ell}\}$ of the following
  difference equation in~$E$,
  \begin{equation} 
    \label{eq:opeq}
    x_{\ell+1} = \Phi (x_{\ell}) + R_{\ell}\,,  \quad \ell\in \N \,\,,
  \end{equation}
  has the form
  \begin{equation}
    \label{eq:opsol}
    x_{\ell} = \Phi^{\ell}(x_0) + \sum_{j=0}^{\ell-1} \Phi^{\ell-j-1}R_j\,.
   \end{equation}
\end{lemma}

By Lemma~\ref{lemma:opeq}, the proof of bounds on deviation of ergodic averagess 
is essentially reduced to estimates on the `remainder terms'~${r}_{\mathcal D}
(x,T,\ell)$ and ${\mathcal R}_{\mathcal C}(x,T,\ell)$. Such estimates can be derived
from Lemma~\ref{lemma:remainder}. For each $(x,T)\in S_\Gamma \times \R^+$ and each $\ell\in \N$, 
let $d_\Gamma(x,T,\ell)$ be the the maximum distance of the endpoints of the horocycle
arc~$\phi^X_{\ell h} (\gamma_{x,T})$ from the thick part:
\begin{equation}
\label{eq:distance}
d_\Gamma(x,T,\ell):= \max \{ d_\Gamma \Bigl(\phi^X_{-\ell h}(x)\Bigr), 
d_\Gamma \Bigl(\phi^X_{-\ell h}\circ\phi^U_T(x)\Bigr)\}\,\,.
\end{equation}

\begin{lemma} (see \cite{FlaFor}, Lemma 5.10)
\label{lemma:diffeq}
For every $s>3$ there exists a constant $C_6:=C_6(s)$ such that, for all $(x,T)\in
S_\Gamma\times \R^+$ and all $\ell\in \N$,
\begin{equation}
\label{eq:projremainder}
\sum_{{\mathcal D}\in {\mathcal B}^s} |{r}_{\mathcal D}(x,T,\ell)|^2 +
\|{\mathcal R}_{\mathcal C}(x,T,\ell)\|_{-s}^2\,\,\le\,\,
\left(\frac{C_6}{T}\right)^2  \max\{ \inj_\Gamma^{-2},  \exp \{d_\Gamma(x,T,\ell)-2\ell h \}\,.
\end{equation}  
For every $s>4$  there exists a constant $C'_6:=C'_6(s)$ such that, for all $(x,T)\in
S_\Gamma\times \R^+$ and all $\ell\in \N$, for every Casimir parameter $\mu:= \mu(\nu) \in (0,1/4)$, 
\begin{equation}
\label{eq:projremainder_bis}
\sum_{{\mathcal D}\in {\mathcal B}^s_\mu} |{r}_{\mathcal D}(x,T,\ell)|^2 \,\,\le\,\,
\left(\frac{C'_6}{T}\right)^2  \max\{ \inj_\Gamma^{-2},  \exp \{(1-\nu)d_\Gamma(x,T,\ell)-2\ell h \}\,.
\end{equation}  

\end{lemma}
 
\begin{proof}
  Let~$C_X(s):= \max_{h\in [1,2]}\|\phi^X_h\|_{-s}$. Then, by the
  definition \eqref{eq:remainderdef} and Lem\-ma~\ref
  {lemma:bounddist}, we have
  $$
  \sum_{{\mathcal D}\in {\mathcal B}^s} |{r}_{\mathcal D}(x,T,\ell)|^2 + \|{\mathcal
    R}_{\mathcal C}(x,T,\ell)\|_{-s}^2\,\,\le C^2_2 C_X^2 \|{\mathcal
    R}(x,T,\ell)\|^2_{-s}\,.
  $$
  But ${\mathcal R}(x,T,\ell)$ is the ``${\mathcal R}$'' component of an
  arc of horocycle of length $ e^{\ell h} T$ whose endpoints are at a
  distance at most $d_\Gamma(x,T,\ell)$ from the thick part (cf.
  \eqref{eq:gatwo},   \eqref{eq:pushforwsplit}, \eqref{eq:coeffell}, \eqref{eq:distance}). Thus by
  Lemma~\ref{lemma:remainder} we have $$ \|{\mathcal R}(x,T,\ell)\|_{-s} < C_5
  \max\{ \inj_\Gamma^{-1}, e^{d_\Gamma(x,T,\ell)/2}\} / e^{\ell h} T$$ and the lemma follows in this case.
  
  \smallskip
  \noindent Similarly,  by the definition \eqref{eq:remainderdef} and Lem\-ma~\ref
  {lemma:bounddist}, for all $\mu \in (0,1/4)$ we have
  $$
 \sum_{{\mathcal D}\in {\mathcal B}^s_\mu} |{r}_{\mathcal D}(x,T,\ell)|^2 \,\,\le C^2_2 C_X^2 \|{\mathcal
    R}(x,T,\ell)  \vert  W^{-s}(H_\mu) \|^2\,,
  $$
and by Lemma~\ref{lemma:remainder} we have $$ \|{\mathcal R}(x,T,\ell) W^{-s}(H_\mu) \| < C_5
  \max\{ \inj_\Gamma^{-1}, e^{\frac{1-\nu}{2} d_\Gamma(x,T,\ell)}\} / e^{\ell h} T$$
 hence the second statement follows and the lemma is completely proved.
\end{proof}

\subsection{Bounds on the components}
In the cuspidal case, the precision of our asymptotics of geodesic
push-forwards of a horocycle arc depends on the rate of escape into
the cusps of its endpoints. Let $d_\Gamma: S_\Gamma \to \R^+$ be the distance function
from thick part of $M_\Gamma$. For any given~$\sigma\in [0,1]$
and~$A\ge 0$ let
$$
V_{A,\sigma}:= \{ x\in S_\Gamma\,|\,  d_\Gamma\left( \phi^X_t(x)\right) \le 
A+\sigma\,|t|\,,\,\,\text{ for all }\,\, t\le 0\}\,.
$$
$$
V_{\sigma}:= \bigcup_{A\ge 0} V_{A,\sigma}
$$
The sets $V_{\sigma}$  are measurable as they can be written as countable unions 
of closed  sets (hence, they are $F_{\sigma}$ sets). Since the geodesic flow has unit 
speed $V_1=S_\Gamma$. On the other hand, by the logarithmic law of geodesics, 
$V_{\sigma}$ has full measure for any~$\sigma >0$. 

\begin{lemma} (\cite{FlaFor}, Lemma 5.12)
  \label{lemma:compbounds}
  For ~$s>3$ and for every~${\mathcal D}\in{\mathcal B}^s$ of order $ S_{\mathcal D}>0$,
  there exists a uniformly bounded sequence of positive bounded functions
  ~$\{K_{\mathcal D}(x,T, \ell)\}_{\ell\in \Z^+}$, $(x,T)\in V_{A,\sigma}
  \times \R^+$, such that the following estimates hold. For every
  horocycle arc~$\gamma_{x,T}$ having endpoints~$x$, $h_T(x)\in
  V_{A,\sigma}$ and for all~$\ell\in \Z^+$ we have, if~${\mathcal D}\in
  {\mathcal B}^s\setminus {\mathcal B}^+_{1/4}$,
  \begin{equation}
    \label{eq:Dcompbounds}
    |c_{\mathcal D}(x,T,\ell)|\le
    \begin{cases} K_{\mathcal D}(x,T,\ell)\, e^{-S_{\mathcal D}\ell h}\,,
    \quad &\text{ if } S_{\mathcal D}< 1-\frac{\sigma}{2}\,,
      \\
      K_{\mathcal D}(x,T,\ell)\, \ell\,e^{-S_{\mathcal D}\ell h}\,, \quad&\text{ if }
      S_{\mathcal D}=1-\frac{\sigma}{2}\,,
      \\
      K_{\mathcal D}(x,T,\ell)\, e^{-(1-\frac{\sigma}{2})\ell h}\,, \quad&\text{ if
        }S_{\mathcal D}> 1-\frac{\sigma}{2}\,.
     \end{cases}
  \end{equation}
  For $s\geq 4$ and for components of the complementary series the above estimates can be improved as follows.
   For every Casimir parameter $\mu:= \mu(\nu) \in (0,1/4)$, 
  \begin{equation}
    \label{eq:Dcompbounds_compl}
    |c_{\mathcal D^\pm_\mu}(x,T,\ell)|\le
    \begin{cases} K_{\mathcal D^\pm_\mu}(x,T,\ell)\, e^{- \frac{1\pm \nu}{2} \ell h}\,,
    \quad &\text{ if } \sigma <  \frac{1\mp \nu}{1-\nu}\,,
      \\
      K_{\mathcal D^\pm_\mu}(x,T,\ell)\, \ell\,e^{-\frac{1\pm \nu}{2}\ell h}\,, \quad&\text{ if }
      \sigma =  \frac{1\mp \nu}{1-\nu}\,,
      \\
      K_{\mathcal D^\pm_\mu}(x,T,\ell)\, e^{-(1- (\frac{1-\nu}{2}) \sigma)\ell h}\,, \quad&\text{ if
        }\sigma > \frac{1\mp \nu}{1-\nu}\,.
     \end{cases}
  \end{equation}

   If~$1/4\in\sigma_{pp}(\square)$ and~${\mathcal D}\in {\mathcal B}^+_{1/4}$, we have
  \begin{equation}
    \label{eq:JDcompbounds}
    |c_{\mathcal D}(x,T,\ell)|\le
    \begin{cases} K_{\mathcal D}(x,T,\ell)\, \ell\,e^{-\ell h/2}\,,
    \quad &\text{ if } \sigma<1\,,
      \\
      K_{\mathcal D}(x,T,\ell)\, \ell^2\,e^{-\ell h/2}\,, \quad&\text{ if }
      \sigma=1\,.
      \\
    \end{cases}
  \end{equation}

  There exists a uniformly bounded sequence of positive bounded functions
  $\{ K_{\mathcal C}(x,T, \ell)\}_{\ell \in \Z}$
  such that the following estimates hold. For every horocycle 
  arc~$\gamma_{x,T}$ as above and for all~$\ell\in\Z^+$, we have
  \begin{equation} 
    \label{eq:contcompbounds}
    \|{\mathcal C}(x,T,\ell)\|_{-s}\le
    \begin{cases} K_{\mathcal C}(x,T, \ell) \, \ell\,e^{-\ell h/2}\,,
      \quad&\text{ if } \sigma<1\,,
      \\
      K_{\mathcal C}(x,T, \ell) \, \ell^2\,e^{-\ell h/2}\,, \quad&\text{ if }
      \sigma=1\,.
    \end{cases}
  \end{equation}
  In addition, there exists a positive constant $K:=K (\sigma,T,s)$
  such that, for all $\gamma_{x,T}$ with endpoints belonging to 
  the set $V_{A,\sigma}$ and for all $\ell\in \Z^+$,
  
   \begin{equation}
    \label{eq:coeffbound}
    \sum_{{\mathcal D}\in {\mathcal B}^s} K^2_{\mathcal D}(x,T,\ell)\,\,+\,\,
    K^2_{\mathcal C}(x,T, \ell) \,\, \le \,\,K^2 \max\{ \inj_\Gamma^{-2}, \max_{y\in \gamma_{x,T}}  e^{A d_\Gamma(y)} \} \} \,\,.
  \end{equation}

\end{lemma}

\begin{proof}
  For all ${\mathcal D}\not\in {\mathcal B}_{1/4}$, by the first difference
  equation in formula~\eqref{eq:diffeq} and by Lemma~\ref{lemma:opeq} with 
  $E:=\C$ and $\Phi$ the multiplication operator by $e^{\lambda_{\mathcal
  D}h}\in \C$, we obtain
  \begin{equation}
    \label{eq:Dest}
    |c_{\mathcal D}(x,T,\ell)| \le |c_{\mathcal D}(x,T,0)|\,
    e^{-S_{\mathcal D}\ell h} \,+\, \Sigma_{\mathcal D}(x,T,\ell)\,,
  \end{equation}
  with
  $$
  \Sigma_{\mathcal D}(x,T,\ell):= \sum_{j=0}^{\ell-1} |{r}_{\mathcal
    D}(x,T,j)| e^{-S_{\mathcal D} h (\ell-j-1)}\,.
  $$
  We must therefore bound the terms~$|c_{\mathcal D}(x,T,0)|$ and~$e^{S_{\mathcal D}
  \ell h}\Sigma_{\mathcal D}(x,T,\ell)$ by~$K_{\mathcal D}(x,T,\ell)$,
  $K_{\mathcal D}(x,T,\ell)\ell$ or~$K_{\mathcal D}(x,T,\ell)e^{(S_{\mathcal D}-1+
  \frac{\sigma}{2})\ell h}$, according to the different values of~$S_{\mathcal D}$,
  for some uniformly bounded sequence of functions~$\{K_{\mathcal D}(x,T,\ell)\}
  _{\ell\in\Z^+}$ on $V_{A,\sigma}\times \R^+$.

  It follows from Corollary~\ref{cor:unifupperb}, taking into account the 
  fact that the endpoints of $\gamma_{x,T}$ belong to the set~$V_{A,\sigma}$, that for all $s>3$ there exists a constant 
    $C_s>0$ such that
  \begin{equation}
    \label{eq:czero}
     \sum_{\mathcal D\in{\mathcal B}^s} |c_{\mathcal D}(x,T,0)|^2 \le C_s \max \{ \inj_\Gamma^{-2}, e^{A}  \max_{y\in \gamma_{x,T}} e^{d_\Gamma (y)} \}      \,\,;
  \end{equation}
  for components of the complementary series, that is, for all $s>4$ there exists a constant 
    $C'_s>0$ such that for all Casimir parameters $\mu:=\mu(\nu) \in (0,1/4)$, 
   \begin{equation}
    \label{eq:czero_compl}
     \sum_{\mathcal D\in{\mathcal B}^s_\mu} |c_{\mathcal D}(x,T,0)|^2 \le C'_s \max \{ \inj_\Gamma^{-2}, e^{(1-\nu)A}  \max_{y\in \gamma_{x,T}}
     e^{(1-\nu)d_\Gamma (y)} \}      \,\,;
  \end{equation}
  thus the term $|c_{\mathcal D}(x,T,0)|$ in formula~\eqref{eq:Dest} satisfies
  estimates finer than~\eqref{eq:Dcompbounds} and~\eqref{eq:coeffbound}.
  
  Using again the fact that the endpoints of
  $\gamma_{x,T}$ belong to the set $V_{A,\sigma}$ and the
  estimate~\eqref{eq:projremainder}, a calculation based on the
  Cauchy-Schwartz inequality yields the following bounds on the
  remainder terms $\Sigma_{\mathcal D}(x,T,\ell)$ for ${\mathcal D}\not
  \in {\mathcal B}_{1/4}$. For all $S>0$,  $\sigma\in [0,1]$, and for all $s>3$,
  there exists a constant~$C':=C'(\sigma,S,s)>0$ such that, for all $A\geq 0$, 
  \begin{equation}
    \label{eq:Sigmaboundone}
    \begin{aligned} 
      \sum_{{\mathcal D}:S_{\mathcal D}\le S} \Sigma^2_{\mathcal D}(x,T,\ell) \,
      e^{2S_{\mathcal D}\ell h} & \le \quad\frac{C'}{T^2}  \max \{ \inj_\Gamma^{-2}, e^{A} \}   \,, \quad &\text{ if } S
      < 1-\frac{\sigma}{2}\,;
      \\
      \sum_{{\mathcal D}:S_{\mathcal D}=S} \Sigma^2_{\mathcal D}(x,T,\ell) \,
      e^{2S_{\mathcal D}\ell h} & \le \quad\frac{C' \,
        \ell^2}{T^2} \max \{ \inj_\Gamma^{-2}, e^{A} \} \,, \quad &\text{ if }S = 1-\frac{\sigma}{2}\,;
      \\
      \sum_{{\mathcal D}:S_{\mathcal D}\ge S} \Sigma^2_{\mathcal D}(x,T,\ell)
        \quad\quad\quad   &\le \quad\frac{C'}{T^2}\max \{ \inj_\Gamma^{-2} e^{A} \} 
        e^{ -2(1-\frac{\sigma}{2})\ell h}
         \quad &\text{ if } S> 1-\frac{\sigma}{2}\,.
    \end{aligned} 
  \end{equation} 
  It follows from~\eqref{eq:czero} and~\eqref{eq:Sigmaboundone} that, for
  each~${\mathcal D}\in {\mathcal B}^s$, the sequence of positive functions
  \begin{equation}
    \label{eq:Kappasigma1}
    K_{\mathcal D}(x,T,\ell):=
    \begin{cases}
      |c_{\mathcal D}(x,T,0)|+\Sigma_{\mathcal D}(x,T,\ell)
      e^{S_{\mathcal D}\ell h}\,, &\text{if } S_{\mathcal D}<
      1-\frac{\sigma}{2}\,,
      \\
      \Big( |c_{\mathcal D}(x,T,0)|+\Sigma_{\mathcal D}(x,T,\ell)
      e^{S_{\mathcal D}\ell h}\Big) \ell^{-1}\,, &\text{if } S_{\mathcal
        D}= 1-\frac{\sigma}{2}\,,
      \\
      \Big(|c_{\mathcal D}(x,T,0)|+\Sigma_{\mathcal D}(x,T,\ell)
      e^{S_{\mathcal D}\ell h}\Big) e^{-(S_{\mathcal
          D}-1+\frac{\sigma}{2})\ell h}\,, &\text{if } S_{\mathcal
        D}> 1-\frac{\sigma}{2}\,.
    \end{cases}
  \end{equation}
  is uniformly bounded for all $\gamma_{x,T}$ with endpoints belonging to 
  the set $V_{A,\sigma}$. In view of the bound~\eqref{eq:Dest}, this proves the 
  estimate~\eqref{eq:Dcompbounds} for each~${\mathcal D}\not\in {\mathcal B}_{1/4}$.
  In addition, since the set of real numbers~$\{S_{\mathcal D}\,|\,{\mathcal D}
  \in {\mathcal B}^s\}$ is finite, the inequalities~\eqref{eq:czero}
  and~\eqref{eq:Sigmaboundone} imply that, for all $\gamma_{x,T}$ with 
  endpoints belonging to the set $V_{A,\sigma}$ and for all $\ell\in \Z^+$,
  \begin{equation}
  \label{eq:cboundone}
  \sum_{{\mathcal D}\in {\mathcal B}^s\setminus{\mathcal B}_{1/4}} K^2_{\mathcal D}(x,T,\ell)
   \,\le\, C'' \max \{ \inj_\Gamma^{-2}, e^{A} \max_{y\in \gamma_{x,T}} e^{d_\Gamma (y)}  \} \,\,,
  \end{equation}
  for some constant $C'':=C''(\sigma,T)>0$, thereby proving the upper
  bound~\eqref{eq:coeffbound} over all ${\mathcal D}$-components with ${\mathcal D}
  \in {\mathcal B}^s \setminus {\mathcal B}_{1/4}$.
  
  \smallskip
  For components of the complementary series  the above estimates can be improved as follows. Since the endpoints of
  $\gamma_{x,T}$ belong to the set $V_{A,\sigma}$, by the estimate~\eqref{eq:projremainder_bis}, 
  for all $S>0$,  $\sigma\in [0,1]$,  and for all $s>4$, there exists a constant~$C^{(3)}:=C^{(3)}(\sigma,S,s)>0$ such that, 
  for all Casimir parameters $\mu:= \mu(\nu) \in (0, 1)$ and for all $A\geq 0$, 
  \begin{equation}
    \label{eq:Sigmaboundone_compl}
    \begin{aligned} 
      \Sigma_{\mathcal D^\pm_\mu }(x,T,\ell) \,
      e^{\frac{1\pm \nu}{2}\ell h} & \le \quad\frac{C^{(3)}}{T}  \max \{ \inj_\Gamma^{-1}, e^{\frac{1-\nu}{2} A} \}   \,, \quad &\text{ if } 
      \sigma < \frac{1 \mp \nu}{1-\nu} \,;
      \\
       \Sigma_{\mathcal D^\pm_\mu}(x,T,\ell) \,
      e^{ \frac{1\pm \nu}{2}  \ell h} & \le \quad\frac{C^{(3)} }\,
       \frac{\ell}{T} \max \{ \inj_\Gamma^{-1}, e^{\frac{1-\nu}{2}A} \} \,, \quad &\text{ if } \sigma = \frac{1 \mp \nu}{1-\nu}\,;
      \\
    \Sigma_{\mathcal D^\pm_\mu}(x,T,\ell)
        \quad\quad\quad   &\le \quad\frac{C^{(3)}}{T}\max \{ \inj_\Gamma^{-1} e^{\frac{1-\nu}{2}A} \} 
        e^{ -(1-\frac{\sigma}{2})\ell h}
         \quad &\text{ if }  \sigma > \frac{1 \mp \nu}{1-\nu}\,.
    \end{aligned} 
  \end{equation} 
  It follows from~\eqref{eq:czero_compl} and~\eqref{eq:Sigmaboundone_compl} that, for
  each~${\mathcal D}\in {\mathcal B}^s$, the sequence of positive functions
  \begin{equation}
    \label{eq:Kappasigma1_compl}
    K_{\mathcal D^\pm_\mu}(x,T,\ell):=
    \begin{cases}
      |c_{\mathcal D^\pm_\mu}(x,T,0)|+\Sigma_{\mathcal D^\pm_\mu}(x,T,\ell)
      e^{\frac{1\pm \nu}{2} \ell h}\,, &\text{if }  \sigma < \frac{1 \mp \nu}{1-\nu}\,,
      \\
      \Big( |c_{\mathcal D^\pm_\mu}(x,T,0)|+\Sigma_{\mathcal D^\pm_\mu}(x,T,\ell)
      e^{\frac{1\pm \nu}{2}\ell h}\Big) \ell^{-1}\,, &\text{if }  \sigma = \frac{1 \mp \nu}{1-\nu}\,,
      \\
      \Big(|c_{\mathcal D^\pm_\mu}(x,T,0)|+\Sigma_{\mathcal D^\pm_\mu}(x,T,\ell)
      e^{\frac{1\pm \nu}{2} \ell h}\Big) e^{(\frac{1\mp \nu}{2} -(\frac{1-\nu}{2})\sigma)\ell h}\,, &\text{if }  \sigma > \frac{1 \mp \nu}{1-\nu}\,.
    \end{cases}
  \end{equation}
is uniformly bounded for all $\gamma_{x,T}$ with endpoints belonging to 
  the set $V_{A,\sigma}$. In view of the bound~\eqref{eq:Dest}, this proves the 
  estimate~\eqref{eq:Dcompbounds} for each~${\mathcal D}\not\in {\mathcal B}_{1/4}$.
  In addition, since the set of real numbers~$\{S_{\mathcal D}\,|\,{\mathcal D}
  \in {\mathcal B}^s\}$ is finite, the inequalities~\eqref{eq:czero_compl}
  and~\eqref{eq:Sigmaboundone_compl} imply that, for all $\gamma_{x,T}$ with 
  endpoints belonging to the set $V_{A,\sigma}$ and for all $\ell\in \Z^+$,
  \begin{equation}
  \label{eq:cboundone_compl}
  \sum_{{\mathcal D}\in {\mathcal B}^s_\mu} K^2_{\mathcal D}(x,T,\ell)
   \,\le\, C^{(4)} \max \{ \inj_\Gamma^{-2}, e^{(1-\nu)A} \max_{y\in \gamma_{x,T}} e^{(1-\nu)d_\Gamma (y)}  \} \,\,,
  \end{equation}
  for some constant $C^{(4)} :=C^{(4)} (\sigma,s)>0$, thereby proving the refined upper
  bound~\eqref{eq:cboundone_compl} over all ${\mathcal D}$-components with
  $\mathcal D \in   {\mathcal B}^s_\mu$ for Casimir parameters $\mu \in (0,1/4)$.

  The proofs of the upper bounds for all pairs $\{{\mathcal D}^+, {\mathcal D}^-\}
  \subset {\mathcal B}_{1/4}$ (if $1/4\in \sigma_{pp}(\square)$) and for the continuous
  component are similar. In the first case, by formula~\eqref{eq:Jdiffeq} we 
  can apply Lemma~\ref{lemma:opeq} with~$E=\R^2$ and 
  \begin{equation}
  \label{eq:Jmatrix}
   \Phi:= e^{-h/2}\,\begin{pmatrix}  1 & -h/2 \\ 0 & 1 \end{pmatrix}\,\,.
  \end{equation}
  By formula~\eqref{eq:opsol}, we obtain
 \begin{equation}
  \label{eq:JDest}
 \begin{aligned}
 |c_{{\mathcal D}^+}(x,T,\ell)|&\le |c_{{\mathcal D}^+}(x,T,0)-\frac{\ell h}{2}
 \,c_{{\mathcal D}^-}(x,T,0)|\,e^{-\ell h/2}+\Sigma_{{\mathcal D}^+}(x,T,\ell)
  \,\,,
   \\
  |c_{{\mathcal D}^-}(x,T,\ell)|&\le |c_{{\mathcal D}^-}(x,T,0)|\,e^{-\ell h/2}+
   \Sigma_{{\mathcal D}^-}(x,T,\ell)\,\,,
\end{aligned}
\end{equation}
with
\begin{equation}
\begin{aligned}
\label{eq:JSigmabound}
\Sigma_{{\mathcal D}^+}(x,T,\ell) & := \sum_{j=0}^{\ell-1} |{r}_{{\mathcal D^+}}
(x,T,j)-\frac{(\ell-j-1) h}{2}\,{r}_{{\mathcal D}^-}(x,T,j)|\,
e^{-h(\ell-j-1)/2}\,,     
\\
\Sigma_{{\mathcal D}^-}(x,T,\ell) & := \sum_{j=0}^{\ell-1} |{r}_{{\mathcal D}^-}
(x,T,j)| \,e^{-h(\ell-j-1)/2}\,.
\end{aligned}
\end{equation}

Since the endpoints of the horocycle arc $\gamma_{x,T}$ belong to the set 
$V_{A,\sigma}$, by the estimates~\eqref{eq:czero} and~\eqref{eq:projremainder},
there exists a constant $K_{\mathcal D^+}:= K_{\mathcal D}^+
  (\sigma,s)>0$ such that the sequence of positive functions

\begin{equation}
    \label{eq:JKappasigma}
    K_{{\mathcal D}^+}(x,T,\ell):=
    \begin{cases}
    \left(|c_{{\mathcal D}^+}-\frac{\ell h}{2}\,c_{{\mathcal D}^-}|
     (x,T,0)\,+\,\Sigma_{{\mathcal D}^+}(x,T,\ell)\,e^{\ell h/2}\right)\,\ell^{-1},
    \quad&\text{ if } \sigma<1\,,
      \\
     \left(|c_{{\mathcal D}^+}-\frac{\ell h}{2}\,c_{{\mathcal D}^-}|(x,T,0)\,
    +\,\Sigma_{{\mathcal D}^+}(x,T,\ell)\,e^{\ell h/2}\right)\,\ell^{-2},
    \quad&\text{ if } \sigma=1\,.
    \end{cases}
  \end{equation}
  is uniformly bounded  as follows: 
  $$
  K_{{\mathcal D}^+}(x,T,\ell) \leq K_{\mathcal D^+} \max\{ \inj_\Gamma^{-1}, e^{A/2}  \max_{y\in \gamma_{x,T}} e^{d_\Gamma (y)/2}  \}\,.
  $$
  If~${\mathcal D}={\mathcal D}^-\in {\mathcal B}^-_{1/4}$, it follows from the second
  lines in~\eqref{eq:JDest} and~\eqref{eq:JSigmabound} that  there exists a
  constant $K_{\mathcal D}:=K_{\mathcal D}(\sigma,s) >0$ such that  the sequence of
  positive functions~$K_{\mathcal D}(x,T,\ell)$, defined as
  in~\eqref{eq:Kappasigma1} with~$S_{\mathcal D}=1/2 \le 1-\frac{\sigma}{2}$, is
  uniformly bounded  as folllows
  $$
  K_{{\mathcal D}^+}(x,T,\ell)  \leq  K_{\mathcal D} \max\{ \inj_\Gamma^{-1}, e^{A/2} \max_{y\in \gamma_{x,T}} e^{d_\Gamma (y)/2} \} \,.
  $$
  
Therefore, by the estimate~\eqref
  {eq:cboundone} there exists a constant $C^{(5)}:=C^{(5)}(s)>0$ such 
  that, for all $\gamma_{x,T}$ with endpoints belonging to 
  the set $V_{A,\sigma}$ and for all $\ell\in\Z^+$,
\begin{equation}
  \label{eq:cboundtwo}
  \sum_{{\mathcal D}\in {\mathcal B}^s} K^2_{\mathcal D}(x,T,\ell) \,\le\, C^{(5)} \max\{ \inj_\Gamma^{-2}, e^{A}  \max_{y\in \gamma_{x,T}} e^{d_\Gamma (y)} \}\,\,.
  \end{equation}

For the continuous component, we apply Lemma~\ref{lemma:opeq}, with 
$E= {\mathcal I}^s_{\mathcal C}$ and $\Phi=\phi^X_h$, to the second difference
equation in formula~\eqref{eq:diffeq}. We obtain 
  \begin{equation}
  \label{eq:Cest}
  \|{\mathcal C}(x,T,\ell)\|_{-s} \le \|\Phi^{\ell}\|_{-s}
  \|{\mathcal C}(x,T,0)\|_{-s} \,+\,\Sigma_{\mathcal C}(x,T,\ell)
  \end{equation}
  with
  $$
  \Sigma_{\mathcal C}(x,T,\ell):=\sum_{j=0}^{\ell-1} \|\Phi^{\ell-j-1}
  {\mathcal R}_{\mathcal C}(x,T,j)\|_{-s}\,.
  $$
  By Lemma~\ref{lemma:normbound} the norm of the operator
  $\phi^X_t$ on ${\mathcal I}^s_{\mathcal C}$ is bounded by $C_1(1+|t|)
  e^{-t/2}$.  Taking into account the fact that the endpoints of
  $\gamma_{x,T}$ belong to the set $V_{A,\sigma}$ we find: (1) by
  Lemmata~\ref{lemma:bounddist} and~\ref{lemma:Sobembed_1} we obtain
  that there exists a constant $C^{(6)}:= C^{(6)}(s)$ such that 
  $$
   \|{\mathcal C}(x,T,0)\|_{-s} \le C^{(6)} \max \{\inj_\Gamma^{-1}, \max_{y\in \gamma_{x,T}} e^{d_\Gamma(y) } \}  \,;
  $$ (2) using the
  estimate~\eqref{eq:projremainder} for $\| {\mathcal R}_{\mathcal
    C}(x,T,j)\|_{-s}$ we find that the sequence of positive functions 
   defined by
  \begin{equation}
    \label{eq:Kappasigma2}
    K_{\mathcal C}(x,T,\ell):=
    \begin{cases}
      \Big( \|{\mathcal C}(x,T,0)\|_{-s}\,+\,\Sigma_{\mathcal C}(x,T,\ell)\,
      e^{\ell h/2}\Big) \ell^{-1} \,,\quad&\text{ if } \sigma<1\,,
      \\
      \Big( \|{\mathcal C}(x,T,0)\|_{-s}\,+\,\Sigma_{\mathcal C} (x,T,\ell)\,e^{\ell
        h/2}\Big)\ell^{-2}\,, \quad&\text{ if } \sigma= 1\,.
    \end{cases}
  \end{equation}
  is uniformly bounded:  there exists a constant  $K_{\mathcal C}:= K_{\mathcal C}(\sigma,s)>0$
  such that
  $$
   K_{\mathcal C}(x,T,\ell) \leq K_{\mathcal C}  \max \{\inj_\Gamma^{-1}, \max_{y\in \gamma_{x,T}} e^{d_\Gamma(y) } \} \,.
  $$
  \end{proof}

For all $t\ge 0$, the push-forward probability measure~$\phi_t^X(\gamma_{x,T})$
is the uniformly distributed probability measure on a stable horocycle arc of 
length $T_t:=e^t\,T$. The following quantitative equidistribution result holds.
Let ${\mathcal I}^s_+(S_\Gamma)\subset {\mathcal I}^s(S_\Gamma)$ be the subspace of invariant
distributions orthogonal to the volume form.
  
\begin{theorem}  (\cite{FlaFor}, Theorem 5.14)
  \label{thm:sigmanoncpt}
  Let~$s>3$. Then there exists a constant $C^{(7)}:=C^{(7)}(\sigma,s)$ such
  that for any horocycle arc $\gamma_{x,T}$ with endpoints belonging
  to the set $V_{A,\sigma}$, for any $t\ge 1$ and for all~$f\in
  W^s(S_\Gamma)$, we have
  \begin{multline}
    \label{eq:sigmapushfwd}
    \phi_t^X(\gamma_{x,T})(f) = \int_{S_\Gamma} f\,d\text{vol} \,+\,\sum_{
      {\mathcal D}\in {\mathcal B}_+^{1-\frac{\sigma}{2}} } c^s_{\mathcal
      D}(x,T,t) \,{\mathcal D}(f)\,T_t^{-S_{\mathcal D}}\,+
    \\
    +\,{\mathcal C}^s(x,T,t)(f) T_t^{-\frac{1}{2}} \log^{\alpha_{\sigma}}
    T_t\,+\, {\mathcal R}^s(x,T,t)(f)\,T_t^{\frac{\sigma}{2}-1}
    \log^{\beta_{\sigma}} T_t\,\,.
  \end{multline}
  with~$c^s_{\mathcal D}(x,T,t)\in \C$,  ${\mathcal C}^s(x,T,t)\in {\mathcal
    I}_{\mathcal C}^s$ and~${\mathcal R}^s (x,T,t)\in W^{-s}(S_\Gamma)$
  satisfying the following upper bounds:
  \begin{align*} 
    \sum_{{\mathcal D}\in{\mathcal B}_+^{1-\frac{\sigma}{2}}} |c^s_{\mathcal
      D}(x,T,t)|^2 &\le C_7 \max\{ \inj_\Gamma^{-1}, e^{A/2} \max_{y\in \gamma_{x,T}} e^{ d_\Gamma(y)/2}\}\,,
    \\
    \|{\mathcal C}^s(x,T,t)\|_{-s} & \le  C_7 \max\{ \inj_\Gamma^{-1}, e^{A/2} \max_{y\in \gamma_{x,T}} e^{ d_\Gamma(y)/2}\}\,,
    \\
    \|{\mathcal R}^s(x,T,t)\|_{-s} & \le  C_7\max\{ \inj_\Gamma^{-1}, e^{A/2} \max_{y\in \gamma_{x,T}} e^{ d_\Gamma(y)/2}\}\,.
  \end{align*}
  In the above asymptotics, the exponent~$\alpha_{\sigma}$ is~$1$ if 
  $\sigma<1$ and equals~$2$ if $\sigma=1$; the exponent~$\beta_{\sigma}$ 
  is~$0$ if every ${\mathcal D}\in\mathcal B^s$ has Sobolev order~$S_{\mathcal D}\not=
  1-\frac{\sigma}{2}$ and equals~$1$ otherwise.
   
  In addition, for $s>4$ the estimate on the irreducible components of the complementary series (corresponding to
  Casimir parameters $\mu= (1-\nu^2)/4  \in (0,1/4)$ (that is $\nu (0,1)$) can be refined as follows:
  $$
  \sum_{{\mathcal D^\pm_\mu }\in{\mathcal B}_\mu} |c^s_{\mathcal
      D}(x,T,t)|^2 \le C_7 \max\{ \inj_\Gamma^{-1},  e^{\frac{1-\nu}{2}} \max_{y\in \gamma_{x,T}} e^{ \frac{1-\nu}{2} d_\Gamma(y)/2}\}\,.
  $$
  
  \end{theorem}
  
\begin{proof}
  Let~$t\ge 1$. There exist $h\in [1,2]$ and $\ell \in \Z^+$ such that
  $t=\ell h$.  The distribution~$\phi^X_t(\gamma_{x,T})\in W^{-s}(S_\Gamma)$
  can be split as in~\eqref{eq:pushforwsplit}, hence the expansion~\eqref
  {eq:sigmapushfwd} follows. The pointwise upper bounds on the coefficients
  can be derived from Lemma~\ref{lemma:compbounds} for the ${\mathcal D}$-components
  and the ${\mathcal C}$-component and, by its definition in~\eqref{eq:coeffell},
  from Lemma~\ref{lemma:remainder} for the remainder term ${\mathcal R}(x,T,\ell)$
  of the splitting~\eqref{eq:pushforwsplit}. We remark that the term with
  coefficient ${\mathcal R}^s(x,T,t)(f)$ in~\eqref{eq:sigmapushfwd} includes the
  contributions of all ${\mathcal D}$-components with~${\mathcal D}\not\in {\mathcal
  B}^{1-\frac{\sigma}{2}}$ as well as the contribution of the remainder term 
  ${\mathcal R}(x,T,\ell)$ of the splitting~\eqref{eq:pushforwsplit}.
  All such estimates are uniform with respect  to $h\in [1,2]$.
\end{proof}

We conclude with the proof of Theorem~\ref{thm:SL2_equi}. 

\begin{proof} [Proof of Theorem~\ref{thm:SL2_equi}]

Let $T>1$ and let $\gamma_{x,T}$ be an orbit segment of the (stable) horocycle flow. Let $x_T := a_{\log T} (x) = \phi^X_{\log T} (x)$ and let $\gamma_{x_T}:=\gamma_{x_T,1}$ denote the stable horocycle orbit segment of unit length. Clearly we have (in distributional sense)
$$
\gamma_{x,T} = \phi_{\log T}^X ( \gamma_{x_T}) = a_{\log T}( \gamma_{x_T})\,.
$$
Let $A:= A_{x,T} = d_\Gamma (x_T) +1$. Clearly by construction $x_T \in V_{A, 1}$ and 
$$
\max_{y\in \gamma_{x_T}} d_\Gamma (y) \leq    d_\Gamma (x_T) + 1 \,.
$$
The result then follows from Theorem~\ref{thm:sigmanoncpt}   applied the horocycle orbit segment $\gamma_{x_T}$
with  $\sigma=1$ and $t = -\log T$.
\end{proof}

\end{document}